\documentclass[10pt]{article}
\usepackage[utf8]{inputenc}
\usepackage[T1]{fontenc}
\usepackage{amsfonts}
\usepackage{amsmath,latexsym}
\usepackage{amssymb,amsthm}
\usepackage{graphicx}
\usepackage[english]{babel}
\usepackage{hyperref}
\usepackage{color}

\usepackage{xspace}

\usepackage{mathrsfs}

\usepackage{tikz}
\usetikzlibrary{decorations.pathreplacing}
\usepackage{caption}

\usepackage{cases}

\usepackage{geometry}
\usepackage[french]{layout}
\providecommand{\keywords}[1]{\textbf{Keywords:} #1}

\newtheorem{theorem}{Theorem}[section]

\newtheorem{lemma}[theorem]{Lemma}
\newtheorem{proposition}[theorem]{Proposition}
\newtheorem{example}[theorem]{Example}

\theoremstyle{definition}
\newtheorem{definition}[theorem]{Definition}
\newtheorem{remark}[theorem]{Remark}

\newtheoremstyle{question}  
{5pt}                        
{10pt}                        
{}                           
{}                           
{\bf \color{red}}               
{:}                          
{.5em}                       
{}                           
 
\theoremstyle{question}

\newcommand\No{n\textsuperscript{o}}

\def\ds{\displaystyle}

\newcommand{\tr}[1]{{#1}^{*}}

\newcommand{\field}[1]{\ensuremath{\mathbb{#1}}}
\newcommand{\C}{\field{C}\xspace}
\newcommand{\R}{\field{R}\xspace}

\newcommand{\N}{\field{N}\xspace}

\newcommand{\ens}[1]{ \left\{#1\right\} }
\newcommand{\lin}[1]{\mathcal{L}(#1)}

\newcommand{\ps}[3]{ {\left\langle #1 , #2 \right\rangle}_{#3} }
\newcommand\abs[1]{\left|#1\right|}
\def\norm#1{\left\|#1\right\|}

\newcommand\diag{\mathrm{diag} \xspace}

\newcommand{\Topt}[1]{T_{\mathrm{inf}}\left(\Lambda,#1,Q\right)}
\newcommand{\Toptlib}[1]{T_{\mathrm{inf}}\left(#1\right)}

\newcommand\Id{\mathrm{Id}}
\newcommand\caract{\chi}
\newcommand\caractt{\tilde{\chi}}

\renewcommand\epsilon{\varepsilon}

\newcommand\pt[1]{\frac{\partial #1}{\partial t}}
\newcommand\px[1]{\frac{\partial #1}{\partial x}}
\newcommand\pxi[1]{\frac{\partial #1}{\partial \xi}}
\newcommand\partials[1]{\frac{\partial #1}{\partial s}}

\newcommand{\dom}[1]{ D(#1) }
\newcommand\ddt{\frac{d}{dt}}
\newcommand\dds{\frac{d}{ds}}
\newcommand\clos[1]{\overline{#1}}

\newcommand{\rank}{ \mathrm{rank} \,}
\newcommand\range{\mathrm{Im} \,}

\newcommand\ssin{s^{\mathrm{in}}}
\newcommand\ssout{s^{\mathrm{out}}}
\newcommand\ssint{\tilde{s}^{\mathrm{in}}}
\newcommand\ssoutt{\tilde{s}^{\mathrm{out}}}

\newcommand\aew{\text{ a.e. }}

\usepackage{dsfont}
\newcommand\indic{\mathds{1}}

\usepackage[shortlabels]{enumitem}

\newcommand\TCN{T_{\mathrm{opt}}}

\title{Minimal time for the exact controllability of one-dimensional first-order linear hyperbolic systems by one-sided boundary controls}

\begin{document}

\author{
Long Hu\thanks{School of Mathematics, Shandong University, Jinan, Shandong 250100, China.  E-mail: \texttt{hul@sdu.edu.cn}.},  
Guillaume Olive\thanks{Institute of Mathematics, Jagiellonian University, Lojasiewicza 6,
30-348 Krakow, Poland. E-mail: \texttt{math.golive@gmail.com} or \texttt{guillaume.olive@uj.edu.pl}}.
}

\maketitle

\begin{abstract}
In this article we study the minimal time for the exact controllability of one-dimensional first-order linear hyperbolic systems when all the controls are acting on the same side of the boundary.
We establish an explicit and easy-to-compute formula for this time with respect to all the coupling parameters of the system.
The proof relies on the introduction of a canonical $UL$--decomposition and the compactness-uniqueness method.
\end{abstract}



\section{Introduction and main result}

\indent

In this article we are interested in the controllability properties of the following class of one-dimensional first-order linear hyperbolic systems, which appears for instance in linearized Saint-Venant equations and many other physical models of balance laws (see e.g. \cite[Chapter 1]{BC16}):
\begin{equation}\label{syst}
\left\{\begin{array}{l}
\ds \pt{y}(t,x)=\Lambda(x) \px{y}(t,x)+M(x) y(t,x), \\
y_+(t,0)=Qy_-(t,0), \quad y_-(t,1)=u(t),  \\
y(0,x)=y^0(x),
\end{array}\right.
\quad t \in (0,+\infty), x \in (0,1).
\end{equation}

In \eqref{syst}, $y(t,\cdot)$ is the state at time $t$, $y^0$ is the initial data and $u(t)$ is the control at time $t$.
We denote by $n\geq 2$ the total number of equations of the system.
The matrix $\Lambda \in C^{0,1}([0,1])^{n \times n}$ is assumed to be diagonal:
\begin{equation}\label{Lambda diag}
\Lambda =\diag(\lambda_1,\ldots,\lambda_n),
\end{equation}
with $p \geq 1$ negative eigenvalues and $m \geq 1$ positive eigenvalues (so that $p+m=n$)  such that:
\begin{equation}\label{hyp speeds}
\lambda_1(x) \leq \cdots \leq \lambda_p(x) <0<\lambda_{p+1}(x) \leq \cdots \leq \lambda_{p+m}(x), \quad \forall x \in [0,1],
\end{equation}
and we assume that, in case two eigenvalues agree somewhere, they agree everywhere:
\begin{equation}\label{hyp Rus78 eg}
\forall i,j \in \ens{1,\ldots,n}, \quad i \neq j, \quad \left(\exists x \in [0,1], \quad \lambda_i(x)=\lambda_j(x)\right) \Longrightarrow \left(\lambda_i(x)=\lambda_j(x), \quad \forall x \in [0,1]\right).
\end{equation}
The assumption \eqref{hyp Rus78 eg} will be commented below.
All along this paper, for a vector (or vector-valued function) $y \in \R^n$ we use the notation
$$y=\begin{pmatrix} y_+ \\ y_- \end{pmatrix},$$
where $y_+ \in \R^p$ and $y_- \in \R^m$.
Finally, the matrix $M \in L^{\infty}(0,1)^{n \times n}$ couples the equations of the system inside the domain and the constant matrix $Q \in \R^{p \times m}$ couples the equations of the system on the boundary $x=0$.

Taking formally the inner product in $\R^n$ (denoted by $\cdot$) of \eqref{syst} with a smooth function $\varphi$ and integrating in time and space, we are lead to the following definition of solution (see e.g. \cite[pp. 250-251]{BC16}):
\begin{definition}\label{def weak sol}
Let $y^0 \in L^2(0,1)^n$ and $u \in L^2(0,+\infty)^m$.
We say that a function $y$ is a (weak) solution to \eqref{syst} if $y \in C^0([0,+\infty);L^2(0,1)^n)$ and, for every $T>0$,
\begin{multline}\label{weak sol}
\int_0^1 y(T,x) \cdot \varphi(T,x) \, dx
-\int_0^1 y^0(x) \cdot \varphi(0,x) \, dx
\\
=\int_0^T \int_0^1 y(t,x) \cdot \left(\pt{\varphi}(t,x)-\Lambda(x)\px{\varphi}(t,x)+\left(-\px{\Lambda}(x)+M(x)^*\right)\varphi(t,x)\right)\, dx dt
\\
+\int_0^T u(t) \cdot \Lambda_{-}(1)\varphi_-(t,1) \, dt,
\end{multline}
for every $\varphi \in C^1([0,T]\times[0,1])^n$ such that $\varphi_+(\cdot,1)=0$ and $\varphi_-(\cdot,0)=R^*\varphi_+(\cdot,0)$, where $R \in \R^{p \times m}$ is defined by
\begin{equation}\label{def R}
R=-\Lambda_{+}(0)Q\Lambda_{-}(0)^{-1},
\end{equation}
and $\Lambda_{+}=\diag(\lambda_1,\ldots,\lambda_p)$ and $\Lambda_{-}=\diag(\lambda_{p+1}, \ldots, \lambda_n)$.
\end{definition}

We recall that $\Lambda \in C^{0,1}([0,1])^{n \times n}=W^{1,\infty}(0,1)^{n \times n}$ so that $\px{\Lambda}$ exists and belongs to $L^{\infty}(0,1)^{n \times n}$.
We can establish that system \eqref{syst} is well-posed, that is, for every $y^0 \in L^2(0,1)^n$ and $u \in L^2(0,+\infty)^m$, there exists a unique solution $y \in C^0([0,+\infty);L^2(0,1)^n)$ to \eqref{syst} and this solution depends continuously on $y^0$ and $u$ on compact time intervals (see e.g. Section \ref{sect abst} below).
The regularity of the solution to \eqref{syst} allows us to consider control problems in $L^2(0,1)^n$.
We say that the system \eqref{syst} is:
\begin{itemize}
\item
exactly controllable in time $T$ if, for every $y^0,y^1 \in L^2(0,1)^n$, there exists $u \in L^2(0,+\infty)^m$ such that the corresponding solution $y \in C^0([0,+\infty);L^2(0,1)^n)$ to system \eqref{syst} satisfies $y(T)=y^1$.
\item
null controllable in time $T$ if the previous property holds at least for $y^1=0$.
\item
approximately controllable in time $T$ if, for every $\epsilon>0$ and every $y^0,y^1 \in L^2(0,1)^n$, there exists $u \in L^2(0,+\infty)^m$ such that the corresponding solution $y \in C^0([0,+\infty);L^2(0,1)^n)$ to system \eqref{syst} satisfies $\norm{y(T)-y^1}_{L^2(0,1)^n} \leq \epsilon$.
\item
approximately null controllable in time $T$ if the previous property holds at least for $y^1=0$.
\end{itemize}
Clearly, exact controllability implies all the other controllability notions and approximate null controllability is implied by all the other controllability notions.
On the other hand, for the system \eqref{syst}, null controllability in a time $T$ implies exact controllability in the same time, if we assume that $\rank Q=p$ (which is a necessary condition for the exact controllability of \eqref{syst} to hold in some time, as we shall see below).
This is easily seen by using a similar argument to that for systems which are reversible in time (even though it is not the case for \eqref{syst}).
Indeed, take any $\clos{Q} \in \R^{m \times p}$ such that $Q\clos{Q}=\Id_{\R^{p \times p}}$ and consider the system without control
$$
\left\{\begin{array}{l}
\ds \pt{\clos{y}}(t,x)=\Lambda(x) \px{\clos{y}}(t,x)+M(x) \clos{y}(t,x), \\
\clos{y}_-(t,0)=\clos{Q} \clos{y}_+(t,0), \quad \clos{y}_+(t,1)=0,  \\
\clos{y}(T,x)=y^1(x),
\end{array}\right.
\quad t \in (0,T), x \in (0,1),
$$
and then the controlled system
$$
\left\{\begin{array}{l}
\ds \pt{\tilde{y}}(t,x)=\Lambda(x) \px{\tilde{y}}(t,x)+M(x) \tilde{y}(t,x), \\
\tilde{y}_+(t,0)=Q\tilde{y}_-(t,0), \quad \tilde{y}_-(t,1)=\tilde{u}(t),  \\
\tilde{y}(0,x)=y^0(x)-\clos{y}(0,x), \quad \tilde{y}(T,x)=0,
\end{array}\right.
\quad t \in (0,T), x \in (0,1).
$$
Taking $u(t)=\clos{y}_-(t,1)+\tilde{u}(t)$ we see by uniqueness that $y=\clos{y}+\tilde{y}$ (in particular, $y(T)=y^1$).

For any $(\Lambda,M,Q)$ that satisfies the above standing assumptions, we denote by $\Topt{M} \in [0,+\infty]$ the minimal time for the exact controllability of \eqref{syst}, that is
\begin{equation}\label{def Topt}
\Topt{M}=
\inf\ens{T>0, \quad \eqref{syst} \mbox{ is exactly controllable in time } T}.
\end{equation}
The time $\Topt{M}$ is named ``minimal time'' according to the current literature, despite it is not always a minimal element of the set.
We keep this naming here, but we use the notation with the ``inf'' to avoid eventual confusions.
Since exact controllability in time $T_1$ clearly implies exact controllability in time $T_2$ for every $T_2 \geq T_1$, the time $\Topt{M} \in [0,+\infty]$ is also the unique time that satisfies the following two properties:
\begin{itemize}
\item
If $T>\Topt{M}$, then \eqref{syst} is exactly controllable in time $T$.
\item
If $T<\Topt{M}$, then \eqref{syst} is not exactly controllable in time $T$.
\end{itemize}

The goal of the present article is precisely to explicitly characterize $\Topt{M}$ in terms of $\Lambda$, $M$ and $Q$.
To the best of our knowledge, finding the minimal time for the controllability of one-dimensional first-order linear hyperbolic systems is a problem that dates back at least to the celebrated survey \cite{Rus78}.
In this article, the author started by introducing two basic times, one for which we always have null controllability after this time, whatever $M$ and $Q$ are, and another one for which in general (i.e. for some $M$ and $Q$) we do not have null controllability before this other time.
The author then tried to sharpen these preliminary results by looking more closely at the boundary coupling term $Q$.
He naturally started his study with the case of no internal coupling term for the adjoint system, i.e. $M=\px{\Lambda}$, but even in this simplified version he did not succeed to obtain the minimal time of null controllability and he left this as an open problem: ``This raises the question, unresolved at the moment, concerning the identification of a ``critical time'' $T_c$ such that observability holds if $T \geq T_c$ and does not hold if $T<T_c$. Such a critical time $T_c$ can readily be shown to exist but no satisfactory characterization of it is available at this writing''.
This problem was completely solved few years later in \cite{Wec82}.
There, for any diagonal $M$, the author gave an explicit expression of this critical time $T_c$ in terms of some indices related to $Q$.
Some exact controllability results for non diagonal $M$ were also obtained in \cite{Rus78}, by assuming in addition that $\rank Q=p$ and using some perturbation arguments, but in these results $M$ has to be either small, either such that the corresponding system is approximately controllable.

Following the works of \cite{Rus78} and \cite{Wec82}, we see that this left open in particular one natural question, which is the characterization of the minimal time for the null or exact controllability of systems with general internal couplings $M$ (which are not necessarily diagonal, small, etc.).
This is obviously a non trivial problem since the equations now become coupled inside the domain as well.
Moreover, the problem is in fact not only technical since, for instance for the null controllability property ($\rank Q<p$), the time $T_c$ found in \cite{Wec82} is not, in general, the minimal time of control when $M$ is not anymore diagonal.
This is implicitly illustrated by a simple $2 \times 2$ example in \cite{Rus78} (see Remark \ref{rem NC pas stab} below).

%
%

This problem was recently investigated in \cite{CN18} using another method: the so-called backstepping method.
Thanks to this technique it is in particular established there that the system remains null controllable in some time (that we will prove below is in fact $T_c$) for internal couplings $M$ of some particular form.
As expected by the counterexample of \cite{Rus78} that we have just mentioned, this was done under some assumptions on $Q$.
Some exact controllability results were also obtained there under these same assumptions and by requiring in addition that $\rank Q=p$.

The purpose of the present paper is to completely characterize the minimal time for the exact controllability of \eqref{syst}, whatever the boundary coupling $Q$ is (thus, generalizing some results of \cite{CN18}) and whatever the internal coupling $M$ is (thus, generalizing the results of \cite{Wec82}).
In particular, we will see that the time of \cite{Wec82} that characterizes the null controllability for diagonal $M$ in fact is also the minimal time for the exact controllability and for general $M$.
As a by-product we will also see that our way to compute this time is more efficient than the procedure introduced in \cite{Wec82}.
Our proof is a development the original ideas of \cite{Rus78}, combined with some results of \cite{DO18} and \cite{NRL86}, and by introducing an accurate factorization of $Q$ similar to the one of \cite{DJM06}.

Finally, we would like to conclude this introductory part by mentioning that there are not a lot of other works in the literature devoted to a characterization of the minimal time of control for this class of systems.
It seems that the attention was mainly directed towards the controllability of quasilinear versions of such systems afterwards, see for instance the book \cite{Li10}, the article \cite{Hu15} and the references therein.
It would be very interesting to see what can be done for such systems regarding the optimality of the control time.

Before going further and precisely stating the main result of this paper, we need to introduce some notations and concepts.
We start with the characteristics associated with system \eqref{syst}.
For every $i \in \ens{1,\ldots,n}$, every $t \geq 0$ and $x \in [0,1]$ fixed, we introduce the characteristic $\caract_i(\cdot;t,x) \in C^1\left(\left[\ssin_i(t,x),\ssout_i(t,x)\right]\right)$ passing through $(t,x)$, that is the solution to the ordinary differential equation:

\begin{equation}\label{caract}
\left\{\begin{array}{l}
\ds \dds\caract_i(s;t,x)=-\lambda_i\left(\caract_i(s;t,x)\right), \quad s \in \left[\ssin_i(t,x),\ssout_i(t,x)\right],\\
\caract_i(t;t,x)=x,
\end{array}\right.
\end{equation}
where $\ssin_i(t,x), \ssout_i(t,x) \in \R$ (with $\ssin_i(t,x)<t<\ssout_i(t,x)$) are the enter and exit parameters of the domain $[0,1]$, that is the unique respective solutions to

\begin{equation}\label{def ssin}
\left\{\begin{array}{lll}
\caract_i(\ssin_i(t,x);t,x)=0, & \caract_i(\ssout_i(t,x);t,x)=1, & \text{ if } i \in \ens{1,\ldots,p}, \\
\caract_i(\ssin_i(t,x);t,x)=1, & \caract_i(\ssout_i(t,x);t,x)=0, & \text{ if } i \in \ens{p+1,\ldots,n}.
\end{array}\right.
\end{equation}
Their existence and uniqueness are guaranteed by the assumption \eqref{hyp speeds}.
We then introduce

$$T_i(\Lambda)=
\left\{\begin{array}{ll}
\ssout_i(0,0) & \text{ if } i \in \ens{1,\ldots,p}, \\
\ssout_i(0,1) & \text{ if } i \in \ens{p+1,\ldots,n}.
\end{array}\right.
$$
Since the speeds do not depend on time, the exact value of $T_i(\Lambda)$ can actually be obtained by integrating over $[0,1]$ the differential equation satisfied by the inverse function $\xi \longmapsto \caract_i^{-1}(\xi;t,x)$:
\begin{equation}\label{comp Ti}
T_i(\Lambda)=
\left\{\begin{array}{ll}
\ds -\int_0^1 \frac{1}{\lambda_i(\xi)} \, d\xi & \text{ if } i \in \ens{1,\ldots,p}, \\
\ds \int_0^1 \frac{1}{\lambda_i(\xi)} \, d\xi & \text{ if } i \in \ens{p+1,\ldots,n}.
\end{array}\right.
\end{equation}
For the rest of this article it is important to keep in mind that the assumption \eqref{hyp speeds} implies the following order relation between the $T_i(\Lambda)$:
\begin{equation}\label{order times}
\left\{\begin{array}{l}
T_1(\Lambda) \leq \ldots \leq T_p(\Lambda), \\
T_{p+m}(\Lambda) \leq \ldots \leq T_{p+1}(\Lambda).
\end{array}\right.
\end{equation}


It is nowadays known that the combination of the two largest times
$$T_p(\Lambda)+T_{p+1}(\Lambda)$$
yields a time for which the null (resp. exact) controllability of \eqref{syst} holds (resp. if $\rank Q=p$).
This was proved for instance in \cite[Theorem 3.2]{Rus78} with a slightly different boundary condition at $x=1$ or in \cite[Theorem 3.2]{Li10} using a constructive method, moreover for quasilinear systems.
It is then not difficult to see that $T_p(\Lambda)+T_{p+1}(\Lambda)$ is the sharpest time for the null (resp. exact) controllability of \eqref{syst} which is uniform with respect to all possible choices of $M$ and $Q$ (resp. if $\rank Q=p$).

In \cite{Rus78}, the author then tried to improve the time $T_p(\Lambda)+T_{p+1}(\Lambda)$ according to the properties of $Q$.
Considering first the case $M=\px{\Lambda}$, he introduced in \cite[Propositions 3.3 and 3.4]{Rus78} two times $T_0,T_1>0$ for which the approximate null controllability fails for $T<T_0$ and the null controllability holds for $T \geq T_1$.
However, he observed that in general these two times do not agree and he left the characterization of the minimal time as an open problem.

On the other hand, assuming that $\rank Q=p$, the author deduced some exact controllability results as immediate consequences of the results for the null controllability.
Using then some perturbation arguments, it is proved in \cite[Theorem 3.7]{Rus78} that the system remains exactly controllable in the same time $T_1$ for non diagonal $M$ but the author has to assume that either $M$ is small (in which case the result is in fact not surprising since the exact controllability is a property that is stable by small bounded perturbations, see e.g. \cite[Theorem 4.1]{DR77}), either $M$ is such that the corresponding system is approximately controllable (which is in general not easy to check).

In the case of diagonal $M$, an explicit expression of the minimal time for the null controllability of \eqref{syst} was found in \cite{Wec82}, solving then the previously open problem raised in \cite{Rus78} (in particular, it is shown in \cite[Section 4]{Wec82} that none of the time $T_0$ or $T_1$ of \cite{Rus78} were the minimal time of control).
To precisely state the important result of \cite{Wec82}, we need to introduce some notations.
First of all, let $C_0 \in \R^{m \times p}$ be the matrix defined by
$$C_0=-\Lambda_-(0)^{-1}Q^*\Lambda_+(0)\Sigma,$$
where $\Sigma \in \R^{p \times p}$ is the permutation matrix whose $(i,j)$ entry is equal to $1$ if $i+j=p+1$ and $0$ otherwise (note that $\Sigma^*=\Sigma$ and $\Sigma^2=\Id_{\R^{p \times p}}$).
The introduction of the matrix $\Sigma$ is needed here because the positive speeds are ordered differently in \cite{Wec82}.
For every $\ell \in \ens{0,\ldots,m}$, let us denote by $E^-_\ell \in \R^{m \times m}$ the diagonal matrix whose $(i,i)$ entries are equal to $0$ for every $i \in \ens{1,\ldots,\ell}$ and $1$ otherwise (with the convention that $E^-_0=\Id_{\R^{m \times m}}$).
On the other hand, for every $k \in \ens{1,\ldots,p}$, let $E^+_k \in \R^{p \times p}$ be the diagonal matrix whose $(i,i)$ entries are equal to $1$ for every $i \in \ens{1,\ldots,k}$ and $0$ otherwise.
For every $k \in \ens{1,\ldots,p}$, let then $\ell(k) \in \ens{1,\ldots,m}$ be the unique index such that
\begin{equation}\label{def ellk}
\ker C_0E^+_k=\ker E^-_1C_0E^+_k=\ldots=\ker E^-_{\ell(k)-1}C_0E^+_k\subsetneq \ker E^-_{\ell(k)}C_0E^+_k,
\end{equation}
if it exists (i.e. $C_0E^+_k \neq 0$) and $\ell(k)=\infty$ otherwise.
Finally, $T_c>0$ is the time defined by
\begin{equation}\label{Time Weck}
T_c=\max_{k \in \ens{1,\ldots,p}}\left(T_{p-k+1}(\Lambda)+T_{p+\ell(k)}(\Lambda),T_{p+1}(\Lambda)\right),
\end{equation}
with the convention $T_{p+\infty}(\Lambda)=0$.
It is then proved in \cite[Theorems 1 and 2]{Wec82} that the system \eqref{syst} with diagonal $M$ is null controllable in time $T$ if, and only if, $T \geq T_c$ (let us warn the reader that the naming of the controllability notions in \cite[Definition 1]{Wec82} is different than ours and the current literature).

More recently, some results for the null and exact controllability of \eqref{syst} with non diagonal $M$ have been obtained in \cite{CN18}.
To be more precise, let us introduce the following condition:
\begin{equation}\label{cond CN18}
\text{the $i \times i$ matrix formed from the last $i$ rows and the last $i$ columns of $Q$ is invertible.}
\end{equation}
Then, using the so-called backstepping method, it was proved in \cite[Theorem 2]{CN18} that for every $Q$ such that \eqref{cond CN18} holds for every $i \in \ens{1,\ldots,p}$ and every $M$ of the form $\gamma C$, with $C \in L^{\infty}(0,1)^{n \times n}$ and $\gamma \in \R$ outside some discrete set (depending on $\Lambda, Q$ and $C$ though), the system \eqref{syst} is exactly controllable in time $\TCN$, where
\begin{equation}\label{time CN18}
\TCN=\max_{i \in \ens{1,\ldots,p}} (T_i(\Lambda)+T_{m+i}(\Lambda),T_{p+1}(\Lambda)).
\end{equation}
A similar result is proved for the null controllability in \cite[Theorem 1]{CN18}.
It is also shown in \cite[Theorem 3]{CN18} that the assumption on the particular form $M=\gamma C$ can be dropped if we look for exact (or null) controllability in times $T>\TCN$, but it is done under obviously too restrictive assumptions ($m=2$, $\Lambda$ constant, $M$ analytic in a neighborhood of $x=0$, etc.).

Finally, let us also mention the result \cite[Theorem 1.1]{Hu15} where it is proved, by developing the constructive approach of \cite[Theorem 3.2]{Li10}, that a quasilinear version of \eqref{syst} with $M=0$ is (locally) exactly controllable in time $T$ for every $T>\max\ens{T_{m+1}(\Lambda)+T_p(\Lambda),T_{p+1}(\Lambda)}$, if the condition \eqref{cond CN18} holds for $i=p$ (we point out that this is stronger than just assuming that $\rank Q=p$ when $m>p$).

In this article we will obtain the minimal time for the exact controllability for any fixed $\Lambda, Q$ and $M$, without assuming anything more than $\rank Q=p$.
As already mentioned before, we use a different approach than in the article \cite{CN18} and we go back to the original perturbation idea of the first paper \cite{Rus78}.

To deal with general $Q$ and state our main result we need to introduce the concept of canonical form for full row rank matrices (a related notion can be found in \cite[Definition 2]{DJM06}):

\begin{definition}\label{def canon}
We say that a matrix $Q^0 \in \R^{p \times m}$ is in canonical form if there exist distinct column indices $c_1(Q^0), \ldots, c_p(Q^0) \in \ens{1,\ldots,m}$ such that:
\begin{equation}\label{def ci}
\forall i \in \ens{1,\ldots,p}, \quad
\left\{\begin{array}{l}
q^0_{i,c_i(Q^0)} \neq 0, \\
q^0_{i,j}=0, \quad \forall j>c_i(Q^0), \quad j \not\in \ens{c_{i+1}(Q^0),\ldots,c_p(Q^0)}, \\
q^0_{i,j}=0, \quad \forall j<c_i(Q^0).
\end{array}\right.
\end{equation}
\end{definition}

\begin{example}\label{ex Qzero}
Consider the following matrices
$$
Q_1^0=
\begin{pmatrix}
0 & \fbox{$1$} & 4 & -1 \\
0 & 0 & \fbox{$2$} & 3 \\
0 & 0 & 0 & \fbox{$1$}
\end{pmatrix},
\qquad
Q_2^0=
\begin{pmatrix}
0 & 0 & \fbox{$4$} \\
\fbox{$1$} & 2 & 0 \\
0 & \fbox{$1$} & 0
\end{pmatrix},
\qquad
Q_3^0=
\begin{pmatrix}
1 & 4 & -1 & 0\\
 0 & 2 & 3 & 0\\
 0 & 0 & 1 & 1
\end{pmatrix}.
$$
The matrices $Q_1^0$ and $Q_2^0$ are both in canonical form, with $c_3(Q_1^0)=4$, $c_2(Q_1^0)=3$, $c_1(Q_1^0)=2$ and $c_3(Q_2^0)=2$, $c_2(Q_2^0)=1$, $c_1(Q_2^0)=3$.
However, $Q_3^0$ is not in canonical form because there is no $c_3(Q_3^0)$ that simultaneously satisfies the second and third conditions of \eqref{def ci}.
\end{example}

\begin{remark}\label{rem zero dessous}
If $Q^0 \in \R^{p\times m}$ is in canonical form, then necessarily:
\begin{enumerate}[(i)]
\item
The indices $c_1(Q^0),\ldots,c_p(Q^0)$ are unique.

\item
$q^0_{i,j}=0$ for every $i \in \ens{1,\ldots,p}$ and $j \not\in \ens{c_1(Q^0),\ldots,c_p(Q^0)}$.

\item
$\rank Q^0=p$.

\item
We have
\begin{equation}\label{cond proof direct}
q^0_{k,c_i(Q^0)}=0,\quad \forall k>i, \quad \forall i \in \ens{1,\ldots,p}.
\end{equation}
\end{enumerate}
The first point is clear since $c_i(Q^0)$ is the column index of the unique non-zero entry of the $i$-th row of $Q^0$ that is not in the columns with indices $c_{i+1}(Q^0),\ldots,c_p(Q^0)$.
The second point immediately follows from the two last conditions in \eqref{def ci}.
The third point is also clear by considering a linear combination of the only $p$ non-zero columns of $Q^0$ and looking first at its last row, then at its last but one row, etc.
For the last point, first note that for $i=p$, \eqref{cond proof direct} is clear since there is no condition ($k \in \ens{1,\ldots,p}$).
For $i=p-1$, we have to check that $q^0_{p,c_{p-1}(Q^0)}=0$.
Since $c_{p-1}(Q^0) \neq c_p(Q^0)$ we have two possibilities, either $c_{p-1}(Q^0)<c_p(Q^0)$ so that the equality follows from the last condition in \eqref{def ci}, either $c_{p-1}(Q^0)>c_p(Q^0)$ so that the equality follows from the second condition in \eqref{def ci}.
Repeating the reasoning for $i=p-2, p-3$, etc. eventually leads to \eqref{cond proof direct}.
\end{remark}

Next, we present a result that comes from the Gaussian elimination and that we will call in this article ``canonical $UL$--decomposition'' ($U$ for upper and $L$ for lower, see also Remark \ref{rem optimality} below for this naming):

\begin{proposition}\label{Gauss elim}
Let $Q \in \R^{p \times m}$ with $\rank Q=p$.
Then, there exists a unique $Q^0 \in \R^{p \times m}$ such that the following two properties hold:
\begin{enumerate}[(i)]
\item
There exists $L \in \R^{m \times m}$ such that $QL=Q^0$ with $L$ lower triangular ($\ell_{ij}=0$ if $i<j$) and with only ones on its diagonal ($\ell_{ii}=1$ for every $i$).
\item
$Q^0$ is in canonical form.
\end{enumerate}
We call $Q^0$ the canonical form of $Q$.
\end{proposition}

We mention that, because of possible zero columns of $Q$, the matrix $L$ is in general not unique.
The proof of Proposition \ref{Gauss elim} is given in Appendix \ref{appendix}.
With this proposition, we can extend the definition of the $c_i$ indices in Definition \ref{def canon} to any full row rank matrix:

\begin{definition}\label{def generale ci}
Let $Q \in \R^{p \times m}$ with $\rank Q=p$.
We define $c_1(Q),\ldots,c_p(Q) \in \ens{1,\ldots,m}$ by
$$c_i(Q)=c_i(Q^0),$$
where $Q^0$ is the canonical form of $Q$ provided by Proposition \ref{Gauss elim}.
\end{definition}

\begin{example}\label{ex Q}
We illustrate how the find the decomposition of Proposition \ref{Gauss elim} in practice.
Consider
$$
Q_1=
\begin{pmatrix}
4 & 6 & 3 & -1 \\
8 & -1 & 5 & 3 \\
2 & -1 & 1 & 1
\end{pmatrix},
\qquad
Q_2=
\begin{pmatrix}
4 & -4 & 4 \\
5 & 2 & 0 \\
2 & 1 & 0
\end{pmatrix}.
$$
Let us deal with $Q_1$ first.
We look at the last row, we take the last nonzero entry as pivot.
We remove the entries to the left on the same row by doing the column substitutions $C_3 \leftarrow  C_3 -C_4$, $C_2 \leftarrow C_2+C_4$ and $C_1 \leftarrow C_1 -2 C_4$ so that
$$
Q_1L_1=
Q_1
\begin{pmatrix}
1 & 0 & 0 & 0 \\
0 & 1 & 0 & 0 \\
0 & 0 & 1 & 0 \\
-2 & 1 & -1 & 1
\end{pmatrix}
=
\begin{pmatrix}
6 & 5 & 4 & -1 \\
2 & 2 & 2 & 3 \\
0& 0 & 0 & 1
\end{pmatrix}
.
$$
We now move up one row and take as new pivot the last nonzero entry that is not in $C_4$.
We remove the entries to the left on the same row by doing the column substitutions $C_2 \leftarrow  C_2 -C_3$ and $C_1 \leftarrow C_1-C_3$ so that
$$
Q_1L_1L_2
=
Q_1L_1
\begin{pmatrix}
1 & 0 & 0 & 0 \\
0 & 1 & 0 & 0 \\
-1 & -1 & 1 & 0 \\
0 & 0 & 0 & 1
\end{pmatrix}
=
\begin{pmatrix}
2 & 1 & 4 & -1 \\
0 & 0 & 2 & 3 \\
0& 0 & 0 & 1
\end{pmatrix}
.
$$
Finally, a last substitution shows that $Q_1$ becomes $Q_1^0$ of Example \ref{ex Qzero}, namely:
$$
Q_1L=
Q_1L_1L_2
\begin{pmatrix}
1 & 0 & 0 & 0 \\
-2 & 1 & 0 & 0 \\
0 & 0 & 1 & 0 \\
0 & 0 & 0 & 1
\end{pmatrix}
=
\begin{pmatrix}
0 & 1 & 4 & -1 \\
0 & 0 & 2 & 3 \\
0& 0 & 0 & 1
\end{pmatrix}
=Q_1^0.
$$
Similarly, it can be checked the canonical form of $Q_2$ is in fact $Q_2^0$ of Example \ref{ex Qzero}.
\end{example}

\begin{remark}
Where we want to put entries to zero in Example \ref{ex Q} in fact depends on the way the times are ordered \eqref{order times}.
This will be more clear during the proof of Theorem \ref{thm unp syst} below.
We mention this point to highlight the fact that the definition of the canonical form is linked to this ordering.
\end{remark}

After such a long but necessary preparation we can now clearly state the main result of this paper:

\begin{theorem}\label{main thm}
Let $\Lambda \in C^{0,1}([0,1])^{n \times n}$ satisfy \eqref{Lambda diag}, \eqref{hyp speeds} and \eqref{hyp Rus78 eg}, $M \in L^{\infty}(0,1)^{n \times n}$ and $Q \in \R^{p\times m}$ be fixed.
We have:
\begin{enumerate}[(i)]
\item\label{conclu 1 main thm}
$\Topt{M}<+\infty$ if, and only if, $\rank Q=p$.

\item\label{conclu 2 main thm}
If $\rank Q=p$, then
\begin{equation}\label{thm Topt}
\Topt{M}=\max_{i \in \ens{1,\ldots,p}} (T_{p+1}(\Lambda), T_i(\Lambda)+T_{p+c_i(Q)}(\Lambda)),
\end{equation}
where $c_1(Q),\ldots,c_p(Q) \in \ens{1,\ldots,m}$ are defined in Definition \ref{def generale ci}.
\end{enumerate}
\end{theorem}

To the best of our knowledge, this is the first result that completely characterizes the minimal time for the exact controllability of \eqref{syst} for any given $M$ and $Q$.
Not only this, but this result also shows that the time \eqref{thm Topt} is explicit in terms of $\Lambda$ (recall \eqref{comp Ti}) and in terms of $Q$ as well, since the computation of the indices $c_i(Q)$ rely on the Gaussian elimination, which is a very efficient algorithm that shows that the minimal time \eqref{thm Topt} is actually easy to compute in practice.

\begin{example}\label{rem ex}
A comparison with the results of \cite{CN18} can be made.
For $Q_1 \in \R^{3 \times 4}$ of Example \ref{ex Q} we have
$$\Toptlib{\Lambda,M,Q_1}=\max\left(T_4(\Lambda),T_1(\Lambda)+T_5(\Lambda),T_2(\Lambda)+T_6(\Lambda),T_3(\Lambda)+T_7(\Lambda)\right)=\TCN.$$
On the contrary, the case of $Q_2 \in \R^{3 \times 3}$ of Example \ref{ex Q} is not covered by the results \cite{CN18}, and for this parameter we have
$$
\begin{array}{rl}
\Toptlib{\Lambda,M,Q_2} &=\max\left(T_4(\Lambda),T_1(\Lambda)+T_6(\Lambda),T_2(\Lambda)+T_4(\Lambda),T_3(\Lambda)+T_5(\Lambda)\right) \\
&=\max\left(T_2(\Lambda)+T_4(\Lambda),T_3(\Lambda)+T_5(\Lambda)\right).
\end{array}
$$
\end{example}

\begin{remark}\label{rem time Wec82}
In Appendix \ref{app time Weck} below we prove that (assuming $\rank Q=p$)
$$\max_{i \in \ens{1,\ldots,p}} (T_{p+1}(\Lambda), T_i(\Lambda)+T_{p+c_i(Q)}(\Lambda))=T_c,$$
where we recall that $T_c$ is given in \eqref{Time Weck}.
Therefore, Theorem \ref{main thm} shows that the time $T_c$ introduced in \cite{Wec82} for the null controllability of \eqref{syst} with diagonal $M$ is also the minimal time for the exact controllability of \eqref{syst} for arbitrary $M$.
As a by-product, our method gives the most efficient way to compute the time $T_c$, which, a priori by the look of \eqref{Time Weck}-\eqref{def ellk}, would require more computations (we invite the reader to consider the example in \cite[Section 4]{Wec82}: it requires a single computation to find $c_2(C)=2$ and $c_1(C)=1$).
\end{remark}

\begin{remark}
The assumption \eqref{hyp Rus78 eg} has been introduced in \cite[Section 3]{Rus78}.
If it is not satisfied, then the conclusion \ref{conclu 2 main thm} of Theorem \ref{main thm} is no longer true in general.
We have detailed a counterexample in Appendix \ref{app counterexample} below that shows that the time $T_p(\Lambda)+T_{p+1}(\Lambda)$ may not be improved in such a case.
\end{remark}

\begin{remark}\label{rem indep M}
Observe that the expression \eqref{thm Topt} of $\Topt{M}$ does not depend on $M$.
This means that the internal coupling terms $M(x)y(t,x)$ in \eqref{syst} have almost no impact on the controllability properties of this system.
All our attention should then be on the coupling on the boundary $Q$.
Let us however mention that whether the infimum in the definition \eqref{def Topt} of $\Topt{M}$ is or is not a minimum depends on the values of $M$.
For instance we will see in Section \ref{sect unp syst} below that for $M=0$ the infimum is reached.
This also remains true for nonzero but sufficiently small $M$ since the exact controllability is a property that is stable by small bounded perturbations (see e.g. \cite[Theorem 4.1]{DR77}).
On the other hand, there exists $M$ such that the infimum is not a minimum.
In fact, by using the techniques we will develop below, it can be shown that the minimum is reached if, and only if, \eqref{syst} is approximately controllable in time $\Topt{M}$, and it is known that this latter property may fail, as for instance illustrated in \cite[pp. 659-661]{Rus78} (see also item 2. of \cite[Theorem 1]{CN18} and Appendix \ref{app counterexample} below).
A complete characterization of the parameters $M$ and $Q$ for which the infimum is equal to the minimum seems still an open problem (some partial results can be found in \cite{CN18}).
\end{remark}

\begin{remark}\label{rem optimality}
We have seen that $T_p(\Lambda)+T_{p+1}(\Lambda)$ is the worst possible time of control.
On the other hand, it can be checked that (assuming that $m \geq p$)
$$
\min_{\substack{(c_1,\ldots,c_p) \in \ens{1,\ldots,m} \\ c_j \neq c_k, \, j \neq k}}
\left(\max_{i \in \ens{1,\ldots,p}} (T_{p+1}(\Lambda), T_i(\Lambda)+T_{p+c_i}(\Lambda))\right)
=\TCN,
$$
where we recall that $\TCN$ is defined in \eqref{time CN18}, and the minimum is reached for $c_i$ satisfying
\begin{equation}\label{best ci}
c_i=m-p+i, \quad \forall i \in \ens{1,\ldots,p}.
\end{equation}
The condition \eqref{best ci} means that the canonical form $Q^0$ of $Q$ is an upper triangular matrix, see e.g. $Q_1^0$ of Example \ref{ex Qzero}.
Thus in this case $Q$ has a ``standard'' $UL$--decomposition.
Moreover, it can be shown with the Gaussian elimination that a full row-rank matrix $Q$ admits such a decomposition if, and only if, $Q$ satisfies \eqref{cond CN18} for every $ i \in \ens{1,\ldots,p}$ (see e.g. \cite[Theorem II.1]{Gan59}).
As a result, we see that we recover the time and the assumption given in \cite[Theorem 2]{CN18}.
Note as well that our observation justifies the name of ``optimal time'' given in this article (before it, there were no real justification to such a naming).
\end{remark}

\begin{remark}
Let us emphasize that all along this work we are interested in the controllability properties in the space $L^2(0,1)^n$, which means that all the components of the system belong to the same space $L^2(0,1)$.
The behavior of \eqref{syst} is very different if we allow the components to lie in different spaces.
For instance, the exact controllability can hold even if $\rank Q<p$ (compare with \ref{conclu 1 main thm} of Theorem \ref{main thm}) and the internal coupling term $M$ can help to make a system become exactly controllable (compare with Remark \ref{rem indep M}).
We refer for instance to \cite[Theorem 10.1]{Li10} for an illustration of such a situation.
\end{remark}

The rest of the paper is organized as follows.
In the next section we simply recast the system \eqref{syst} into its abstract form and prove basic properties.
In Section \ref{sect unp syst}, we make use of the notion of canonical $UL$--decomposition to establish necessary and sufficient conditions for the system \eqref{syst} to be exactly controllable in a given time when there are no internal coupling terms, i.e. when $M=0$.
In Section \ref{sect cpct uniq} we use compactness-uniqueness arguments to show that the minimal time of control remains the same when we add a bounded perturbation $M$.
Finally, we postponed in the appendix several auxiliary results for the sake of
the presentation.

\section{Abstract setting}\label{sect abst}

It is well-known that the system \eqref{syst} can equivalently be rewritten as an abstract evolution system:
\begin{equation}\label{abst syst}
\left\{\begin{array}{l}
\ds \ddt y(t)=A_M y(t)+Bu(t), \quad t \in (0,+\infty), \\
y(0)=y^0,
\end{array}\right.
\end{equation}
also to be referred to as $(A_M,B)$ in the sequel, where we can identify the operators $A_M$ and $B$ through their adjoints by formally taking the inner product of \eqref{abst syst} with a smooth function $\varphi$ and then comparing with \eqref{weak sol}.
The state and control spaces are
$$H=L^2(0,1)^n, \quad U=\R^m.$$
They are equipped with their usual inner products and identified with their dual.
The unbounded linear operator $A_M: \dom{A_M} \subset H \longrightarrow H$ is defined, for every $y \in \dom{A_M}$ by
$$
A_My(x)=\Lambda(x) \px{y}(x)+M(x) y(x),
\quad x \in (0,1),
$$
with domain
$$\dom{A_M}
=\ens{y \in H^1(0,1)^n, \quad y_+(0)=Qy_-(0),\quad y_-(1)=0}.
$$
It is clear that $\dom{A_M}$ is dense in $H$ since it contains $C^{\infty}_c(0,1)^n$.
A computation shows that
$$
\dom{A_M^*}
=\ens{z \in H^1(0,1)^n, \quad z_+(1)=0, \quad z_-(0)=\tr{R}z_+(0)},
$$
where we recall that $R \in \R^{p \times m}$ is defined in \eqref{def R}, and we have, for every $z \in \dom{A_M^*}$,
\begin{equation}\label{op adj}
A_M^*z(x)=
-\Lambda(x)\px{z}(x)
+\left(-\px{\Lambda}(x)+M(x)^*\right)z(x),
\quad x \in (0,1).
\end{equation}
Note that in fact $\dom{A_M^*}$ does not depend on $M$.
On the other hand, the control operator $B \in \lin{U,\dom{A_M^*}'}$ is given for every $u \in U$ and $z \in \dom{A_M^*}$ by
$$
\ps{Bu}{z}{\dom{A_M^*}',\dom{A_M^*}}=u \cdot  \Lambda_{-}(1) z_-(1).
$$
Note that $B$ is well-defined since $Bu$ is continuous on $H^1(0,L)^n$ (by the trace theorem $H^1(0,1)^n \hookrightarrow C^0([0,1])^n$) and since $\norm{\cdot}_{\dom{A_M^*}}$ and $\norm{\cdot}_{H^1(0,1)^n}$ are equivalent norms on $\dom{A_M^*}$.
Finally, the adjoint $B^* \in \lin{\dom{A_M^*},U}$ is given for every $z \in \dom{A_M^*}$ by
$$B^*z=\Lambda_{-}(1) z_-(1).$$

Using the method of characteristics, it is not difficult to show that the operator $A_M$ generates a $C_0$-semigroup when $M$ is diagonal and we even have an explicit formula for it.
Since we will mainly perform computations on the adjoint semigroup in the sequel, it is then when $M=\px{\Lambda}$ that the adjoint semigroup will have the simplest expression (see \eqref{op adj}).

\begin{proposition}\label{prop semigp explicit}
For every $i \in \ens{1,\ldots,p}$ and $j \in \ens{1,\ldots,m}$, let $\phi_i,\phi_{p+j} \in C^{1,1}([0,1])$ be the non-negative and increasing functions defined for every $x \in [0,1]$ by
\begin{equation}\label{def phi}
\phi_i(x)=-\int_0^x \frac{1}{\lambda_i(\xi)} \, d\xi,
\quad
\phi_{p+j}(x)=\int_0^x \frac{1}{\lambda_{p+j}(\xi)} \, d\xi,
\end{equation}
(note that $\phi_i(1)=T_i(\Lambda)$ and $\phi_{p+j}(1)=T_{p+j}(\Lambda)$, see \eqref{comp Ti}).
Then, the operator $A_{\px{\Lambda}}^*$ generates a $C_0$-semigroup on $H$ given, for every $t \geq 0$ and $z^0 \in H$, by
\begin{equation}\label{formula semigp 1}
\left(S_{A_{\px{\Lambda}}}(t)^*z^0\right)_i(x)
=
\left\{\begin{array}{cl}
\ds z_i^0\left(\phi_i^{-1}\left(t+\phi_i(x)\right)\right), & \mbox{ if }t+\phi_i(x)<\phi_i(1), \\
\ds 0, & \mbox{ if } t+\phi_i(x)>\phi_i(1),
\end{array}\right.
\end{equation}
for every $i \in \ens{1,\ldots,p}$ and a.e. $x \in (0,1)$, and by
\begin{multline}\label{formula semigp 2}
\left(S_{A_{\px{\Lambda}}}(t)^*z^0\right)_{p+j}(x)
\\
=
\left\{\begin{array}{cl}
\ds z_{p+j}^0\left(\phi_{p+j}^{-1}\left(\phi_{p+j}(x)-t\right)\right), & \mbox{ if } t-\phi_{p+j}(x)<0, \\
\ds \sum_{i=1}^p r_{i,p+j}
z^0_i\left(\phi_i^{-1}\left(t-\phi_{p+j}(x)\right)\right), & \mbox{ if } 0<t-\phi_{p+j}(x)<\phi_1(1), \\
\vdots & \vdots \\
\ds \sum_{i=k+1}^p r_{i,p+j}
z^0_i\left(\phi_i^{-1}\left(t-\phi_{p+j}(x)\right)\right), & \mbox{ if } \phi_k(1)<t-\phi_{p+j}(x)<\phi_{k+1}(1), \\
\vdots & \vdots \\
0 &  \mbox{ if }\phi_p(1)<t-\phi_{p+j}(x),
\end{array}\right.
\end{multline}
for every $j \in \ens{1,\ldots,m}$ and a.e. $x \in (0,1)$.
\end{proposition}

\begin{proof}
We only show how to find the formula \eqref{formula semigp 1} and \eqref{formula semigp 2}.
It can be checked afterwards that these formula define a $C_0$-semigroup and that $A_{\px{\Lambda}}^*$ is indeed the corresponding generator (by using the very definition of what is a $C_0$-semigroup).
We recall that $\tilde{z}(t)=S_{A_{\px{\Lambda}}}(t)^*z^0$ is the unique solution to the following abstract O.D.E. when $z^0 \in \dom{A_{\px{\Lambda}}^*}$ (see e.g. \cite[Lemma II.1.3]{EN00}):
$$
\left\{\begin{array}{l}
\ds \ddt \tilde{z}(t)=A_{\px{\Lambda}}^* \tilde{z}(t), \quad t \in [0,+\infty), \\
\tilde{z}(0)=z^0.
\end{array}\right.
$$
Therefore, we expect $\tilde{z}$ to solve 
\begin{equation}\label{syst adj f}
\left\{\begin{array}{l}
\ds \pt{\tilde{z}}(t,x)=-\Lambda(x) \px{\tilde{z}}(t,x), \\
\tilde{z}_+(t,1)=0, \quad \ds \tilde{z}_-(t,0)=R^*\tilde{z}_+(t,0),  \\
\tilde{z}(0,x)=z^0(x),
\end{array}\right.
\quad t \in [0,+\infty), x \in (0,1).
\end{equation}
Let us now introduce the characteristics associated to the system \eqref{syst adj f}.
For every $i \in \ens{1,\ldots,n}$, every $t \geq 0$ and $x \in [0,1]$ fixed, we introduce the characteristic $\caractt_i(\cdot;t,x) \in C^1\left(\left[\ssint_i(t,x),\ssoutt_i(t,x)\right]\right)$ passing through $(t,x)$, that is the solution to the ordinary differential equation:
$$
\left\{\begin{array}{l}
\ds \dds\caractt_i(s;t,x)=\lambda_i\left(\caractt_i(s;t,x)\right), \quad s \in \left[\ssint_i(t,x),\ssoutt_i(t,x)\right],\\
\caractt_i(t;t,x)=x,
\end{array}\right.
$$
where $\ssint_i(t,x), \ssoutt_i(t,x) \in \R$ (with $\ssint_i(t,x)<t<\ssoutt_i(t,x)$) are the enter and exit parameters of the domain $[0,1]$, that is the unique respective solutions to
\begin{equation}\label{def ssin bis}
\left\{\begin{array}{lll}
\caractt_i(\ssint_i(t,x);t,x)=1, & \caractt_i(\ssoutt_i(t,x);t,x)=0, & \text{ if } i \in \ens{1,\ldots,p}, \\
\caractt_i(\ssint_i(t,x);t,x)=0, & \caractt_i(\ssoutt_i(t,x);t,x)=1, & \text{ if } i \in \ens{p+1,\ldots,n}.
\end{array}\right.
\end{equation}
Let us first find $\tilde{z}_i$ for $i \in \ens{1,\ldots,p}$.
Since $\tilde{z}_i$ solves
$$
\left\{\begin{array}{l}
\ds \pt{\tilde{z}_i}(t,x)+\lambda_i(x) \px{\tilde{z}_i}(t,x)=0, \\
 \tilde{z}_i(t,1)=0,  \\
\tilde{z}_i(0,x)=z_i^0(x),
\end{array}\right.
\quad t \in [0,+\infty), x \in (0,1),
$$
along the characteristic $\caractt_i$ we have
$$\frac{d}{ds} \tilde{z}_i\left(s,\caractt_i(s;t,x)\right)=0, \quad \forall s \in [\ssint_i(t,x),\ssoutt_i(t,x)], \quad s \in [0,+\infty).$$
It follows that
\begin{equation}\label{formula semigp 1 gen}
\tilde{z}_i(t,x)=
\left\{\begin{array}{cl}
\ds z_i^0\left(\caractt_i(0;t,x)\right), & \mbox{ if } \ssint_i(t,x)<0, \\
\ds 0, & \mbox{ if } \ssint_i(t,x)>0.
\end{array}\right.
\end{equation}
On the other hand, for $j \in \ens{1,\ldots,m}$, similar computations lead to
\begin{equation}\label{formula semigp 2 gen}
\tilde{z}_{p+j}(t,x)=
\left\{\begin{array}{cl}
\ds z_{p+j}^0\left(\caractt_{p+j}(0;t,x)\right), & \ssint_{p+j}(t,x)<0, \\
\ds \sum_{i=1}^p r_{i,p+j}\tilde{z}_i\left(\ssint_{p+j}(t,x),0\right), & \ssint_{p+j}(t,x)>0.
\end{array}\right.
\end{equation}

Now, since $\lambda_i$ does not depend on time, we have a more explicit formula for $\caractt_i(0;t,x)$ and $\ssint_i(t,x)$.
Indeed, the inverse function $\xi \mapsto \caractt_i^{-1}(\xi;t,x)$ solves
\begin{equation}\label{caract inv zi}
\left\{\begin{array}{l}
\ds \pxi{\caractt_i^{-1}}(\xi;t,x)=\frac{1}{\partials{\caractt_i}\left(\caractt_i^{-1}(\xi;t,x);t,x\right)}=\frac{1}{\lambda_i(\xi)}, \quad \xi \in [0,1],\\
\caractt_i^{-1}(x;t,x)=t.
\end{array}\right.
\end{equation}
Therefore, $\caractt_i^{-1}(y;t,x)=t+\int_x^y \frac{1}{\lambda_i(\xi)} \, d\xi$.
Using the functions \eqref{def phi}, we have
\begin{equation}\label{inv caract}
\caractt_i^{-1}(y;t,x)=
\left\{\begin{array}{ll}
\ds t+\phi_i(x)-\phi_i(y), & \mbox{ if } i \in \ens{1,\ldots,p}, \\
\ds t-\phi_i(x)+\phi_i(y), & \mbox{ if } i \in \ens{p+1,\ldots,n}.
\end{array}\right.
\end{equation}
Recalling the definition \eqref{def ssin bis} of $\ssint_i(t,x)$, we then have
\begin{equation}\label{express ssin}
\ssint_i(t,x)=
\left\{\begin{array}{ll}
\ds t+\phi_i(x)-\phi_i(1), & \mbox{ if } i \in \ens{1,\ldots,p}, \\
\ds t-\phi_i(x), & \mbox{ if } i \in \ens{p+1,\ldots,n},
\end{array}\right.
\end{equation}
and
\begin{equation}\label{express caract}
\caractt_i(0;t,x)=
\left\{\begin{array}{ll}
\ds \phi_i^{-1}\left(t+\phi_i(x)\right), & \mbox{ if } i \in \ens{1,\ldots,p} \mbox{ and } \ssint_i(t,x)<0, \\
\ds \phi_i^{-1}\left(\phi_i(x)-t\right), & \mbox{ if } i \in \ens{p+1,\ldots,n} \mbox{ and } \ssint_i(t,x)<0.
\end{array}\right.
\end{equation}
Plugging these formula in \eqref{formula semigp 1 gen} and \eqref{formula semigp 2 gen}, and taking into account that $\phi_i(1) \leq \phi_{i+1}(1)$ for every $i \in \ens{1,\ldots,p-1}$ by \eqref{order times}, we obtain \eqref{formula semigp 1} and \eqref{formula semigp 2}.

\end{proof}

\begin{remark}\label{rem trace}
Observe that the right-hand sides in \eqref{formula semigp 1} and \eqref{formula semigp 2}, considered as functions of $t$ and $x$,  make sense for $z^0 \in L^2(0,1)^n$ only (i.e. the compositions are well-defined), either for every $t \geq 0$ and a.e. $x \in [0,1]$, or for every $x \in [0,1]$ and a.e. $t \geq 0$.
For instance for \eqref{formula semigp 1} this follows from the fact that the maps $x \in (0,\phi_i^{-1}(\phi_i(1)-t)) \mapsto \phi_i^{-1}\left(t+\phi_i(x)\right)$ and $t \in (0,\phi_i(1)-\phi_i(x)) \mapsto \phi_i^{-1}\left(t+\phi_i(x)\right)$ are $C^1$-diffeomorphisms (for every $t \in [0,\phi_i(1))$ and $x \in [0,1)$, respectively).
For the rest of this article, we then abuse the notation $S_{A_{\px{\Lambda}}}(t)^*z^0(x)$   to denote either of these functions when $z^0 \in L^2(0,1)^n$.
\end{remark}

Let us now turn out to the properties of the control operator $B$.
First of all, it can be checked directly from the formula \eqref{formula semigp 2} that, when $z^0 \in \dom{A_{\px{\Lambda}}^*}$, the function $x \mapsto \left(S_{A_{\px{\Lambda}}}(t)^*z^0\right)_-(x)$ belongs to $H^1(0,1)^m$ and has a trace at $x=1$ equal to $\left(S_{A_{\px{\Lambda}}}(t)^*z^0\right)_-(1)$ since the right-hand side of \eqref{formula semigp 2} is a continuous function of $x$ on $[0,1]$ for such $z^0$.
A simple change of variable then easily shows that, for any $0<T<\phi_n(1)$, there exists $C>0$ such that
\begin{equation}\label{admiss}
\int_0^T \norm{B^*S_{A_{\px{\Lambda}}}(t)^*z^0}_U^2 \, dt \leq C\norm{z^0}_H^2 , \quad \forall z^0 \in \dom{A_{\px{\Lambda}}^*}.
\end{equation}
This property shows that $B$ is a so-called admissible control operator for $A_{\px{\Lambda}}$ (see e.g. \cite[Theorem 4.4.3]{TW09}).

Since the operator $A_M$ is nothing but a bounded perturbation of $A_{\px{\Lambda}}$, it follows that $A_M$ also generates a $C_0$-semigroup on $H$ (see e.g. \cite[Theorem III.1.3]{EN00}) and that $B$ is also admissible for $A_M$ (see e.g. \cite[p. 401]{DO18}).
It also follows that the abstract system \eqref{abst syst} is well-posed in the sense that: for every $y^0 \in H$ and every $u \in L^2(0,+\infty;U)$, there exists a unique solution $y \in C^0([0,+\infty);H)$ to \eqref{abst syst} given by the Duhamel formula (see e.g. \cite[Proposition 4.2.5]{TW09}):
\begin{equation}\label{duhamel}
y(T)=S_{A_M}(T)y^0+\Phi_M(T)u, \quad \forall T \geq 0,
\end{equation}
where $\Phi_M(T)$ is the so-called input map of $(A_M,B)$, that is the linear operator defined for every $u \in L^2(0,+\infty;U)$ by
$$\Phi_M(T) u=\int_0^T S_{A_M}(T-s)Bu(s) \, ds.$$
We recall that a priori $\range \Phi_M(T) \subset \dom{A_M^*}'$ but the admissibility of $B$ in fact means that $\range \Phi_M(T) \subset H$ for some (and hence all) $T>0$ (see e.g. \cite[Definition 4.2.1]{TW09}).
From this assumption it follows that the function $T \in [0,+\infty) \mapsto \Phi_M(T)u \in H$ is continuous for every $u \in L^2(0,+\infty;U)$ (see e.g. \cite[Proposition 4.2.4]{TW09}), so that the function $y$ defined by \eqref{duhamel} indeed belongs to $C^0([0,+\infty);H)$.
From the admissibility of $B$ it also follows that $\Phi_M(T) \in \lin{L^2(0,+\infty;U),H}$ (see e.g. \cite[Proposition 4.2.2]{TW09}).
The adjoint $\Phi_M(T)^* \in \lin{H,L^2(0,+\infty;U)}$ is nothing but the unique continuous linear extension to $H$ of the map that takes $z^1 \in \dom{A_M^*}$ and associates to it the following function of $L^2(0,+\infty;U)$ (see e.g. \cite[Proposition 4.4.1]{TW09}):
$$
t \in (0,+\infty) \longmapsto
\left\{\begin{array}{cl}
B^*S_{A_M}(T-t)^*z^1, & \mbox{ if } t \in (0,T), \\
0, & \mbox{ if } t>T.
\end{array}\right.
$$
Finally, it can be checked that the function $y$ defined by \eqref{duhamel} satisfies \eqref{weak sol} and is thus the (weak) solution to \eqref{syst} in the sense of Definition \ref{def weak sol} (see e.g. \cite[pp. 63-65]{Cor07}).

Let us now recall that all the notions of controllability can be reformulated in terms of $\range \Phi_M(T)$.
Indeed, it is not difficult to see that $(A_M,B)$ is exactly (resp. approximately, approximately null) controllable in time $T$ if, and only if, $\range \Phi_M(T)=H$ (resp. $\clos{\range \Phi_M(T)}=H$, $\range S_{A_M}(T) \subset \clos{\range \Phi_M(T)}$).
It is also well-known that the controllability has a dual concept named observability.
More precisely (see e.g. \cite[Theorem 11.2.1]{TW09}):
\begin{itemize}
\item
$(A_M,B)$ is exactly controllable in time $T$ if, and only if, there exists $C>0$ such that
\begin{equation}\label{obs}
\norm{z^1}_H^2 \leq C \int_0^T \norm{\Phi_M(T)^*z^1(t)}_U^2 \, dt, \quad \forall z^1 \in H.
\end{equation}

\item
$(A_M,B)$ is approximately controllable in time $T$ if, and only if,
\begin{equation}\label{ucp}
\left(\Phi_M(T)^*z^1(t)=0, \quad \aew t \in (0,T)\right) \Longrightarrow z^1=0, \quad \forall z^1 \in H.
\end{equation}

\item
$(A_M,B)$ is approximately null controllable in time $T$ if, and only if,
\begin{equation}\label{ucp ANC}
\left(\Phi_M(T)^*z^1(t)=0, \quad \aew t \in (0,T)\right) \Longrightarrow S_{A_M}(T)^*z^1=0, \quad \forall z^1 \in H.
\end{equation}
\end{itemize}

Finally, for $M=\px{\Lambda}$, the adjoint of the input map $\Phi_{\px{\Lambda}}(T)^*$ is explicit.
Indeed, we see from the formula \eqref{formula semigp 2} that the operator $z^1 \in L^2(0,1)^n \longmapsto \Lambda_{-}(1)\left(S_{A_{\px{\Lambda}}}(T-\cdot)^*z^1\right)_-(1)$ (extended by zero outside $(0,T)$) belongs to $\lin{H,L^2(0,+\infty;U)}$.
Since we have already seen that it agrees with $B^*S_{A_{\px{\Lambda}}}(T-\cdot)^*z^1$ for $z^1 \in  \dom{A_{\px{\Lambda}}^*}$, by uniqueness of the continuous extension, this shows that the adjoint of the input map is given, for every $z^1 \in H$, by
$$\Phi_{\px{\Lambda}}(T)^*z^1(t)=\Lambda_{-}(1)\left(S_{A_{\px{\Lambda}}}(T-t)^*z^1\right)_-(1), \quad \mbox{ a.e. } t \in (0,T).$$

\section{Controllability of the unperturbed system}\label{sect unp syst}

The goal of this section is to characterize the minimal time for the exact controllability of the unperturbed system $(A_0,B)$, i.e. of the system
\begin{equation}\label{unp syst}
\left\{\begin{array}{l}
\ds \pt{y}(t,x)=\Lambda(x) \px{y}(t,x), \\
y_+(t,0)=Qy_-(t,0), \quad y_-(t,1)=u(t),  \\
y(0,x)=y^0(x),
\end{array}\right.
\quad t \in (0,+\infty), x \in (0,1).
\end{equation}

It is indeed natural to first investigate what happens when $M=0$ and constitutes a first step towards our main result Theorem \ref{main thm}.
We will then use a perturbation argument in the next section to deal with internal couplings $M \neq 0$.
For the system \eqref{unp syst} we will actually establish an even more precise result, namely:

\begin{theorem}\label{thm unp syst}
Let $\Lambda \in C^{0,1}([0,1])^{n \times n}$ satisfy \eqref{Lambda diag} and \eqref{hyp speeds}, and $Q \in \R^{p\times m}$ be fixed.
For every $T>0$, \eqref{unp syst} is exactly controllable in time $T$ if, and only if, the following two properties hold:
\begin{enumerate}[(i)]
\item\label{thm unp syst i1}
$\rank Q=p$.
\item\label{thm unp syst i2}
$T \geq \max_{i \in \ens{1,\ldots,p}} (T_{p+1}(\Lambda), T_i(\Lambda)+T_{p+c_i(Q)}(\Lambda))$.
\end{enumerate}
\end{theorem}

\begin{remark}
Note that the assumption \eqref{hyp Rus78 eg} is not needed in Theorem \ref{thm unp syst}.
We also point out that the assumption \eqref{hyp speeds} could be weaken all along this section into the following:
\begin{equation}\label{hyp speeds bis}
\lambda_i(x)<0<\lambda_{p+j}(x), \quad \forall x \in [0,1], \quad \forall i \in \ens{1,\ldots,p}, \quad \forall j \in \ens{1,\ldots,m},
\end{equation}
as long as we assume that the eigenvalues are ordered in such a way that \eqref{order times} holds, which can always be done without loss of generality.
\end{remark}

Theorem \ref{thm unp syst} follows in fact from \cite[Theorems 1 and 2]{Wec82} if we show that $\rank Q=p$ is necessary for the exact controllability and that the time in \ref{thm unp syst i2} is the time of \cite{Wec82}.
The first point is easy as we shall see below and the second point is proven in Appendix \ref{app time Weck} below as already mentioned before.
However, we would like to present a slightly different proof here.
The motivation of this is twofold.
Firstly, it is not really explained where the definition of the indices $\ell(k)$ (through the condition \eqref{def ellk}) comes from in \cite{Wec82}.
Secondly, even if we choose to use the results of \cite{Wec82}, to obtain Theorem \ref{thm unp syst} as it is stated we still need to prove the two points mentioned above (the second being non trivial).
As a result, our proof has the advantage to show why we introduced the notion of canonical $UL$--decomposition and, in addition, it naturally gives an expression of the time $T_c$ that is in practice faster to compute than in the formulation of \cite{Wec82} (see Remark \ref{rem time Wec82}).

Let us also recall that Theorem \ref{thm unp syst} has been obtained independently in \cite[Proposition 1]{CN18} but only under stronger assumptions on $Q$ (namely, it has to satisfy \eqref{cond CN18} for every $i \in \ens{1,\ldots,p}$).
Finally, we would like to mention \cite[Theorem 1.1]{Hu15} for a related result concerning a quasilinear version of \eqref{unp syst}.

The key point to solve this problem is to carefully investigate the boundary condition at $x=0$, which somehow allows to transfer the actions of the controls to the indirectly controlled components (i.e. to the components associated with positive speeds in our framework).
This is where the introduction of the canonical $UL$--decomposition of $Q$ is crucial.
It can be considered as the counterpart of how the boundary condition was handled in \cite[Lemma p.5]{Wec82}.

Before giving the proof of Theorem \ref{thm unp syst} we mention that we can add any diagonal matrix to the system \eqref{unp syst} without changing its controllability properties.
We use it to simplify the diagonal terms in the adjoint system, and thus the computations below (in other words, we can use the formula \eqref{formula semigp 1} and \eqref{formula semigp 2}).

\begin{proposition}\label{prop stab diag mat}
Let $\Lambda \in C^{0,1}([0,1])^{n \times n}$ satisfy \eqref{Lambda diag} and \eqref{hyp speeds} and $Q \in \R^{p\times m}$.
For every $T>0$, $(A_0,B)$ is exactly controllable in time $T$ if, and only if, $(A_{\px{\Lambda}},B)$ is exactly controllable in time $T$.
\end{proposition}

The proof of Proposition \ref{prop stab diag mat} is a simple change of variable.
It is contained in Appendix \ref{app pert diag}.

\subsection{Sufficient conditions}\label{sect suff cond}

In this part we establish the positive result, that is we assume that $\rank Q=p$ and that $T \geq \max_{i \in \ens{1,\ldots,p}} (T_{p+1}(\Lambda), T_i(\Lambda)+T_{p+c_i(Q)}(\Lambda))$ and we are going to prove that in this case $(A_0,B)$ is exactly controllable in time $T$.
Thanks to Proposition \ref{prop stab diag mat}, it is equivalent to prove the exact controllability of $(A_{\px{\Lambda}},B)$.
Now, to prove that $(A_{\px{\Lambda}},B)$ is exactly controllable in time $T$, we will use the duality and show that there exists $C>0$ such that, for every $z^1 \in L^2(0,1)^n$, we have
\begin{equation}\label{key obs}
\norm{z^1}_{L^2(0,1)^n}^2
\leq 
C\int_0^T \norm{z_-(t,1)}_{\R^m}^2 \, dt,
\end{equation}
where $z \in C^0([0,T];L^2(0,1)^n)$ is the solution to the adjoint system, i.e. $z(t)=S_{A_{\px{\Lambda}}}(T-t)^*z^1$.

In what follows, $C>0$ is a positive constant that may change from line to line but that does not depend on $z^1$.

\begin{enumerate}[1)]

\item
For $j \in \ens{1,\ldots,m}$, since in particular $T \geq T_{p+1}(\Lambda) \geq T_{p+j}(\Lambda)$, using the method of characteristics (see e.g. Figure \ref{figureA} or \eqref{formula semigp 2} with $z^0=z^1$ and $T-t$ in place of $t$), we have
\begin{equation}\label{obs ineq zm}
\norm{z^1_{p+j}}_{L^2(0,1)}^2 \leq C\int_{T-T_{p+j}(\Lambda)}^T \abs{z_{p+j}(t,1)}^2 \, dt.
\end{equation}
These terms are good because it concerns $z_-(t,1)$ (see \eqref{key obs}).
Similarly, for $i \in \ens{1,\ldots,p}$, since $T \geq T_i(\Lambda)$, we have (see e.g. Figure \ref{figureB} or \eqref{formula semigp 1})
\begin{equation}\label{obs ineq zp}
\norm{z^1_i}_{L^2(0,1)}^2 \leq C\int_{T-T_i(\Lambda)}^T \abs{z_i(t,0)}^2 \, dt.
\end{equation}
These terms are not good because it concerns $z_+(t,0)$.
We would like to get ride of it.
The only information that we know about $z_+(t,0)$ is through the boundary condition
\begin{equation}\label{BC equ}
z_-(t,0)=R^*z_+(t,0).
\end{equation}
Since $\rank Q=p$ we also have $\rank R=p$.
Therefore, $R^* \in \R^{m \times p}$ has at least one left-inverse and we can express $z_+(t,0)$ in function of $z_-(t,0)$.
However, we do not really want to completely inverse this relation without looking more closely at it as it will eventually lead to the observability inequality \eqref{key obs} only for times $T$ larger or equal than the time $T_p(\Lambda)+T_{p+1}(\Lambda)$, which is not the minimal one in general.

\item
This is where we use the decomposition of Proposition \ref{Gauss elim}.
According to it, there exist a canonical form $Q^0 \in \R^{p \times m}$ and a lower triangular matrix $L \in \R^{m \times m}$ such that
$$QL=Q^0.$$
As a result, \eqref{BC equ} implies that (we recall that $R=-\Lambda_{+}(0)Q\Lambda_{-}(0)^{-1}$)
\begin{equation}\label{BC equ bis}
(Q^0)^*\Lambda_{+}(0)z_+(t,0)
=-L^*\Lambda_{-}(0)z_-(t,0).
\end{equation}
We now look carefully at this relation row by row for the row indices $c_i(Q)$.
Let $i \in \ens{1,\ldots,p}$ be fixed.
The $c_i(Q)$-th row of \eqref{BC equ bis} is
$$
\sum_{k=1}^p q^0_{k,c_i(Q)}\lambda_k(0)z_k(t,0)=-\sum_{j=1}^m \ell_{j,c_i(Q)} \lambda_{p+j}(0) z_{p+j}(t,0).
$$
Using some of the structural properties of $Q^0$ and $L$, namely, $q^0_{k,c_i(Q)}=0$ for $k>i$ (see \eqref{cond proof direct} in Remark \ref{rem zero dessous}) and $\ell_{i,j}=0$ for $i<j$, this is equivalent to
$$
\sum_{k<i} q^0_{k,c_i(Q)}\lambda_k(0)z_k(t,0)
+q^0_{i,c_i(Q)}\lambda_i(0)z_i(t,0)
=
-\sum_{j \geq c_i(Q)} \ell_{j,c_i(Q)} \lambda_{p+j}(0) z_{p+j}(t,0).
$$
Using now the fact that $q^0_{i,c_i(Q)} \neq 0$, we obtain
\begin{equation}\label{equ zi}
z_i(t,0)=
\frac{1}{q^0_{i,c_i(Q)}\lambda_i(0)}\left(
-\sum_{k<i} q^0_{k,c_i(Q)}\lambda_k(0)z_k(t,0)
-\sum_{j \geq c_i(Q)} \ell_{j,c_i(Q)} \lambda_{p+j}(0) z_{p+j}(t,0)
\right).
\end{equation}
We recall that the goal is to estimate $z_i(t,0)$ on the time interval $(T-T_i(\Lambda),T)$ (see \eqref{obs ineq zp}).
Therefore, we estimate each term in the brackets in \eqref{equ zi} on this interval.

\item
To estimate the first term, we first observe, using the method of characteristics and the boundary condition $z_+(\cdot,1)=0$  (see Figure \ref{figureC} or \eqref{formula semigp 1}), that we have
\begin{equation}\label{BC zuno}
z_k(t,0)=0, \quad \mbox{ a.e. } t \in (0,T-T_k(\Lambda)), \quad \forall k \in \ens{1,\ldots,p},
\end{equation}
so that
$$\int_{T-T_i(\Lambda)}^T \abs{z_k(t,0)}^2 \, dt=\int_{T-T_k(\Lambda)}^T \abs{z_k(t,0)}^2 \, dt,
\quad \forall k \leq i.
$$
Therefore, for the first term on the right-hand side of \eqref{equ zi}, we have
$$
\begin{array}{rl}
\ds \int_{T-T_i(\Lambda)}^T \abs{\sum_{k<i}  q^0_{k,c_i(Q)}\lambda_k(0)z_k(t,0)}^2 \, dt
&\ds  \leq C \int_{T-T_i(\Lambda)}^T \sum_{k<i}  \abs{z_k(t,0)}^2 \, dt \\
&\ds  = C \sum_{k<i}  \int_{T-T_i(\Lambda)}^T \abs{z_k(t,0)}^2 \, dt \\
&\ds  = C \sum_{k<i}  \int_{T-T_k(\Lambda)}^T \abs{z_k(t,0)}^2 \, dt.
\end{array}
$$
The important point is that it is estimated by a similar expression to the one we want to estimate but that contains only terms for $k<i$.

\item
Let us now estimate the second term.
This is where we finally use the assumption on the time $T$.
This assumption says that $T-T_i(\Lambda) \geq T_{p+j}(\Lambda)$ for every $j \geq c_i(Q)$.
Thus,
$$
\int_{T-T_i(\Lambda)}^T \abs{z_{p+j}(t,0)}^2 \, dt
\leq \int_{T_{p+j}(\Lambda)}^T \abs{z_{p+j}(t,0)}^2 \, dt,
\quad \forall j \geq c_i(Q).
$$
On the other hand, using the method of characteristics (see Figure \ref{figureD} or \eqref{formula semigp 2}), we see that
\begin{equation}\label{estim zpj}
\int_{T_{p+j}(\Lambda)}^T \abs{z_{p+j}(t,0)}^2 \, dt
\leq C \int_0^{T-T_{p+j}(\Lambda)} \abs{z_{p+j}(t,1)}^2 \, dt,
\quad \forall j \in \ens{1,\ldots,m}.
\end{equation}
As a result,
$$\int_{T-T_i(\Lambda)}^T \abs{z_{p+j}(t,0)}^2 \, dt 
\leq C  \int_{0}^{T-T_{p+j}(\Lambda)} \abs{z_{p+j}(t,1)}^2 \, dt,
\quad \forall j \geq c_i(Q).
$$
Therefore, for the second term on the right-hand side of \eqref{equ zi}, we have
$$
\begin{array}{rl}
\ds \int_{T-T_i(\Lambda)}^T \abs{\sum_{j \geq c_i(Q)} \ell_{j,c_i(Q)} \lambda_{p+j}(0) z_{p+j}(t,0)}^2 \, dt
&\ds  \leq C \int_{T-T_i(\Lambda)}^T \sum_{j \geq c_i(Q)} \abs{z_{p+j}(t,0)}^2 \, dt \\
&\ds  = C \sum_{j \geq c_i(Q)}  \int_{T-T_i(\Lambda)}^T \abs{z_{p+j}(t,0)}^2 \, dt \\
&\ds  \leq C \sum_{j \geq c_i(Q)}  \int_{0}^{T-T_{p+j}(\Lambda)} \abs{z_{p+j}(t,1)}^2 \, dt \\
&\ds \leq C\int_0^T \norm{z_-(t,1)}_{\R^m}^2 \, dt.
\end{array}
$$

\item
To summarize, we have obtained the following estimate, valid for every $i \in \ens{1,\ldots,p}$:
$$
\int_{T-T_i(\Lambda)}^T \abs{z_i(t,0)}^2 \, dt
\leq 
C\sum_{k<i}  \int_{T-T_k(\Lambda)}^T \abs{z_k(t,0)}^2 \, dt
+ C\int_0^T \norm{z_-(t,1)}_{\R^m}^2 \, dt.
$$

By induction (starting with $i=1$) we easily deduce that, for every $i \in \ens{1,\ldots,p}$,
$$
\int_{T-T_i(\Lambda)}^T \abs{z_i(t,0)}^2 \, dt
\leq 
C\int_0^T \norm{z_-(t,1)}_{\R^m}^2 \, dt.
$$
Combined with \eqref{obs ineq zp} and \eqref{obs ineq zm} this establishes \eqref{key obs} and conclude the proof of the positive result.
\qed

\end{enumerate}

\begin{minipage}[c]{.46\linewidth}
\centering

\scalebox{0.75}{
\begin{tikzpicture}[scale=2]
\fill [gray!10]
	(0,0) rectangle (1,2.5);

\fill [white, domain=0:1, variable=\x]
	plot ({\x}, {2.5-(exp(2*\x)-1)/6})
	--(1,0)--(0,0);

\draw [<->] (0,3) node [left]  {$t$} -- (0,0) -- (1.5,0) node [below right] {$x$};
\draw (0,0) rectangle (1,2.5);
\draw (-0.1,2.5) node [left] {$T$};
\draw (0,0) node [left, below] {$0$};
\draw (1,0) node [below] {$1$};

\draw[decoration={brace, amplitude=6pt},decorate]
(0.05,2.55) -- node[above=4pt] {$z^1_{p+j}(x)$} (0.95,2.55);

\draw[decoration={brace, amplitude=6pt,mirror},decorate]
(1.05,1.49) -- node[right=4pt] {$z_{p+j}(t,1)$} (1.05,2.45);
    
\draw[ultra thick, domain=0:1.10] plot (\x, {2.5-(exp(2*\x)-1)/6});

\draw[dashed] (0,1.44) node [left] {$T-T_{p+j}(\Lambda)$} -- (1,1.44);
      \end{tikzpicture}
}


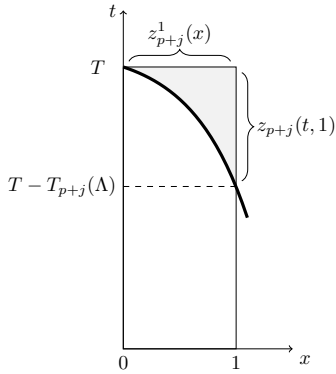
\captionof{figure}{Control of $z^1_{p+j}$ by $z_{p+j}(\cdot,1)$}\label{figureA}
\end{minipage}
\begin{minipage}[c]{.46\linewidth}
\centering
\scalebox{0.75}{
\begin{tikzpicture}[scale=2]
\fill [gray!10]
	(0,0) rectangle (1,2.5);

\fill [white, domain=0:1, variable=\x]
	plot ({\x}, {2.5+(exp(\x)-exp(1))/2})
	--(1,0)--(0,0);

\draw [<->] (0,3) node [left]  {$t$} -- (0,0) -- (1.5,0) node [below right] {$x$};
\draw (0,0) rectangle (1,2.5);
\draw (-0.1,2.5) node [left] {$T$};
\draw (0,0) node [left, below] {$0$};
\draw (1,0) node [below] {$1$};

\draw[decoration={brace, amplitude=6pt},decorate]
(0.05,2.55) -- node[above=4pt] {$z^1_i(x)$} (0.95,2.55);

\draw[decoration={brace, amplitude=6pt},decorate]
(-0.05,1.69) -- node[left=4pt] {$z_i(t,0)$} (-0.05,2.45);
    
\draw[ultra thick, domain=-0.35:1] plot (\x, {2.5+(exp(\x)-exp(1))/2});

\draw[dashed] (0,1.64) -- (1,1.64) node [right] {$T-T_i(\Lambda)$};
      \end{tikzpicture}
      }

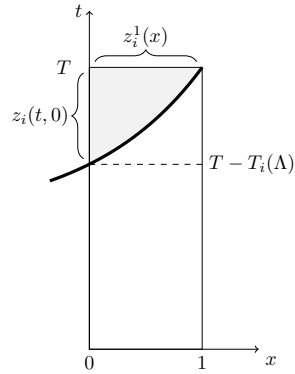
\captionof{figure}{Control of $z^1_i$ by $z_i(\cdot,0)$}\label{figureB}
\end{minipage}

\begin{minipage}[c]{.46\linewidth}
\centering

\scalebox{0.75}{
\begin{tikzpicture}[scale=2]

\fill[gray!10, domain=0:1, variable=\x]
	plot ({\x}, {2.5+(exp(\x)-exp(1))/3})
	--(1,0)--(0,0);
\fill[white, domain=0:1, variable=\x]
	plot ({\x}, {(exp(\x)-1)/3})
	--(1,0)--(0,0);

\draw [<->] (0,3) node [left]  {$t$} -- (0,0) -- (1.5,0) node [below right] {$x$};
\draw (0,0) rectangle (1,2.5);
\draw (-0.1,2.5) node [left] {$T$};
\draw (0,0) node [left, below] {$0$};
\draw (1,0) node [below] {$1$};

\draw[ultra thick, domain=0:1.25] plot (\x, {2.5+(exp(\x)-exp(1))/3});
\draw[ultra thick, domain=0:1.25] plot (\x, {(exp(\x)-1)/3});

\draw (-0.1,1.93) node [left] {$T-T_k(\Lambda)$};
\draw[decoration={brace, amplitude=6pt},decorate]
(-0.05,0.05) -- node[left=4pt] {$z_k(t,0)$} (-0.05,1.88);

\draw[decoration={brace, amplitude=6pt,mirror},decorate]
(1.05,0.67) -- node[right=4pt] {$z_k(t,1)$} (1.05,2.45);

      \end{tikzpicture}
}

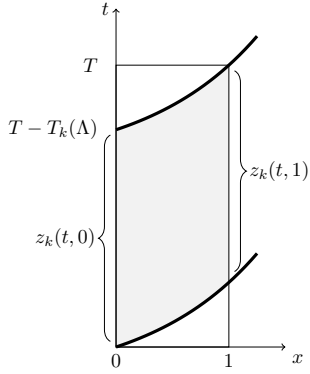
\captionof{figure}{Control of $z_k(\cdot,0)$ by $z_k(\cdot,1)$}\label{figureC}

\end{minipage}
\begin{minipage}[c]{.46\linewidth}
\centering

\scalebox{0.75}{
\begin{tikzpicture}[scale=2]

\fill[gray!10, domain=0:1, variable=\x]
	plot ({\x}, {2.5-(exp(2*\x)-1)/6})
	--(1,0)--(0,0);
\fill[white, domain=0:1, variable=\x]
	plot ({\x}, {(exp(2)-exp(2*\x))/6})
	--(1,0)--(0,0);

\draw [<->] (0,3) node [left]  {$t$} -- (0,0) -- (1.5,0) node [below right] {$x$};
\draw (0,0) rectangle (1,2.5);
\draw (-0.1,2.5) node [left] {$T$};
\draw (0,0) node [left, below] {$0$};
\draw (0.90,0) node [below] {$1$};

\draw[ultra thick, domain=0:1.05] plot (\x, {2.5-(exp(2*\x)-1)/6});
\draw[ultra thick, domain=0:1.05] plot (\x, {(exp(2)-exp(2*\x))/6});

\draw[decoration={brace, amplitude=4pt,mirror},decorate]
(1.05,0.05) -- node[right=4pt] {$z_{p+j}(t,1)$} (1.05,1.29);
\draw (1.1,1.44) node [right] {$T-T_{p+j}(\Lambda)$};

\draw[decoration={brace, amplitude=6pt},decorate]
(-0.05,1.11) -- node[left=4pt] {$z_{p+j}(t,0)$} (-0.05,2.45);
\draw (-0.1,1.06) node [left] {$T_{p+j}(\Lambda)$};


      \end{tikzpicture}
      }

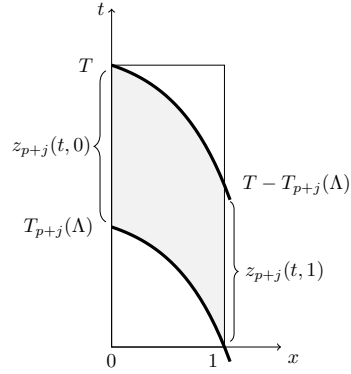
\captionof{figure}{Control of $z_{p+j}(\cdot,0)$ by $z_{p+j}(\cdot,1)$}\label{figureD}
\end{minipage}

\subsection{Necessary conditions}

We now turn out to the proof of the negative result, that is we assume that $(A_0,B)$ is exactly controllable in a time $T>0$ and we show that both conditions \ref{thm unp syst i1} and \ref{thm unp syst i2} of Theorem \ref{thm unp syst} necessary hold.
In both cases we argue by contraposition.

\begin{enumerate}[1)]
\item
First we show that, if $\rank Q<p$, then $(A_{\px{\Lambda}},B)$ is not even approximately controllable in time $T$ for any $T>0$.
To this end, we use the duality and show that there exists $z^1 \in L^2(0,1)^n$ such that
$$z_-(t,1)=0, \quad \mbox{ a.e. } t \in (0,T), \qquad z^1 \neq 0,$$
where as usual $z \in C^0([0,T];L^2(0,1)^n)$ is the solution to the adjoint system, i.e. $z(t)=S_{A_{\px{\Lambda}}}(T-t)^*z^1$.
Let then $T>0$ be fixed.
Since $\rank R=\rank Q$, by assumption, there exists $\eta \in \R^p$ such that
$$R^*\eta=0, \quad \eta \neq 0.$$
Let us then define $z^1 \in L^2(0,1)^n$ for every $x \in (0,1)$ by
$$
z^1_i(x)=
\begin{cases}
\eta_i & \mbox{ if } x \in \left(0,\phi_i^{-1}\left(T_1(\Lambda)\right)\right),\\
0 & \mbox{ otherwise, }
\end{cases}
\quad \forall i \in \ens{1,\ldots,p},
\qquad
z^1_-(x)=0.$$
Note that it is well-defined since $T_1(\Lambda)=\phi_1(1) \leq \phi_i(1)$ for every $i \in \ens{1,\ldots,p}$.
Let $z \in C^0([0,T];L^2(0,1)^n)$ be the solution to the adjoint system corresponding to this data.
Using the method of characteristics and the boundary condition $z_+(\cdot,1)=0$ (see Figure \ref{figureE} or \eqref{formula semigp 1}), we have
$$
\forall i \in \ens{1,\ldots,p}, \quad
z_i(t,0)=
\begin{cases}
\eta_i & \mbox{ if } t \in \left(T-T_1(\Lambda),T\right) \mbox{ and } t>0,\\
0 & \mbox{ otherwise, }
\end{cases}
$$
so that
\begin{equation}\label{rstar zero}
R^*z_+(t,0)=0, \quad \mbox{ a.e. } t \in (0,T).
\end{equation}
Since $z_-^1=0$ and $z_-(t,0)=R^*z_+(t,0)=0$, it follows that $z_-=0$ (see \eqref{formula semigp 2}).
In particular, $z_-(t,1)=0$ a.e. $t \in (0,T)$.
Since it is clear that $z^1 \neq 0$, this shows that $(A_{\px{\Lambda}},B)$ is not approximately controllable in time $T$ for any $T>0$ if $\rank Q<p$.

\item
Let us now prove the necessity of \ref{thm unp syst i2}.
We assume that $\rank Q=p$ but
$$T<\max_{i \in \ens{1,\ldots,p}} (T_{p+1}(\Lambda), T_i(\Lambda)+T_{p+c_i(Q)}(\Lambda)).$$
Let us show that, in this case, the system $(A_{\px{\Lambda}},B)$ is not even approximately null controllable in time $T$.
To this end, we use the duality and show that there exists $z^1 \in L^2(0,1)^n$ such that
$$z_-(t,1)=0, \quad \mbox{ a.e. } t \in (0,T), \qquad z(0,\cdot) \neq 0.$$

\item
First of all, we can always assume that
$$T \geq \max \left(T_p(\Lambda),T_{p+1}(\Lambda)\right).$$
Indeed, if $T<T_p(\Lambda)$, we define $z^1 \in L^2(0,1)^n$ for every $x \in (0,1)$ by
$$
z^1_i(x)=0, \quad \forall i \in \ens{1,\ldots,p-1},
\qquad
z^1_p(x)=
\begin{cases}
0 & \mbox{ if } x \in \left(0,\phi_p^{-1}(T)\right),\\
1 & \mbox{ if } x \in \left(\phi_p^{-1}(T),1\right),
\end{cases}
\qquad
z^1_-(x)=0.
$$
Let $z \in C^0([0,T];L^2(0,1)^n)$ be the solution to the adjoint system corresponding to this data.
The method of characteristics (see Figure \ref{figureF} or \eqref{formula semigp 1}) shows that
$$
\left\{\begin{array}{l}
\ds z_+(t,0)=0, \quad \mbox{ a.e. } t \in (0,T), \\
\ds z_p(0,x)=1, \quad \mbox{ a.e. } x \in (0,\phi_p^{-1}(T_p(\Lambda)-T)).
\end{array}\right.
$$
Since we have again \eqref{rstar zero}, we conclude as before that $z_-(t,1)=0$ a.e. $t \in (0,T)$ (but $z(0,\cdot) \neq 0$).

On the other hand, if $T <T_{p+1}(\Lambda)$, then we use $z^1 \in L^2(0,1)^n$ defined for every $x \in (0,1)$ by
$$
z^1_+(x)=0,
\qquad
z^1_{p+1}(x)=
\begin{cases}
1 & \mbox{ if } x \in \left(0,\phi_{p+1}^{-1}(T_{p+1}(\Lambda)-T)\right),\\
0 & \mbox{ if } x \in \left(\phi_{p+1}^{-1}(T_{p+1}(\Lambda)-T),1\right),
\end{cases}
\qquad
z^1_{p+j}(x)=0, \quad \forall j \geq 2.
$$
It is not difficult to see that $z_+=0$ and $z_-(\cdot,1)=0$ but $z_{p+1}(0,\cdot) \neq 0$.

\item
From now on, let $i_0 \in \ens{1,\ldots,p}$ be fixed such that
$$T_{i_0}(\Lambda)+T_{p+c_{i_0}(Q)}(\Lambda)
=\max_{i \in \ens{1,\ldots,p}} (T_i(\Lambda)+T_{p+c_i(Q)}(\Lambda)).$$
Let us now construct the final data $z^1 \in L^2(0,1)^n$ for which the controllability will fail.
We refer to Figure \ref{figureG} to clarify the geometric situation.
To explain the construction of such a data, we first observe some necessary conditions.
First of all, we point out that the three first steps in the proof of the sufficient part of Theorem \ref{thm unp syst} and the estimate \eqref{estim zpj} (Section \ref{sect suff cond}) only used the fact that $T \geq T_{p+1}(\Lambda)$ and $T \geq T_p(\Lambda)$, which can always be assumed as we have seen in the previous step.
In particular, if we aim to prove that
\begin{equation}\label{non AC}
z_-(t,1)=0, \quad \mbox{ a.e. } t \in (0,T),
\end{equation}
we see from \eqref{obs ineq zm} and \eqref{estim zpj} (see also Figures \ref{figureA} and \ref{figureD}) that it is necessary that
$$
\left\{\begin{array}{l}
\ds z^1_-(x)=0, \quad  \mbox{ a.e. } x \in (0,1), \\
\ds z_{p+j}(t,0)=0, \quad  \mbox{ a.e. } t \in (T_{p+j}(\Lambda),T).
\end{array}\right.
$$
In particular, the function $z_{p+j}(\cdot,0)$ is of the form
\begin{equation}\label{CN zpj tzero}
z_{p+j}(t,0)=
\begin{cases}
0 & \mbox{ if } t \in (T_{p+j}(\Lambda),T), \\
\alpha_{p+j}(t) & \mbox{ if } t \in (0,T_{p+j}(\Lambda)),
\end{cases}
\end{equation}
for some $\alpha_{p+j} \in L^2(0,T_{p+j}(\Lambda))$ to be determined below, which will then also define the value of $z_{p+j}(0,\cdot)$.
On the other hand, we recall from \eqref{BC zuno} that it is necessary that
$$z_i(t,0)=0, \quad \mbox{ a.e. } t \in (0,T-T_i(\Lambda)).$$
Thus, the function $z_i(\cdot,0)$ is of the form
\begin{equation}\label{CN zi tzero}
z_i(t,0)=
\begin{cases}
\alpha_i(t) & \mbox{ if } t \in (T-T_i(\Lambda),T), \\
0 & \mbox{ if } t \in (0,T-T_i(\Lambda)),
\end{cases}
\end{equation}
for some $\alpha_i \in L^2(T-T_i(\Lambda),T)$ to be determined below, which will then also define the value of $z^1_i$.

\item
Thanks to the assumption $T<T_{i_0}(\Lambda)+T_{p+c_{i_0}(Q)}(\Lambda)$, we see that the intervals $(0,T_{p+c_{i_0}(Q)}(\Lambda))$ and $(T-T_{i_0}(\Lambda),T)$ intersect each other (see e.g. Figure \ref{figureG}).
We thus propose to look for $\alpha_{p+j}$ and $\alpha_i$ as piecewise constant functions as follows:
$$
\forall k \in \ens{1,\ldots,n},
\quad
\alpha_k(t)=
\begin{cases}
\alpha_k  & \text{ if } T-T_{i_0}(\Lambda)<t<T_{p+c_{i_0}(Q)}(\Lambda),\\
0 & \text{ otherwise, }
\end{cases}
$$
for some $\alpha_k \in \R$ to be determined below such that
\begin{equation}\label{alpha pj zero}
\alpha_{p+j}=0, \quad \forall j>c_{i_0}(Q),
\end{equation}
(in order that $\alpha_{p+j}(t)=0$ if $t \in (T_{p+j}(\Lambda),T_{p+c_{i_0}(Q)}(\Lambda))$ when $j>c_{i_0}(Q)$, to be compatible with \eqref{CN zpj tzero}), and such that
\begin{equation}\label{alpha i zero}
\alpha_i=0, \quad \forall i<i_0,
\end{equation}
(in order that $\alpha_i(t)=0$ if $t \in (T-T_{i_0}(\Lambda),T-T_i(\Lambda))$ when $i<i_0$,  to be compatible with \eqref{CN zi tzero}).
In particular, the expressions \eqref{CN zpj tzero} and \eqref{CN zi tzero} now become of the same form:
$$
\forall k \in \ens{1,\ldots,n},
\quad
z_k(t,0)=
\begin{cases}
\alpha_k  & \text{ if } T-T_{i_0}(\Lambda)<t<T_{p+c_{i_0}(Q)}(\Lambda),\\
0 & \text{ otherwise. }
\end{cases}
$$
Let us denote $\alpha_+=(\alpha_1,\ldots,\alpha_p) \in \R^p$ and $\alpha_-=(\alpha_{p+1},\ldots,\alpha_n) \in \R^m$.
Since the time interval does not depend on the index $k$, it is clear that the boundary condition $z_-(t,0)=R^*z_+(t,0)$ is equivalent to
\begin{equation}\label{BC reform}
\alpha_-=R^*\alpha_+.
\end{equation}
We thus define $\alpha_-$ by this equation.
Let us now define $\alpha_+$ such that \eqref{alpha pj zero} and \eqref{alpha i zero} are satisfied.

\item
By definition \eqref{def R} of $R$ and factorization of $Q=Q^0 L^{-1}$, \eqref{BC reform} is equivalent to
$$\alpha_-=-\Lambda_{-}(0)^{-1}\left(L^*\right)^{-1}(Q^0)^*\Lambda_{+}(0)\alpha_+.$$
Let $\beta \in \R^m$ be defined by
\begin{equation}\label{def beta}
\beta=(Q^0)^*\Lambda_{+}(0)\alpha_+,
\end{equation}
so that
$$\alpha_-=-\Lambda_{-}(0)^{-1}\left(L^*\right)^{-1}\beta.$$
Since $L^*$ is an upper triangular matrix, we see that \eqref{alpha pj zero} holds if we have
\begin{equation}\label{prop beta}
\beta_j=0, \quad \forall j>c_{i_0}(Q).
\end{equation}
First of all, since $\beta_j=\sum_{k=1}^p q^0_{k,j}\lambda_k(0)\alpha_k$ and $q^0_{k,j}=0$ if $j \not\in \ens{c_1(Q),\ldots,c_p(Q)}$ (see Remark \ref{rem zero dessous}), we see that (whatever $\alpha_+$ is)
$$\beta_j=0, \quad \forall j \not\in \ens{c_1(Q),\ldots,c_p(Q)}.$$
Let us now look at the identity \eqref{def beta} for the row indices $c_i(Q)$.
Using the property \eqref{cond proof direct}, we see that
\begin{equation}\label{beta sur ci}
\beta_{c_i(Q)}=\sum_{k<i}q^0_{k,c_i(Q)}\lambda_k(0)\alpha_k+q^0_{i,c_i(Q)}\lambda_i(0)\alpha_i.
\end{equation}
Thus, we see that the following $\alpha_+ \in \R^p$ has all the desired properties:
\begin{equation}\label{def alphai}
\alpha_i=
\begin{cases}
0 & \mbox{ if } i \in \ens{1,\ldots,i_0-1}, \\
1 & \mbox{ if } i=i_0, \\
\ds \frac{-1}{q^0_{i, c_i(Q)}\lambda_i(0)}\sum_{k<i} q^0_{k, c_i(Q)}\lambda_k(0)\alpha_k & \mbox{ if } i \in \ens{i_0+1,\ldots,p}.
\end{cases}
\end{equation}
Note in addition that, using the properties of $L^*$, the property \eqref{prop beta} and \eqref{beta sur ci}, we have
\begin{equation}\label{alpha m notzero}
\begin{array}{rl}
\alpha_{p+c_{i_0}(Q)} &\ds =-\frac{1}{\lambda_{p+c_{i_0}(Q)}} \beta_{c_{i_0}(Q)} \\
&\ds =-\frac{1}{\lambda_{p+c_{i_0}(Q)}}q^0_{i_0,c_{i_0}(Q)}\lambda_{i_0}(0) \neq 0.
\end{array}
\end{equation}

\item
As a result, we define $z^1 \in L^2(0,1)^n$ for every $x \in (0,1)$ by
$$
z^1_i(x)=
\left\{\begin{array}{ll}
\alpha_i & \mbox{ if } i \in \ens{i_0,\ldots,p} \mbox{ and } \phi_i^{-1}\left(T-T_{p+c_{i_0}(Q)}(\Lambda)\right)<x<\phi_i^{-1}\left(T_{i_0}(\Lambda)\right), \\
0 & \mbox{ otherwise, }
\end{array}\right.
$$
where $\alpha_i$ is given by \eqref{def alphai}.
Note that $z^1$ is well-defined since $0 \leq T-T_{p+c_{i_0}(Q)}(\Lambda)<T_{i_0}(\Lambda)$ by assumption and since $T_{i_0}(\Lambda)=\phi_{i_0}(1) \leq \phi_i(1)$ for $i \in \ens{i_0,\ldots,p}$.
Using the method of characteristic (see \eqref{formula semigp 2}), we can also check that
$$z_{p+c_{i_0}(Q)}(0,x)=\alpha_{p+c_{i_0}(Q)}, \quad \mbox{ a.e. } x \in (\phi_{p+c_{i_0}(Q)}^{-1}(T-T_{i_0}(\Lambda)),1),$$
so that $z_{p+c_{i_0}(Q)}(0,\cdot) \neq 0$ by the computations \eqref{alpha m notzero}.
Finally, this data $z^1$ has been constructed in such a way that $z_-(\cdot,1)=0$ a.e. in $(0,T)$ but since $z(0,\cdot) \neq 0$ we see that the system $(A_{\px{\Lambda}},B)$ is not approximately null controllable in time $T$.
\qed

\end{enumerate}

\begin{minipage}[c]{.46\linewidth}
\centering

\scalebox{0.75}{
\begin{tikzpicture}[scale=2]

\fill[gray!10, domain=0:1, variable=\x]
	plot ({\x}, {2.5+(exp(\x)-exp(1))/3})
	--(1,0)--(0,0);
	
	\fill[white]
	(0,2.5) rectangle (1,3);
      
\draw [<->] (0,3) node [left]  {$t$} -- (0,0) -- (1.5,0) node [below right] {$x$};
\draw (0,0) rectangle (1,2.5);
\draw (-0.1,2.5) node [left] {$T$};
\draw (0,0) node [left, below] {$0$};
\draw (1,0) node [below] {$1$};

\draw[ultra thick, domain=0:1] plot (\x, {1.93+(exp(\x)-1)/2});
\draw[dashed, domain=0:1.15] plot (\x, {2.5+(exp(\x)-exp(1))/3});

\draw[decoration={brace, amplitude=6pt},decorate]
(0.05,2.55) -- node[above=4pt] {$z^1_+ \neq 0$} (0.75,2.55);

\draw (-0.1,1.93) node [left] {$T-T_1(\Lambda)$};

\draw[decoration={brace, amplitude=6pt},decorate]
(-0.05,0.05) -- node[left=4pt] {$R^*z_+(\cdot,0)=0$} (-0.05,2.45);


      \end{tikzpicture}
}

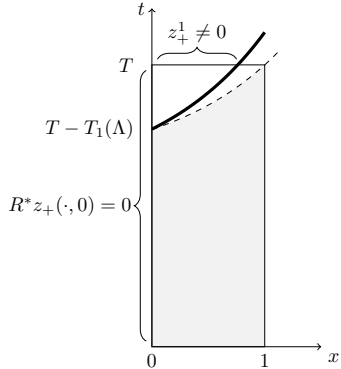
\captionof{figure}{Counterexample if $\rank Q<p$}\label{figureE}

\end{minipage}
\begin{minipage}[c]{.46\linewidth}
\centering

\scalebox{0.75}{
\begin{tikzpicture}[scale=4]

\fill[gray!10, domain=0.5:1, variable=\x]
	plot ({\x}, {1/2+(exp(\x)-exp(1))/2})
	--(1,0)--(0,0);

\draw [<->] (0,1) node [left]  {$t$} -- (0,0) -- (1.5,0) node [below right] {$x$};
\draw[dashed] (0,0) rectangle (1,0.86);
\draw (0,0) rectangle (1,0.5);
\draw (-0.1,0.86) node [left] {$T_p(\Lambda)$};
\draw (-0.1,0.50) node [left] {$T$};
\draw (0,0) node [left, below] {$0$};
\draw (1,0) node [below] {$1$};

\draw[ultra thick, domain=0:1.05] plot (\x, {(exp(\x)-1)/2});
\draw[ultra thick, domain=0.525:1.1] plot (\x, {1/2+(exp(\x)-exp(1))/2});

\draw[decoration={brace, amplitude=6pt},decorate]
(-0.05,0.05) -- node[left=4pt] {$z_+(\cdot,0)=0$} (-0.05,0.45);

\draw[decoration={brace, mirror, amplitude=4pt},decorate]
(0.05,-0.05) -- node[below=4pt] {$z_+(0,\cdot) \neq 0$} (0.475,-0.05);


      \end{tikzpicture}
      }

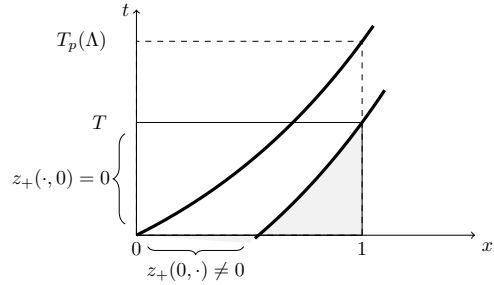
\captionof{figure}{Counterexample if $T<T_p(\Lambda)$}\label{figureF}
\end{minipage}

\begin{center}
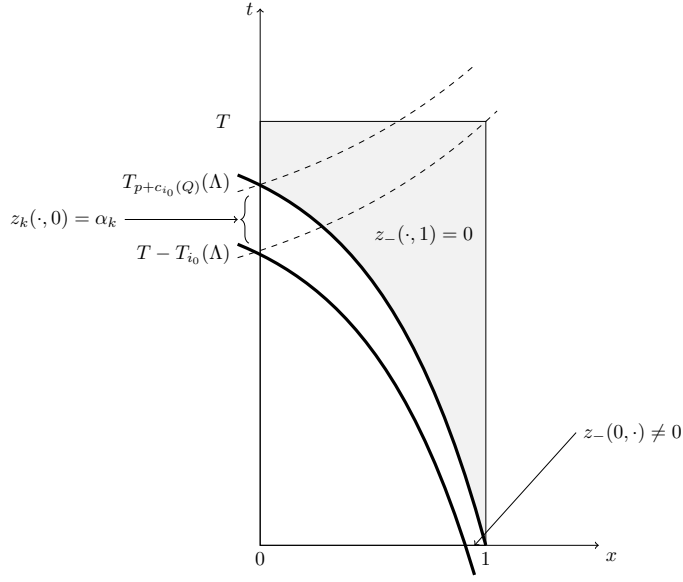

      
\scalebox{0.75}{
\begin{tikzpicture}[scale=4]

	\fill[gray!10]
	(0,0) rectangle (1,1.88);

\fill[white, domain=0:1, variable=\x]
	plot ({\x}, {(exp(2)-exp(2*\x))/4})
	--(1,0)--(0,0);

\draw [<->] (0,2.38) node [left]  {$t$} -- (0,0) -- (1.5,0) node [below right] {$x$};

\draw (0,0) rectangle (1,1.88);
\draw (-0.1,1.60) node [left] {$T_{p+c_{i_0}(Q)}(\Lambda)$};
\draw (-0.1,1.29) node [left] {$T-T_{i_0}(\Lambda)$};
\draw (-0.1,1.88) node [left] {$T$};
\draw (0,0) node [left, below] {$0$};
\draw (1,0) node [below] {$1$};

\draw (0.725,1.3) node [above] {$z_-(\cdot,1)=0$};

\draw[dashed, domain=-0.1:1.05] plot (\x, {1.88+(exp(\x)-exp(1))/3});
\draw[ultra thick, domain=-0.1:1] plot (\x, {(exp(2)-exp(2*\x))/4});

\draw[dashed, domain=-0.1:0.95] plot (\x, {1.60+(exp(\x)-1)/3});
\draw[ultra thick, domain=-0.1:0.95] plot (\x, {1.29-(exp(2*\x)-1)/4});

\draw[decoration={brace, amplitude=4pt},decorate]
(-0.05,1.34) --  (-0.05,1.55);

\draw[<-] (-0.1,1.445) -- (-0.6,1.445) node[left] {$z_k(\cdot,0)=\alpha_k$};

\draw[<-] (0.95,0)  -- (1.4,0.5) node[right] {$z_-(0,\cdot) \neq 0$};


      \end{tikzpicture}
      }
\captionof{figure}{Counterexample if $T<T_{i_0}(\Lambda)+T_{p+c_{i_0}(Q)}(\Lambda)$}\label{figureG}

\end{center}

\section{Stability of the minimal time of control}\label{sect cpct uniq}

In this section we show that the internal coupling term $M$ in \eqref{syst} has almost no impact on the exact controllability properties of \eqref{syst} and that it can be completely removed without affecting the minimal time of control.
More precisely, the goal of this section is to establish the following perturbation result:

\begin{theorem}\label{thm M to zero}
For every $\Lambda \in C^{0,1}([0,1])^{n \times n}$ that satisfies \eqref{Lambda diag}, \eqref{hyp speeds} and \eqref{hyp Rus78 eg}, $Q \in \R^{p\times m}$ and $M \in L^{\infty}(0,1)^{n \times n}$, we have
\begin{equation}\label{stab Topt}
\Topt{M}=\Topt{0}.
\end{equation}
\end{theorem}

Note that this will achieve the proof of our main result Theorem \ref{main thm}, when combined with Theorem \ref{thm unp syst} of the previous section.

\begin{remark}\label{rem NC pas stab}
Let us mention again that such a perturbation result is in general not true for the null controllability property (if $\rank Q<p$).
This is easily seen using the simple $2 \times 2$ system (3.40) in \cite[pp. 657-658]{Rus78}, namely:
\begin{equation}\label{syst CE NC}
\left\{\begin{array}{l}
\ds \pt{y_1}(t,x)=-\px{y_1}(t,x)-\epsilon y_2(t,x), \\
\ds \pt{y_2}(t,x)=\px{y_2}(t,x), \\
y_1(t,0)=0, \quad y_2(t,1)=u(t),  \\
y_1(0,x)=y^0_1(x), \quad y_2(0,x)=y^0_2(x),
\end{array}\right.
\quad t \in (0,+\infty), x \in (0,1),
\end{equation}
with $\epsilon \in \R$.
Note that $Q=0$ in this example.
By explicit computations it can be checked that:
\begin{itemize}
\item
If $\epsilon=0$, then \eqref{syst CE NC} is null controllable in time $T$ if, and only if, $T \geq 1$.
\item
If $\epsilon \neq 0$, then \eqref{syst CE NC} is null controllable in time $T$ if, and only if, $T \geq 2$.
\end{itemize}
\end{remark}

\subsection{Idea of the proof and preliminary results}

The key point in the proof of Theorem \ref{thm M to zero} is to show that the difference between the input maps of two systems (not exactly $(A_0,B)$ and $(A_M,B)$, but some perturbations of them) is a compact operator.
Indeed, the conclusion will then follow from the following general abstract result:

\begin{theorem}\label{cor DO18}
Let $H$ and $U$ be two complex Hilbert spaces.
Let $A_1:\dom{A_1} \subset H \longrightarrow H$ be the generator of a $C_0$-semigroup on $H$ and let $B \in \lin{U,\dom{A_1^*}'}$ be admissible for $A_1$.
Let $P \in \lin{H}$ be a bounded operator and let us form the unbounded operator $A_2=A_1+P$ with $\dom{A_2}=\dom{A_1}$.
For $i=1,2$, let $\Phi_i(T) \in \lin{L^2(0,+\infty;U),H}$ be the input map of $(A_i,B)$ at time $T \geq 0$, and let
$$\Toptlib{A_i,B}=\inf\ens{T>0, \quad (A_i,B) \mbox{ is exactly controllable in time } T} \in [0,+\infty].$$
We assume that:
\begin{enumerate}[(i)]
\item\label{FHtest}
For $i=1,2$, $(A_i,B)$ satisfies the Fattorini-Hautus test, i.e.
\begin{equation}\label{FH test}
\ker(\lambda-A_i^*) \cap \ker B^*=\ens{0}, \quad \forall \lambda \in \C.
\end{equation}

\item\label{impmap cpct}
$\Phi_1(T)^*-\Phi_2(T)^*$ is compact for every $T>0$.
\end{enumerate}
Then, we have $\Toptlib{A_2,B}=\Toptlib{A_1,B}$.
\end{theorem}

This general result was already noticed in \cite[Remarks 2.4 and 1.5]{DO18} and similar ideas have also been used earlier in \cite[p. 657, p. 659]{Rus78} (with a stronger assumption than \ref{FHtest} though, see below).
The proof of Theorem \ref{cor DO18} is a simple application of the compactness-uniqueness result \cite[Theorem 4.1]{DO18}, it is detailed at the beginning of Appendix \ref{app impmap cpct} for the sake of completeness.

Let us now point out that concerning our system \eqref{syst} it is actually claimed (without proof) in \cite[p. 657]{Rus78} that ``A somewhat involved, but not conceptually difficult, argument allows one to see that the operator differences $S^*-S_d^*$, $C^*-C_d^*$ are both compact.'' (see also \cite[p. 659]{Rus78}), where $C^*-C_d^*$ corresponds to $\Phi_{M}(T)^*-\Phi_{M_d}(T)^*$ in our notation, where $M_d$ denotes the diagonal part of $M$ (strictly speaking it is only almost true, since we recall that a different boundary condition at $x=1$ is considered in \cite{Rus78}).
However, it appears to us that the proof of this claim is not straightforward at all, in particular because the solution to the adjoint system of $(A_M,B)$ is not explicit if $M$ has no particular structure.
We also think that it deserves more than these three lines since it is in fact the key point to transfer the controllability properties of one system onto another, thanks to Theorem \ref{cor DO18}.
The main goal of Section \ref{sect cpct uniq} is thus to provide a complete proof of this fact.
As already mentioned, once this is done, Theorem \ref{thm M to zero} will be an immediate consequence of Theorem \ref{cor DO18}, because the assumption \ref{FHtest} will be easily checked in our case.

We would also like to emphasize that, even though the fact that the difference between the input maps is compact have been suggested in \cite{Rus78}, Theorem \ref{thm M to zero} could not have been obtained with the techniques in \cite{Rus78}.
The reason is that the author, interested in keeping the exact same time of control for the perturbed system, used a different (in some sense, weaker) version of the compactness-uniqueness result Theorem \ref{cor DO18}.
Namely, the author used the equivalence between exact and approximate controllability for such systems.
The conclusion is slightly stronger than in Theorem \ref{cor DO18} since one obtains the exact controllability in the same time for the perturbed system but the assumption is also harder to check since proving the approximate controllability of the system \eqref{syst} with a general $M$ does not seem a much easier task.

Now, in order to check that the difference between the input maps of two systems is compact, we developed the following practical sufficient condition involving only the unperturbed system:

\begin{lemma}\label{lem small times}
Under the framework of Theorem \ref{cor DO18} (we do not assume \ref{FHtest} and \ref{impmap cpct} here though), we assume that:
\begin{enumerate}[label={(\roman*)$'$}]
\setcounter{enumi}{1}
\item\label{hyp ii prime}
There exist $\epsilon>0$, a Hilbert space $\widehat{H}$, a function $G \in L^2(0,\epsilon;\lin{H,\widehat{H}})$ with $G(t)$ compact for a.e. $t \in (0,\epsilon)$ and $C>0$ such that, for a.e. $t \in (0,\epsilon)$,
$$\norm{B^*V\tilde{z}(t)}_U+\norm{V\tilde{z}(t)}_H \leq C \norm{G(t)z^0}_{\widehat{H}}, \quad \forall z^0 \in \dom{A_1^*},$$
where $V\tilde{z}(t)=\int_0^t K(t,s)\tilde{z}(s) \, ds$ is the Volterra operator with kernel $K(t,s)=S_{A_1}(t-s)^*P^*$ and $\tilde{z}(t)=S_{A_1}(t)^*z^0$.
\end{enumerate}
Then, the assumption \ref{impmap cpct} of Theorem \ref{cor DO18} holds.
\end{lemma}

The proof of Lemma \ref{lem small times} is postponed to Appendix \ref{app impmap cpct} for the sake of the presentation.
It relies on some ideas of \cite{NRL86} and an estimate that can be found in \cite{DO18}.

\begin{remark}
It is crucial to observe that the assumption \ref{hyp ii prime} in Lemma \ref{lem small times} only concerns the semigroup of the unperturbed system $(A_1,B)$.
This is what makes this result usable in practice.
Note as well that this assumption has to be checked only for small times, which makes the computation easier in our case.
Finally, let us also mention that another more general condition than \ref{hyp ii prime} can be found in Proposition \ref{thm DO18++} below.
\end{remark}

Roughly speaking, the proof of Theorem \ref{thm M to zero} will then be reduced to check the assumption \ref{hyp ii prime} of Lemma \ref{lem small times}.
We will see in the next section that the computation of $V\tilde{z}(t)$ will reveal some integral operators of a particular form, for which we will need the following technical result to conclude (see also \cite[Lemma 4]{NRL86}):

\begin{lemma}\label{lem cpct op}
Let $\Omega \subset \R^2$ be the bounded open subset defined by
$$\Omega=\ens{(s,x) \in \R^2, \quad x \in (0,1), \quad a(x)<b(x), \quad s \in (a(x),b(x))},$$
for some functions $a,b \in C^{0,1}([0,1])$.
We assume that $\Omega \neq \emptyset$.
Let $\beta \in C^1(\clos{\Omega})$ with $\beta(\Omega)\subset (0,1)$ and
$$
\partials{\beta}(s,x) \neq 0, \quad \forall (s,x) \in \clos{\Omega}.
$$
Denoting the inverse of the map $s \mapsto \beta(s,x)$ by $\beta^{-1}(\cdot,x)$, we also assume that $x \mapsto \pxi{\beta^{-1}}(\xi,x)$ does not depend on $x$.
For every $x \in [0,1]$, let $J(x)\subset \R$ be the bounded open subset defined by
$$J(x)=
\left\{\begin{array}{cl}
\ens{s \in (a(x),b(x)), \quad f_1\left(\beta(s,x)\right)<f_2(x)} & \mbox{ if } a(x)<b(x), \\
\emptyset & \mbox{ otherwise,}
\end{array}\right.
$$
for some $f_1,f_2 \in C^{1,1}([0,1])$ with $\pxi{f_1}>0$ in $[0,1]$ or $\pxi{f_1}<0$ in $[0,1]$.
Let then $\Omega' \subset \Omega$ be the bounded open subset defined by
$$\Omega'=\ens{(s,x) \in \R^2, \quad x \in (0,1), \quad s \in J(x)}.$$
Let $\alpha \in C^1(\clos{\Omega'})$ with $\alpha(\Omega') \subset (0,1)$ and
\begin{equation}\label{ab nonzero}
\partials{\alpha}(s,x) \neq 0, \quad \forall (s,x) \in \clos{\Omega'}.
\end{equation}
Finally, let $k \in L^{\infty}(0,1)$.

Then, for every $f \in L^2(0,1)$ and $x \in [0,1]$, the function $s \mapsto k(\beta(s,x))f(\alpha(s,x))$ belongs to $L^1(J(x))$ with the estimate
\begin{equation}\label{estim cpct op}
\abs{\int_{J(x)} k(\beta(s,x))f(\alpha(s,x)) \, ds} \leq \frac{\norm{k}_{L^{\infty}(0,1)}}{\inf_{s \in J(x)} \abs{\partials{\alpha}(s,x)}}\norm{f}_{L^2(0,1)}.
\end{equation}
Moreover, the linear operator defined for every $f \in L^2(0,1)$ and $x \in [0,1]$ by
\begin{equation}\label{int op}
Kf(x)=\int_{J(x)} k(\beta(s,x))f(\alpha(s,x)) \, ds,
\end{equation}
has the following properties:
\begin{enumerate}[(i)]
\item\label{conclu1}
$K\left(L^2(0,1)\right) \subset L^2(0,1)$ and the operator $f \in L^2(0,1) \mapsto Kf \in L^2(0,1)$ is compact.

\item\label{conclu2}
$K\left(H^1(0,1)\right) \subset H^1(0,1)$ and, for every $f \in H^1(0,1)$ and $x \in [0,1]$, the trace of $Kf$ at $x$ is equal to $Kf(x)$.

\item\label{conclu3}
For every $x \in [0,1]$, the operator $f \in L^2(0,1) \mapsto Kf(x) \in \R$ is compact.
\end{enumerate}

\end{lemma}

\begin{proof}
~
\begin{enumerate}[1)]
\item\label{proof kL2}
By assumption \eqref{ab nonzero}, the function $s \in J(x) \mapsto \alpha(s,x)$ is a $C^1$-diffeomorphism for every $x \in [0,1]$ such that $a(x)<b(x)$.
Its inverse will be denoted by $\alpha^{-1}(\cdot,x)$.
Using the change of variable $s \mapsto \alpha(s,x)$ we see that the function $s \mapsto k(\beta(s,x))f(\alpha(s,x))$ belongs to $L^1(J(x))$ and
\begin{equation}\label{op K good form}
Kf(x)=\int_0^1 h(\xi,x)f(\xi) \, d\xi,
\qquad
h(\xi,x)=
\left\{\begin{array}{cl}
\ds \frac{k\left(\beta\left(\alpha^{-1}\left(\xi,x\right),x\right)\right)}{\abs{\partials{\alpha}\left(\alpha^{-1}(\xi,x),x\right)}}
\indic_{\alpha\left(J(x),x\right)}(\xi) & \mbox{ if } a(x)<b(x), \\
\ds 0 & \mbox{ otherwise.}
\end{array}\right.
\end{equation}
The Cauchy-Schwarz inequality immediately gives the estimate \eqref{estim cpct op}.
Since the kernel $h \in L^{\infty}((0,1) \times(0,1))$, it is well-known that the operators of the form \eqref{op K good form} are compact, so that \ref{conclu1} holds.

\item
For the proof of item \ref{conclu2} we assume for instance that we are in the case $\partials{\beta}>0$ in $\clos{\Omega}$ and $\pxi{f_1}>0$ in $[0,1]$.
Using then the change of variable $s \mapsto \beta(s,x)$ when $a(x)<b(x)$ shows that
$$
Kf(x)=
\int_{\beta\left(a(x),x\right)}^{c(x)}
k(\xi)
f\left(\alpha\left(\beta^{-1}\left(\xi,x\right),x\right)\right)
\pxi{\beta^{-1}}(\xi,x)
\, d\xi,
$$
where
$$
c(x)=
\left\{\begin{array}{cl}
\beta\left(b(x),x\right) & \mbox{ if } a(x)<b(x) \mbox{ and } f_1(\beta(b(x),x))<f_2(x), \\
\ds f_1^{-1}\left(f_2(x)\right) & \mbox{ if } a(x)<b(x) \mbox{ and } f_1(\beta(a(x),x)) \leq f_2(x) \leq f_1(\beta(b(x),x)), \\
\ds \beta(a(x),x) & \mbox{ otherwise.}
\end{array}\right.
$$

Thanks to our regularity assumptions, we see that, when $f \in H^1(0,1)$, $Kf$ is continuous on $[0,1]$ and piecewise $H^1(0,1)$, which yields $Kf \in H^1(0,1)$ with trace at $x \in [0,1]$ equal to $Kf(x)$.

\item
Finally, the compactness of $f \in L^2(0,1) \mapsto Kf(x)$ is immediate since this operator is bounded by the estimate \eqref{estim cpct op} and its range is a finite-dimensional space.
\end{enumerate}
\end{proof}

We conclude this section with the statement of a last lemma.
We will see during the proof of Theorem \ref{thm M to zero} below that it is crucial to have only integral terms on subsets of the form $J(x)$ satisfying the assumptions of the previous lemma.
Since these subsets do not in general agree with $(0,1)$, we may have other undesirable integral terms.
The goal of the next lemma is to show that we can remove these possible other ``bad'' integral terms if we assume \eqref{hyp Rus78 eg}, which is the main purpose of this assumption.

\begin{lemma}\label{prop hyp Rus78}
For every $i \in \ens{1,\ldots,n}$, let
$$E_i=\ens{j \in \ens{1,\ldots,n} \, \middle| \, \quad \exists x \in [0,1], \quad \lambda_j(x)=\lambda_i(x)}.$$
Assume that \eqref{hyp Rus78 eg} holds, i.e.
$$E_i=\ens{j \in \ens{1,\ldots,n} \, \middle| \, \quad \lambda_j(x)=\lambda_i(x), \quad \forall x \in [0,1]}.$$
Then, for every $M \in L^{\infty}(0,1)^{n \times n}$, there exists $\widetilde{M} \in L^{\infty}(0,1)^{n \times n}$ such that the following two properties hold:
\begin{enumerate}[(i)]
\item\label{equiv mt}
For every $T>0$, $(A_{\widetilde{M}},B)$ is exactly controllable in time $T$ if, and only if, $(A_M,B)$ is exactly controllable in time $T$.

\item\label{mt zero}
For every $i \in \ens{1,\ldots,n}$ and every $j \in E_i$, we have
$$\widetilde{m}_{i,j}(x)=\delta_{i,j}\px{\lambda_i}(x), \quad \mbox{ a.e. } x \in (0,1),$$
where $\delta_{i,j}$ denotes the Kronecker delta, i.e. $\delta_{i,j}=1$ if $i=j$ and $\delta_{i,j}=0$ otherwise.
\end{enumerate}
\end{lemma}

In fact, we can prescribe any $L^\infty$ function on the diagonal of $\widetilde{M}$, we chose $\px{\lambda_i}$ only for later computational purposes.
The proof of Lemma \ref{prop hyp Rus78} is technical and it is postponed to Appendix \ref{app pert diag} for the sake of clarity (see also \cite[Remark 6]{HDMVK16} for the constant case).
It is essentially an appropriate change of variable.


\begin{remark}
Let us mention that it is assumed in \cite[p. 322]{NRL86} that
\begin{equation}\label{hyp NRL86}
\forall i,j \in \ens{1,\ldots,n}, \quad i \neq j, \quad \left(\exists x \in [0,1], \quad \lambda_i(x)=\lambda_j(x)\right) \Longrightarrow \left(m_{i,j}=0 \mbox{ in } (0,1)\right).
\end{equation}
Therefore, Lemma \ref{prop hyp Rus78} shows that the assumption \eqref{hyp Rus78 eg} of \cite{Rus78} is stronger than the assumption \eqref{hyp NRL86} of \cite{NRL86}.
All the results of the present article remain valid if \eqref{hyp Rus78 eg} is replaced by \eqref{hyp NRL86} (and \eqref{hyp speeds} can also be replaced by \eqref{hyp speeds bis}, as long as we assume \eqref{order times}).
We chose to work under the assumptions of \cite{Rus78} simply because they are more standard.
Finally, we mention that our counterexample in Section \ref{app counterexample} below also shows that the time $T_p(\Lambda)+T_{p+1}(\Lambda)$ may not be improved if \eqref{hyp NRL86} fails.
\end{remark}

\subsection{Proof of Theorem \ref{thm M to zero}}

The main steps of the proof of Theorem \ref{thm M to zero} have been explained in the previous section.
Let us now go into the details.

\begin{enumerate}[1)]
\item
Let $M \in L^{\infty}(0,1)^{n \times n}$ be fixed and let $\widetilde{M} \in L^{\infty}(0,1)^{n \times n}$ be the corresponding matrix provided by Lemma \ref{prop hyp Rus78}.
The idea is to apply Theorem \ref{cor DO18} with
$$A_1=A_{\px{\Lambda}},
\quad
A_2=A_{\widetilde{M}},
\quad
P=\widetilde{M}-\px{\Lambda}.$$
Once the assumptions of this theorem will be checked, we will obtain
$$\Topt{\widetilde{M}}=\Topt{\px{\Lambda}}.$$
The desired identity \eqref{stab Topt} will then follows from item \ref{equiv mt} of Lemma \ref{prop hyp Rus78} and Proposition \ref{prop stab diag mat}.

\item
First of all, we have to check that $(A_{\px{\Lambda}},B)$ and $(A_{\widetilde{M}},B)$ satisfy the Fattorini-Hautus test.
This is an easy step.
In fact, let us show that $(A_M,B)$ satisfies the Fattorini-Hautus test for every $M \in L^{\infty}(0,1)^{n \times n}$.
Let $\lambda \in \C$ and $z \in \dom{A_M^*}$ be such that $A_M^*z=\lambda z$ and $B^*z=0$.
Thus, $z \in H^1(0,1)^n$ solves the system of O.D.E.
$$
\left\{\begin{array}{l}
\ds \px{z}(x)=-\Lambda(x)^{-1}\left(\lambda \Id_{\R^{n \times n}} +\px{\Lambda}(x)-M(x)^*\right)z(x), \quad x \in (0,1), \\
z(1)=0,
\end{array}\right.
$$
so that $z=0$ by uniqueness.

\item
We now turn out to the proof of the second condition \ref{impmap cpct} in Theorem \ref{cor DO18}.
We recall that it is enough to check the assumption \ref{hyp ii prime} of Lemma \ref{lem small times}.
In our case, we will do it for
$$\epsilon=\phi_1(1),$$
so that the expression \eqref{formula semigp 2} of the unperturbed semigroup has only two possibilities when $t \in (0,\epsilon)$, which will make the computations below easier.
In order to check this condition \ref{hyp ii prime}, we will show that $(V\tilde{z}(t))_i(x)$ is in fact a sum of integral terms of the form \eqref{int op}, with the corresponding assumptions of Lemma \ref{lem cpct op} being satisfied.
The conclusion will then follow from this lemma (see below).

First of all, we recall that, for every $t \geq 0$ and $f \in L^2(0,t;L^2(0,1)^n)$, we have the identity
$$\left(\int_0^t f(s) \, ds\right)_i(x)=\int_0^t f_i(s,x) \, ds, \quad \mbox{ a.e. } x \in (0,1), \quad \forall i \in \ens{1,\ldots,n}.$$
This can be seen using for instance the property $L_{i,\varphi}\left(\int_0^t f(s) \, ds\right)=\int_0^t L_{i,\varphi}(f(s)) \, ds$ with the continuous linear forms $L_{i,\varphi} g=\ps{g_i}{\varphi}{L^2(0,1)}$, where $\varphi \in L^2(0,1)$, and Fubini's theorem.
Therefore, for every $i \in \ens{1,\ldots,n}$, we can write
$$\left(V\tilde{z}(t)\right)_i(x)
=\int_0^t \left(S_{A_{\px{\Lambda}}}(t-s)^*P^*\tilde{z}(s)\right)_i(x) \, ds.
$$

\item
We first perform the computations for $i \in \ens{1,\ldots,p}$.
From the expression \eqref{formula semigp 1 gen} of the semigroup, we have
$$
\left(V\tilde{z}(t)\right)_i(x)=
\int_{J_i^-(t,x)} \left(P^*\tilde{z}(s)\right)_i\left(\caractt_i(0;t-s,x)\right) \, ds,
$$
where $J_i^-(t,x)$ is open set defined for every $t \geq 0$ and $x \in [0,1]$ by
$$J_i^-(t,x)=\ens{s \in (0,t), \quad \ssint_i(t-s,x)<0}.$$

On the other hand, denoting the entries of $\widetilde{M}^*-\px{\Lambda}$ by $\left(p^*_{i,j}\right)_{1 \leq i,j \leq n}$, we have,
$$
\left(P^*\tilde{z}(s)\right)_i\left(\caractt_i(0;t-s,x)\right)
=
\sum_{k=1}^n p^*_{i,k} \left(\caractt_i(0;t-s,x)\right)
\tilde{z}_k\left(s,\caractt_i(0;t-s,x)\right).
$$
As a result, combining both expressions yields
$$
\left(V\tilde{z}(t)\right)_i(x)=
 \int_{J_i^-(t,x)} 
 \sum_{k=1}^n p^*_{i,k}\left(\caractt_i(0;t-s,x)\right)
\tilde{z}_k\left(s,\caractt_i(0;t-s,x)\right)
  \, ds.
$$

Let us now recall that $\widetilde{M}$ has been constructed in such a way that (see item \ref{mt zero} of Lemma \ref{prop hyp Rus78})
$$p^*_{i,k}(\xi)=0, \quad \mbox{ a.e. } \xi \in (0,1), \quad \forall k \in E_i,$$
so that
$$
\left(V\tilde{z}(t)\right)_i(x)=
\sum_{\substack{k=1 \\ k \not\in E_i}}^n
\int_{J_i^-(t,x)} 
  p^*_{i,k}\left(\caractt_i(0;t-s,x)\right)
\tilde{z}_k\left(s,\caractt_i(0;t-s,x)\right)
  \, ds.
$$

We split the sum into two sums, according to whether $k \in \ens{1,\ldots,p}$ or $k \in \ens{p+1,\ldots,n}$: $\left(V\tilde{z}(t)\right)_i(x)=(V\tilde{z}(t))_{i,\leq p}(x)+(V\tilde{z}(t))_{i,>p}(x)$ with
$$
(V\tilde{z}(t))_{i,\leq p}(x)=
\sum_{\substack{k=1 \\ k \not\in E_i}}^p 
\int_{J_i^-(t,x)} 
p^*_{i,k}\left(\caractt_i(0;t-s,x)\right)
\tilde{z}_k\left(s,\caractt_i(0;t-s,x)\right) \, ds,
$$
and
$$
(V\tilde{z}(t))_{i,>p}(x)=
\sum_{k=p+1}^n
\int_{J_i^-(t,x)} 
p^*_{i,k}\left(\caractt_i(0;t-s,x)\right)
\tilde{z}_k\left(s,\caractt_i(0;t-s,x)\right) \, ds.
$$

Let us deal with the first sum $(V\tilde{z}(t))_{i,\leq p}(x)$.
Thanks to the semigroup formula \eqref{formula semigp 1 gen}, we have
$$
(V\tilde{z}(t))_{i,\leq p}(x)=
\sum_{\substack{k=1 \\ k \not\in E_i}}^p 
\int_{J^{--}_{i,k}(t,x)} p^*_{i,k}\left(\caractt_i(0;t-s,x)\right)
z^0_k\left(\caractt_k\left(0;s,\caractt_i(0;t-s,x)\right)\right)\, ds,
$$
where $J^{--}_{i,k}(t,x) \subset J_i^-(t,x)$ is open set defined by
$$J^{--}_{i,k}(t,x)=\ens{s \in J_i^-(t,x), \quad \ssint_k\left(s,\caractt_i(0;t-s,x)\right)<0}, \quad \forall k \in \ens{1,\ldots,p}.$$

Let us now deal with the second sum $(V\tilde{z}(t))_{i,>p}(x)$.
Thanks to the semigroup formula \eqref{formula semigp 2 gen} and \eqref{formula semigp 1 gen} (here we use the fact that $t<\phi_1(1)$), we have
\begin{multline*}
(V\tilde{z}(t))_{i,>p}(x)=
\sum_{k=p+1}^n
\int_{J^{--}_{i,k}(t,x)} p^*_{i,k}\left(\caractt_i(0;t-s,x)\right)
z^0_k\left(\caractt_k\left(0;s,\caractt_i(0;t-s,x)\right)\right)\, ds,
\\
+
\sum_{k=p+1}^n
\int_{J^{-+}_{i,k}(t,x)} p^*_{i,k}\left(\caractt_i(0;t-s,x)\right)
\sum_{\ell=1}^p
r_{\ell,k}
z^0_\ell\left(\caractt_\ell\left(0;\ssint_k\left(s,\caractt_i(0;t-s,x)\right),0\right)\right)\, ds,
\end{multline*}
where $J^{--}_{i,k}(t,x),J^{-+}_{i,k}(t,x) \subset J_i^-(t,x)$ are the open sets defined by
$$
\begin{array}{l}
J^{--}_{i,k}(t,x)=\ens{s \in J_i^-(t,x), \quad \ssint_k\left(s,\caractt_i(0;t-s,x)\right)<0}, \quad \forall k \in \ens{p+1,\ldots,n}, \\
J^{-+}_{i,k}(t,x)=\ens{s \in J_i^-(t,x), \quad \ssint_k\left(s,\caractt_i(0;t-s,x)\right)>0}, \quad \forall k \in \ens{p+1,\ldots,n}.
\end{array}
$$

In summary, for every $i \in \ens{1,\ldots,p}$, we have
\begin{multline}\label{calcul vi}
\left(V\tilde{z}(t)\right)_i(x)=(V\tilde{z}(t))_{i,\leq p}(x)+(V\tilde{z}(t))_{i,>p}(x)
\\
=
\sum_{\substack{k=1 \\ k \not\in E_i}}^n 
\int_{J^{--}_{i,k}(t,x)} p^*_{i,k}\left(\caractt_i(0;t-s,x)\right)
z^0_k\left(\caractt_k\left(0;s,\caractt_i(0;t-s,x)\right)\right)\, ds
\\
+
\sum_{k=p+1}^n
\int_{J^{-+}_{i,k}(t,x)} p^*_{i,k}\left(\caractt_i(0;t-s,x)\right)
\sum_{\ell=1}^p
r_{\ell,k}
z^0_\ell\left(\caractt_\ell\left(0;\ssint_k\left(s,\caractt_i(0;t-s,x)\right),0\right)\right)\, ds.
\end{multline}

\item
Similar computations for $j \in \ens{1,\ldots,m}$ show that
\begin{multline}\label{calcul vpj}
\left(V\tilde{z}(t)\right)_{p+j}(x)
=
\sum_{\substack{k=1 \\ k \not\in E_{p+j}}}^n 
\int_{J^{--}_{p+j,k}(t,x)}
p^*_{p+j,k}\left(\caractt_{p+j}(0;t-s,x)\right)
z^0_k\left(\caractt_k\left(0;s,\caractt_{p+j}(0;t-s,x)\right)\right)\, ds
\\
+
\sum_{\substack{k=p+1 \\ k \not\in E_{p+j}}}^n 
\int_{J^{-+}_{p+j,k}(t,x)}
p^*_{p+j,k}\left(\caractt_{p+j}(0;t-s,x)\right)
\sum_{i=1}^p
r_{i,k}
z^0_i\left(\caractt_i\left(0;\ssint_k\left(s,\caractt_{p+j}(0;t-s,x)\right),0\right)\right)\, ds
\\
+
\sum_{i=1}^p r_{i,p+j} \sum_{\substack{k=1 \\ k \not\in E_i}}^n
\int_{J_{p+j,k,i}^{+-}(t,x)}
p^*_{i,k}\left(\caractt_i\left(0;\ssint_{p+j}(t-s,x),0\right)\right)
\\
\times
z^0_k\left(\caractt_k\left(0;s,\caractt_i\left(0;\ssint_{p+j}(t-s,x),0\right)\right)\right) \, ds
\\
+
\sum_{i=1}^p r_{i,p+j} \sum_{k=p+1}^n
\int_{J_{p+j,k,i}^{++}(t,x)}
p^*_{i,k}\left(\caractt_i\left(0;\ssint_{p+j}(t-s,x),0\right)\right)
\\
\times
\sum_{\ell=1}^p r_{\ell,k}
z^0_\ell\left(\caractt_\ell\left(0;\ssint_k\left(s,\caractt_i\left(0;\ssint_{p+j}(t-s,x),0\right)\right),0\right)\right) \, ds,
\end{multline}
where $J^{--}_{p+j,k}(t,x),J^{-+}_{p+j,k}(t,x)$ and $J^{+-}_{p+j,k,i}(t,x),J^{++}_{p+j,k,i}(t,x)$ are the open sets defined for every $t \geq 0$ and $x \in [0,1]$ by
$$
\begin{array}{l}
\ds J^{-\mp}_{p+j,k}(t,x)
=\ens{s \in J_{p+j}^-(t,x), \quad \pm\ssint_k\left(s,\caractt_{p+j}(0;t-s,x)\right)<0},
\\
\ds J^{+\mp}_{p+j,k,i}(t,x)
=\{s \in J_{p+j}^+(t,x), \quad \pm\ssint_k\left(s,\caractt_i\left(0;\ssint_{p+j}(t-s,x),0\right)\right)<0\},
\end{array}
$$
and where $J_{p+j}^-(t,x),J_{p+j}^+(t,x)$ are the open sets defined by
$$J_{p+j}^{\mp}(t,x)=\ens{s \in (0,t), \quad \pm\ssint_{p+j}(t-s,x)<0}.$$

\item
We have just seen that, for every $t \in (0,\epsilon)$, $z ^0 \in L^2(0,1)^n$, $i \in \ens{1,\ldots,n}$ and a.e. $x \in (0,1)$, $\left(V\tilde{z}(t)\right)_i(x)$ is a sum of terms of the form \eqref{int op}.
If we manage to prove that each of these terms satisfies the assumptions of Lemma \ref{lem cpct op}, then this will show that the expressions on the right-hand sides of \eqref{calcul vi} and \eqref{calcul vpj} make sense for every $x \in [0,1]$ (not only a.e.) and belong to $H^1(0,1)$ when $z^0 \in H^1(0,1)^n$, with a trace at $x=1$ equal to the same expression but with $x$ changed into $1$.
A natural candidate for the function $G$ of Lemma \ref{lem small times} will then be the function defined for every $t \in (0,\epsilon)$ and $z^0 \in L^2(0,1)^n$ by
\begin{equation}\label{def G}
G(t)z^0=\left(\left(V\tilde{z}(t)\right)_{p+1}(1),\ldots,\left(V\tilde{z}(t)\right)_n(1),\left(V\tilde{z}(t)\right)_1,\ldots,\left(V\tilde{z}(t)\right)_n\right),
\end{equation}
where $G(t)$ is considered as an operator from the space $H=L^2(0,1)^n$ onto the product space $\widehat{H}=\R^m \times L^2(0,1)^n$ and where, by abuse of notation, $\left(V\tilde{z}(t)\right)_i$ in \eqref{def G} denotes in fact the function defined for every $x \in [0,1]$ by the expression on the right-hand side of \eqref{calcul vi} (if $i \in \ens{1,\ldots,p}$) or \eqref{calcul vpj} (if $i=p+j \in \ens{p+1,\ldots,n}$).
We use a similar abuse of notation for $\left(V\tilde{z}(t)\right)_{p+j}(1)$.

\item
Let us now check that each of the integral terms in \eqref{calcul vi} and \eqref{calcul vpj} satisfies the assumptions of Lemma \ref{lem cpct op}.
We focus on the terms in $\left(V\tilde{z}(t)\right)_{p+j}(x)$ since they are the most important ones (because $\left(V\tilde{z}(t)\right)_{p+j}(1)$ appears in \eqref{def G} and since the terms in \eqref{calcul vi} can be treated similarly to the first two terms in \eqref{calcul vpj}).
Let then $j \in \ens{1,\ldots,m}$ be fixed.
For obvious reasons of presentation we will also only treat one type of integrals in $\left(V\tilde{z}(t)\right)_{p+j}(x)$.
Let us point out that the a priori extra assumptions in Lemma \ref{lem cpct op} are used to treat all the other cases.
We choose to deal with the first type of integrals in \eqref{calcul vpj}, namely,
$$
\int_{J^{--}_{p+j,k}(t,x)}
p^*_{p+j,k}\left(\caractt_{p+j}(0;t-s,x)\right)
z^0_k\left(\caractt_k\left(0;s,\caractt_{p+j}(0;t-s,x)\right)\right)\, ds
=K(t)z^0_k(x).
$$
Let then $k \in \ens{1,\ldots,n}$ with $k \not\in E_{p+j}$ be fixed.
We are in the configuration of Lemma \ref{lem cpct op} with
$$
\begin{array}{c}
J(x)=J_{p+j,k}^{--}(t,x), \quad \beta(s,x)=\caractt_{p+j}(0;t-s,x), \quad \alpha(s,x)=\caractt_k\left(0;s,\caractt_{p+j}(0;t-s,x)\right),
\\
\beta^{-1}(\xi,x)=\phi_{p+j}(\xi)-\phi_{p+j}(x)+t, \\
\Omega=\ens{(s,x) \in \R^2, \quad x \in (0,1), \quad s \in J_{p+j}^-(t,x)}, \quad a(x)=\max\left(0,t-\phi_{p+j}(x)\right), \quad b(x)=t,
\\
f_1(\xi)=
\left\{\begin{array}{ll}
\ds \phi_{p+j}(\xi)+\phi_k(\xi)-\phi_k(1), & \mbox{ if } k \leq p, \\
\ds \phi_{p+j}(\xi)-\phi_k(\xi), & \mbox{ if } k >p, \quad k \not\in E_{p+j},
\end{array}\right.
\quad
f_2(x)=\phi_{p+j}(x)-t.
\end{array}
$$
The regularities of these functions are clear.
Note that, for this case, we have $a(x)<b(x)$ for every $x \in (0,1]$ since $t>0$.
Recalling the definition \eqref{def phi} of the $\phi_k$, and thanks to \eqref{hyp speeds}, we can check that, if $k \leq p$, then $\pxi{f_1}>0$ in $[0,1]$ and, if $k>p$ with $k \not\in E_{p+j}$, then either $\pxi{f_1}>0$ in $[0,1]$ or $\pxi{f_1}<0$ in $[0,1]$.
Let us now compute the derivatives of $\beta$ and $\alpha$.
First of all, it can be checked (using for instance the explicit formula \eqref{express caract}) that, for every $i \in \ens{1,\ldots,n}$, every $t>0$ and $x \in (0,1)$ such that $\ssint_i(t,x)<0$, we have
$$
\left\{\begin{array}{l}
\ds \pt{\caractt_i}(0;t,x)=-\lambda_i\left(\caractt_i(0;t,x)\right), \\
\ds \px{\caractt_i}(0;t,x)=\lambda_i\left(\caractt_i(0;t,x)\right)\frac{1}{\lambda_i(x)}.
\end{array}\right.
$$

It follows that
\begin{equation}\label{ps term 1}
\partials{\beta}(s,x)
=-\pt{\caractt_{p+j}}\left(0;t-s,x\right)=\lambda_{p+j}\left(\caractt_{p+j}(0;t-s,x)\right),
\end{equation}
and
\begin{equation}\label{ps term 2}
\begin{array}{rl}
\ds \partials{\alpha}(s,x)
&\ds =
\pt{\caractt_k}\left(0;s,\caractt_{p+j}(0;t-s,x)\right)
-\px{\caractt_k}\left(0;s,\caractt_{p+j}(0;t-s,x)\right)\pt{\caractt_{p+j}}(0;t-s,x),
\\
&\ds =
-\lambda_k\left(\caractt_k\left(0;s,\caractt_{p+j}(0;t-s,x)\right)\right)
\left(1-\frac{\lambda_{p+j}\left(\caractt_{p+j}\left(0;t-s,x\right)\right)}{\lambda_k\left(\caractt_{p+j}\left(0;t-s,x\right)\right)}\right).
\end{array}
\end{equation}

From these computations, we see that none of these terms are equal to zero.
For the first term \eqref{ps term 1}, this follows from the basic assumption \eqref{hyp speeds}.
For the second term \eqref{ps term 2} this is where we use in a crucial way that $k \not\in E_{p+j}$.
As a result, all the assumptions of Lemma \ref{lem cpct op} are satisfied.

\item
Finally, thanks again to the fact that $k \not\in E_{p+j}$, we have
$$\min_{\xi \in [0,1]} \abs{\lambda_k(\xi)}>0, \qquad \min_{\xi \in [0,1]} \abs{1-\frac{\lambda_{p+j}(\xi)}{\lambda_k(\xi)}}>0.$$
Thus, we see from \eqref{ps term 2} that $\abs{\partials{\alpha}}$ can be estimated from below by a positive constant that does not depend on $t,s$ or $x$.
As a consequence, from the estimate \eqref{estim cpct op} of Lemma \ref{lem cpct op} we obtain that there exists $C>0$ such that
$$\abs{K(t)z^0_k(1)}+\norm{K(t)z^0_k}_{L^2(0,1)} \leq C \norm{z^0_k}_{L^2(0,1)}, \quad \forall t \in (0,\epsilon).$$
Since similar estimates hold for the other integrals and the other components, this shows that for the function $G$ defined by \eqref{def G} we also have $G \in L^{\infty}(0,\epsilon;\lin{H,\widehat{H}}) \subset L^2(0,\epsilon;\lin{H,\widehat{H}})$.
All the assumptions of Lemma \ref{lem small times} are now satisfied.
This completes the proof of Theorem \ref{thm M to zero}.
\qed

\end{enumerate}


\section*{Acknowledgements}
The authors would like to thank Yacine Mokhtari for bringing the article \cite{Wec82} to their attention after they posted a first preprint online.
This project was supported by the Natural Science Foundation of China (No. 11601284), the Young Scholars Program of Shandong University (No. 2016WLJH52) and the National Postdoctoral Program for Innovative Talents (No. BX201600096).

\appendix

\section{Canonical $UL$--decomposition}\label{appendix}

In this appendix we give a proof of Proposition \ref{Gauss elim}, which is a crucial result to define the key elements in our main result Theorem \ref{main thm}.
Let $Q \in \R^{p \times m}$ with $\rank Q=p$ be given.
We recall that want to prove that there exists a unique $Q^0 \in \R^{p \times m}$ such that the following two properties hold:
\begin{enumerate}[(i)]
\item
There exists $L \in \R^{m \times m}$ such that $QL=Q^0$ with $L$ lower triangular ($\ell_{ij}=0$ if $i<j$) and with only ones on its diagonal ($\ell_{ii}=1$ for every $i$).
\item
$Q^0$ is in canonical form (Definition \ref{def canon}).
\end{enumerate}

\begin{proof}[Proof of Proposition \ref{Gauss elim}]
~
\begin{enumerate}[1)]
\item
The existence follows from the Gaussian elimination, as shown for instance in Example \ref{ex Q}.
We briefly recall the general procedure.
Since $\rank Q=p$, the last row of $Q \in \R^{p \times m}$ cannot be zero.
Let then $c_p \in \ens{1,\ldots,m}$ be the column index of the last non-zero entry of the last row of $Q$.
We then remove the entries of $Q$ at the left of $q_{p,c_p}$.
In matricial form this means that we multiply $Q$ to the right by a lower triangular matrix with only ones on its diagonal and zero everywhere else, except for its $c_p$-row whose first $c_p-1$ entries are equal to $\frac{-q_{p,1}}{q_{p,c_p}}, \ldots, \frac{-q_{p,c_p-1}}{q_{p,c_p}}$.
We then obtain an equivalent matrix to $Q$ which has only one non zero entry on its last row.
We then forget about the last row to obtain a $(p-1) \times m$ matrix with full-row rank and we repeat the procedure ($c_{p-1}$ being the last non-zero entry of such a matrix which is not in the $c_p$ column, etc.).
It is not difficult to see that the matrix resulting from these operations is in canonical form.

\item
To show the uniqueness, we assume that there exist two canonical forms $Q^0,\widetilde{Q}^0 \in \R^{p \times m}$ and two lower triangular matrices with only ones on their diagonal $L,\widetilde{L} \in \R^{m \times m}$ such that $QL=Q^0$ and $Q\widetilde{L}=\widetilde{Q}^0$ and we prove that $Q^0=\widetilde{Q}^0$.
Denoting $L'=\widetilde{L}^{-1}L$, we have
$$Q^0=\widetilde{Q}^0L',$$
and $L'$ is a lower triangular matrix with only ones on its diagonal.
Looking at this equality column by column, we have
$$Q^0_j=\widetilde{Q}^0_j+\sum_{i=j+1}^m \ell'_{i,j}\widetilde{Q}^0_i, \quad \forall j \in \ens{1,\ldots,m}.$$
We want to prove that $Q^0_j=\widetilde{Q}^0_j$ for every $j$.
For $j=m$ it is clear.
For $j=m-1$, we have
\begin{equation}\label{uniq 1}
Q^0_{m-1}=\widetilde{Q}^0_{m-1}+\ell'_{m,m-1}\widetilde{Q}^0_m.
\end{equation}
If $\widetilde{Q}^0_m=0$ then we are done.
Assume then that $\widetilde{Q}^0_m \neq 0$.
This necessarily means that $m \in \ens{c_1(\widetilde{Q}^0),\ldots,c_p(\widetilde{Q}^0)}$ by the two last conditions in \eqref{def ci}.
Let us write $m=c_{i_m}(\widetilde{Q}^0)$.
Then, $\tilde{q}^0_{i_m,m-1}=0$ by the last condition in \eqref{def ci}.
On the other hand, since $Q^0_m=\widetilde{Q}^0_m$ by the previous step, the same considerations apply to $Q^0_m$, i.e. $m=c_{k_m}(Q^0)$ for some $k_m$.
Let us show that we necessarily have $k_m=i_m$.
If $k_m>i_m$, then $\tilde{q}^0_{k_m,m}=0$ by \eqref{cond proof direct} in Remark \ref{rem zero dessous}.
Since $q^0_{k_m,m} \neq 0$ by the first condition in \eqref{def ci}, the identity $q^0_{k_m,m}=\tilde{q}^0_{k_m,m}$ would fail.
By the same arguments, $i_m>k_m$ is not possible either.
As a result, $m=c_{i_m}(Q^0)$ and thus $q^0_{i_m,m-1}=0$ as well.
Therefore, looking at the $i_m$-th row of the equality \eqref{uniq 1}, we obtain
$$0=\ell'_{m,m-1}\tilde{q}^0_{i_m,m}.$$
Since $\tilde{q}^0_{i_m,m} \neq 0$ by the first condition in \eqref{def ci}, we obtain that $\ell'_{m,m-1}=0$.
Coming back to \eqref{uniq 1} we have established that $Q^0_{m-1}=\widetilde{Q}^0_{m-1}$.
Reasoning by induction we easily obtain that $Q^0_j=\widetilde{Q}^0_j$ for every $j$.
This completes the proof of the uniqueness part.
\end{enumerate}
\end{proof}

\section{Equality between $\Topt{M}$ and the time of \cite{Wec82}}\label{app time Weck}

In this appendix, we show that the expression of the time $T_c$ given by \eqref{Time Weck} and introduced in \cite{Wec82} for the null controllability of \eqref{syst} with diagonal $M$ coincides with the expression of the minimal time $\Topt{M}$ introduced here in \eqref{thm Topt} for the exact controllability of \eqref{syst} with arbitrary $M$.
More precisely, assuming $\rank Q=p$, we prove the equality
\begin{equation}\label{eq times}
\max_{k \in \ens{1,\ldots,p}} T_{p-k+1}(\Lambda)+T_{p+\ell(k)}(\Lambda)
=\max_{i \in \ens{1,\ldots,p}} T_i(\Lambda)+T_{p+c_i(Q)}(\Lambda).
\end{equation}
We recall that, for every $k \in \ens{1,\ldots,p}$, $\ell(k) \in \ens{1,\ldots,m}$ is the unique index such that
$$\ker C_0E^+_k=\ker E^-_1C_0E^+_k=\ldots=\ker E^-_{\ell(k)-1}C_0E^+_k\subsetneq \ker E^-_{\ell(k)}C_0E^+_k,$$
where $C_0=-\Lambda_-(0)^{-1}Q^*\Lambda_+(0)\Sigma$ and
$$
E^-_\ell=\diag(\underbrace{0,\ldots,0}_{\ell \text{ times }},1\ldots,1),
\qquad E^+_k=\diag(\underbrace{1,\ldots,1}_{k \text{ times }},0\ldots,0),
\qquad \Sigma=
\begin{pmatrix}
(0) &  & 1 \\
 & \reflectbox{$\ddots$} &  \\
1 &  & (0) \\
\end{pmatrix}.
$$

\begin{enumerate}[1)]
\item
The first step is to show that
\begin{equation}\label{what is lk}
\ell(k)=\min(c_p(Q),\ldots,c_{p-k+1}(Q)), \quad \forall k \in \ens{1,\ldots,p}.
\end{equation}
Let $k \in \ens{1,\ldots,p}$ be fixed.
By uniqueness, it is equivalent to prove the following two properties for $\ell(k)$ given by \eqref{what is lk}:
\begin{equation}\label{properties for lk}
\left\{\begin{array}{l}
\ds \ker E^-_\ell C_0E^+_k=\ker C_0E^+_k, \quad \forall \ell \in \ens{1,\ldots,\ell(k)-1}, \\
\ds \ker E^-_{\ell(k)-1} C_0E^+_k \neq \ker E^-_{\ell(k)} C_0E^+_k.
\end{array}\right.
\end{equation}
Using the canonical $UL$--decomposition $Q=Q^0L^{-1}$, a computation shows that, for every $\ell \in \ens{0,\ldots,m-1}$, we have
\begin{equation}\label{caract ker}
\ker E^-_\ell C_0E^+_k=\ens{w_+ \in \R^p, \quad \sum_{j=p-k+1}^p q^0_{j,i}\lambda_j(0)w_{p+1-j}=0, \quad \forall i \in \ens{\ell+1,\ldots,m}}.
\end{equation}
Using \eqref{cond proof direct} and reasoning by induction, we can deduce that, if $\ell \in \ens{0,\ldots,\ell(k)-1}$ (recall that $\ell(k)$ is given by \eqref{what is lk}), then
\begin{equation}\label{caract kern}
\ker E^-_\ell C_0E^+_k=\ens{w_+ \in \R^p, \quad w_i=0, \quad \forall i \in \ens{1,\ldots,k}}.
\end{equation}
This shows in particular that $\ker E^-_\ell C_0E^+_k$ does not depend on $\ell$ if $\ell \in \ens{0,\ldots,\ell(k)-1}$, so that the first property in \eqref{properties for lk} is proved.

To prove the second property in \eqref{properties for lk}, let $i_k \in \ens{p-k+1,\ldots,p}$ be such that $\ell(k)=c_{i_k}(Q)$.
Let us then construct the data $w_+ \in \R^p$ defined by
$$
w_{p+1-r}=
\begin{cases}
0 & \mbox{ if } r \in \ens{1,\ldots,i_k-1}, \\
1 & \mbox{ if } r=i_k, \\
\ds \frac{-1}{q^0_{r, c_r(Q)}\lambda_r(0)}\sum_{j=i_k}^{r-1} q^0_{j, c_r(Q)}\lambda_j(0)w_{p+1-j} & \mbox{ if } r \in \ens{i_k+1,\ldots,p}.
\end{cases}
$$
Firstly, using the characterization \eqref{caract kern}, it is clear that $w_+ \not\in \ker C_0E^+_k$ since $w_{p+1-i_k}=1$ and $p+1-i_k \in \ens{1,\ldots,k}$ by definition of $i_k$.
Let us now show that $w_+ \in \ker E^-_{\ell(k)}C_0E^+_k$.
If $\ell(k)=m$, then $E^-_{\ell(k)}=0$ and this is clear.
We thus assume that $\ell(k) \leq m-1$, and use the characterization \eqref{caract ker} to prove that $w_+ \in \ker E^-_{\ell(k)}C_0E^+_k$.
Let then $i \in \ens{\ell(k)+1,\ldots,m}$ be fixed.
Since $w_{p+1-j}=0$ if $j<i_k$ by construction and $i_k \geq p-k+1$ by definition, we have to show that
\begin{equation}\label{appB proof}
\sum_{j=i_k}^p q^0_{j,i}\lambda_j(0)w_{p+1-j}=0.
\end{equation}
Firstly, observe that this identity is clear if $i \not\in \ens{c_1(Q),\ldots,c_p(Q)}$ since $q^0_{j, i}=0$ in such a case  (see Remark \ref{rem zero dessous}).
Let us then consider $i=c_r(Q)$ for some $r \in \ens{1,\ldots,p}$ and such that $i \in \ens{\ell(k)+1,\ldots,m}$.
In particular, $r \neq i_k$.
If $r<i_k$, then \eqref{appB proof} follows from the fact that $q^0_{j,c_r(Q)}=0$ for every $j>r$ (see \eqref{cond proof direct}).
If $r>i_k$, then we can write
$$
\sum_{j=i_k}^p q^0_{j,c_r(Q)}\lambda_j(0)w_{p+1-j}=
\sum_{j=i_k}^{r} q^0_{j, c_r(Q)}\lambda_j(0)w_{p+1-j}
+\sum_{j=r+1}^p q^0_{j, c_r(Q)}\lambda_j(0)w_{p+1-j}.
$$
On the right hand side, the first sum is equal to zero by construction and the second sum is also equal to zero since $q^0_{j, c_r(Q)}=0$ for $j>r$ (see again \eqref{cond proof direct}).
This establishes \eqref{appB proof}, so that $w_+ \in \ker E^-_{\ell(k)}C_0E^+_k$.
The proof of \eqref{what is lk} is complete.

\item
Let us now see that \eqref{what is lk} implies \eqref{eq times}.
First of all, the inequality ``$\geq$'' is clear thanks to \eqref{order times} since $c_i(Q) \geq \ell(k)$ for $k=p-i+1$.
Let us then show the reversed inequality.
By induction on $k$, we show that each term $T_{p-k+1}(\Lambda)+T_{p+\ell(k)}(\Lambda)$ is less than the right hand side in \eqref{eq times}.
For $k=1$ this is clear since $\ell(1)=c_p(Q)$.
For $k=2$, we have
$$T_{p-1}(\Lambda)+T_{p+\ell(2)}(\Lambda)=T_{p-1}(\Lambda)+T_{p+\min(c_p(Q),c_{p-1}(Q))}(\Lambda).$$
If $\min(c_p(Q),c_{p-1}(Q))=c_{p-1}(Q)$, this is clear since in this case
$$T_{p-1}(\Lambda)+T_{p+\ell(2)}(\Lambda)=T_{p-1}(\Lambda)+T_{p+c_{p-1}(Q)}(\Lambda).$$
On the other hand, if $\min(c_p(Q),c_{p-1}(Q))<c_{p-1}(Q)$, then we have $\ell(2)=\ell(1)$ and, using \eqref{order times}, we obtain
$$T_{p-1}(\Lambda)+T_{p+\ell(2)}(\Lambda)=T_{p-1}(\Lambda)+T_{p+\ell(1)}(\Lambda) \leq T_p(\Lambda)+T_{p+\ell(1)}(\Lambda).$$
Since the right hand side is the term that we have estimated in the previous step $k=1$, the proof is completed for $k=2$.
Reasoning by induction we easily obtain the reversed inequality.
\qed
\end{enumerate}

\section{A counterexample when the assumption \eqref{hyp Rus78 eg} is not satisfied}\label{app counterexample}

In this appendix we construct a counterexample to the conclusion of our main result Theorem \ref{main thm} when the assumption \eqref{hyp Rus78 eg} is not satisfied.
To this end, we consider the following $4\times4$ system:
\begin{equation}\label{4 dimen sys}
 \left\{\begin{array}{l}
\ds \pt{y_1}(t,x)=-\px{y_1}(t,x)+a(x) y_2(t,x), \\
\ds \pt{y_2}(t,x)=\lambda_2(x) \px{y_2}(t,x)-a(x) y_1(t,x), \\
\ds \pt{y_3}(t,x)=\frac{1}{2}\px{y_3}(t,x), \\
\ds \pt{y_4}(t,x)=\px{y_4}(t,x),
\end{array}\right. 
\end{equation}
with boundary conditions
\begin{equation}\label{4 dimen sys BC}
 \left\{\begin{array}{l}
y_1(t,0)=y_3(t,0), \\
y_2(t,0)=y_4(t,0),
\end{array}\right. 
\quad
\left\{\begin{array}{l}
y_3(t,1)=u_1(t), \\
y_4(t,1)=u_2(t),
\end{array}\right. 
\end{equation}
where $\lambda_2 \in C^{0,1}([0,1])$ and $a \in L^{\infty}(0,1)$ are any functions such that (see also Remark \ref{rem regu} below)
\begin{equation}\label{hyp l2}
\left\{\begin{array}{l}
\ds -1 \leq \lambda_2(x)<0, \quad \forall x \in [0,1], \\
\ds \lambda_2(x)=-1, \quad \forall x \in \left[0,\frac{1}{2}\right], \\
\ds -\int_{\frac{1}{2}}^1 \frac{1}{\lambda_2(\xi)} \, d\xi=\frac{3}{2},
\end{array}\right. 
\qquad
\left\{\begin{array}{l}
\ds a(x)=0, \quad \mbox{ a.e. } x \in \left(\frac{1}{2},1\right), \\
\ds \int_0^{\frac{1}{2}} a(x) \, dx=\frac{\pi}{2}.
\end{array}\right. 
\end{equation}
Note that we are in the case $p=m=2$, the parameters $\Lambda,M$ and $Q$ are
$$\Lambda(x)=
\begin{pmatrix}
-1 & 0 & 0 & 0 \\
0 & \lambda_2(x) & 0 & 0 \\
0 & 0 & \frac{1}{2} & 0 \\
0 & 0 & 0 & 1
\end{pmatrix}
, \quad
M(x)=
\begin{pmatrix}
0 & a(x) & 0 & 0 \\
-a(x) & 0 & 0 & 0 \\
0 & 0 & 0 & 0 \\
0 & 0 & 0 & 0
\end{pmatrix}
, \quad
Q=
\begin{pmatrix}
1 & 0 \\
0 & 1
\end{pmatrix}
,
$$
and the times are
$$T_1(\Lambda)=1,\quad T_2(\Lambda)=2, \quad T_3(\Lambda)=2, \quad T_4(\Lambda)=1.$$
Clearly, the assumption \eqref{hyp Rus78 eg} is not satisfied here.
We then have the following result:
\begin{proposition}\label{prop counterex}
Let $\lambda_2 \in C^{0,1}([0,1])$ and $a \in L^{\infty}(0,1)$ satisfy \eqref{hyp l2}.
Then, the system \eqref{4 dimen sys}-\eqref{4 dimen sys BC} is exactly controllable in time $T$ if, and only if $T\geq 4$.
\end{proposition}

\begin{remark}
The time in Proposition \ref{prop counterex} is in fact the worst possible control time $T_p(\Lambda)+T_{p+1}(\Lambda)$.
Note that we are in the best possible situation for $Q$ though (see Remark \ref{rem optimality}).
Let us also recall that Theorem \ref{thm unp syst} shows that the system \eqref{4 dimen sys}-\eqref{4 dimen sys BC} with $a=0$ is exactly controllable in time $T$ if, and only if, $T\geq \max \ens{T_3(\Lambda),T_1(\Lambda)+T_3(\Lambda), T_2(\Lambda)+T_4(\Lambda)}=3$.
Thus, we see that, while assuming \eqref{hyp Rus78 eg} a bounded perturbation can not produce a system that is not exactly controllable after the time $\Topt{0}+\epsilon$, whatever how small $\epsilon>0$ is (by Theorem \ref{thm M to zero}), Proposition \ref{prop counterex} shows that the situation is much worse if we try to drop this assumption.
\end{remark}

\begin{remark}\label{rem regu}
Let us mention that this counterexample is not linked to the regularity of the data.
Indeed, we can always construct smooth functions $\lambda_2$ and $a$ such that \eqref{hyp l2} is satisfied.
We can take for instance
$$\lambda_2(x)=-e^{-C_1 \eta(x)}, \qquad a(x)=C_2\eta(1-x),$$
where $\eta \in C^{\infty}(\R)$ is
$$
\eta(x)=
\left\{\begin{array}{cl}
0 & \mbox{ if } x \leq \frac{1}{2}, \\
e^{\frac{1}{\frac{1}{2}-x}} & \mbox{ if } x>\frac{1}{2},
\end{array}\right. 
$$
and $C_1,C_2>0$ are suitable constants to ensure that $-\int_{\frac{1}{2}}^1 \frac{1}{\lambda_2(\xi)} \, d\xi=\frac{3}{2}$ and $\int_0^{\frac{1}{2}} a(x) \, dx=\frac{\pi}{2}$.
\end{remark}

\begin{proof}[Proof of Proposition \ref{prop counterex}] 
~
\begin{enumerate}[1)]
\item
The sufficiency is known since $4=T_p(\Lambda)+T_{p+1}(\Lambda)$.
As already mentioned in the introduction, this was proved for instance in \cite[Theorem 3.2]{Rus78} (with a slightly different boundary condition at $x=1$), see also \cite[Theorem 3.2]{Li10}.

\item
Let us now show that the system \eqref{4 dimen sys}-\eqref{4 dimen sys BC} is not even approximately null controllable in time $T$ if $T<4$.
Since such a property is true in time $T_2$ if it is true in time $T_1 \leq T_2$, it is sufficient to prove it when
$$\frac{5}{2} \leq T< 4.$$
Let $y^0 \in L^2(0,1)^4$ be any initial data with its third component being
\begin{equation}\label{CE idata}
y^0_3(x)
=
\left\{\begin{array}{ll}
\ds 1 &\ds  \mbox{ if } x \in \left(\frac{T-2}{2},1\right), \\
\ds 0 &\ds  \mbox{ otherwise. }
\end{array}\right. 
\end{equation}
Note that it is well-defined since $T<4$.
We argue by contradiction and assume that, for every $\epsilon>0$, there exist controls $u_1,u_2 \in L^2(0,+\infty)$ such that the corresponding solution $y\in C^0([0,+\infty); L^2(0,1)^4)$ to the initial-boundary value problem \eqref{4 dimen sys}-\eqref{4 dimen sys BC}-\eqref{CE idata} satisfies
\begin{equation}\label{zerofinaldata}
\norm{y(T)}_{L^2(0,1)^4} \leq \epsilon.
\end{equation}

\item
Let us show how we obtain a contradiction.
We refer to Figure \ref{figureLong} to clarify the geometric situation.
Since the equation satisfied by $y_3$ is not coupled with the other ones, using the method of characteristics and the fact that $y^0_3=1$ in $(\frac{T-2}{2},1)$, we see that
$$y_3(t,0)=1, \quad \mbox{ a.e. } t \in \left(T-2,2\right).$$
The boundary condition $y_1(t,0)=y_3(t,0)$ then immediately yields
\begin{equation}\label{eq yuz}
y_1(t,0)=1, \quad \mbox{ a.e. } t \in \left(T-2,2\right).
\end{equation}

On the other hand, since the equations of $(y_1,y_2)$ are not coupled with the equations of $(y_3,y_4)$, and since $\lambda_2(x)=-1$ for every $x\in [0,\frac{1}{2}]$ by construction \eqref{hyp l2}, the method of characteristics shows that the solution $(y_1,y_2)$ of the corresponding sub-system satisfies
\begin{equation}\label{sol subsyst}
\begin{pmatrix}
y_1(t,x)\\
y_2(t,x)
\end{pmatrix}
=e^{\left(\int_0^x a(\xi)d\xi\right) J}
\begin{pmatrix}
y_1(t-x,0)\\
y_2(t-x,0)
\end{pmatrix},
\end{equation}
for every $x \in [0,\frac{1}{2}]$ and a.e. $t \in (0,T)$ such that $t-x \in (0,T)$, where
$$J=\begin{pmatrix} 0 & 1 \\ -1 & 0 \end{pmatrix}.$$

Since $\int_0^{\frac{1}{2}} a(\xi)d\xi=\frac{\pi}{2}$ by construction \eqref{hyp l2}, we have
$$e^{\left(\int_0^{\frac{1}{2}} a(\xi)d\xi\right) J}
=e^{\frac{\pi}{2}J}=J.$$
Therefore, taking $x=\frac{1}{2}$ in \eqref{sol subsyst}, we obtain
$$
\begin{pmatrix}
y_1\left(t,\frac{1}{2}\right)\\
y_2\left(t,\frac{1}{2}\right)
\end{pmatrix}=
J
\begin{pmatrix}
y_1\left(t-\frac{1}{2},0\right) \\
y_2\left(t-\frac{1}{2},0\right)
\end{pmatrix}
=
\begin{pmatrix}
y_2\left(t-\frac{1}{2},0\right) \\
-y_1\left(t-\frac{1}{2},0\right)
\end{pmatrix}
,
$$
for a.e. $t \in (0,T)$ such that $t-\frac{1}{2} \in (0,T)$.
In particular,
$$
y_2\left(t,\frac{1}{2}\right)=-y_1\left(t-\frac{1}{2},0\right), \quad \mbox{ a.e. } t \in \left(\frac{1}{2},T\right),
$$
and it follows from \eqref{eq yuz} that (note that $(T-2,2) \subset (\frac{1}{2},T)$ since $T \geq \frac{5}{2}$)
\begin{equation}\label{yd mu}
y_2\left(t,\frac{1}{2}\right)=-1, \quad \mbox{ a.e. } t \in \left(T-\frac{3}{2},\frac{5}{2}\right).
\end{equation}

On the other hand, since $a=0$ in $(\frac{1}{2},1)$ by construction \eqref{hyp l2}, the equation $y_2$ is not coupled with the equation of $y_1$ in $(0,T) \times (\frac{1}{2},1)$.
As a result, we can use the method of characteristics and obtain (recall the definition \eqref{caract} of $\chi_2$)
$$y_2\left(T,\caract_2\left(T;t,\frac{1}{2}\right)\right)=y_2\left(t,\frac{1}{2}\right), \quad \mbox{ a.e. } t \in \left(T-\frac{3}{2},T\right).$$
Using \eqref{yd mu}, we deduce that
$$y_2\left(T,x\right)=-1, \quad \mbox{ a.e. } x \in \omega,$$
where $\omega \subset (0,1)$ is the non empty open subset defined by $\omega=\ens{\caract_2(T;t,\frac{1}{2}) \, \middle| \, t \in (T-\frac{3}{2},\frac{5}{2})}$.
It then follows that \eqref{zerofinaldata} is not possible if $\epsilon<1$, a contradiction.

\end{enumerate}
\end{proof}

\begin{center}

\scalebox{0.75}{
\begin{tikzpicture}[scale=4]

\fill [gray!10]
	(0,0) rectangle (0.5,3.5);
      
\draw [<->] (0,4) node [left]  {$t$} -- (0,0) -- (1.5,0) node [below right] {$x$};
\draw (0,0) rectangle (1,3.5);
\draw (-0.1,3.5) node [left] {$T$};
\draw (0,0) node [left, below] {$0$};
\draw (1,0) node [below] {$1$};

\draw (0.5,0) -- (0.5,3.5);
\draw (0.5,0) node [below] {$\frac{1}{2}$};

\draw (0,2) node [left] {$2$};
\draw (0,1.5) node [left] {$T-2$};
\draw (0,1.5) node [left] {$T-2$};

\draw[dashed, domain=0:1] plot (\x, {2*(1-\x)});
\draw[dashed, domain=0:0.75] plot (\x, {2*(0.75-\x)});

\draw (0.75,0) node [below] {$\frac{T-2}{2}$};

\draw[ultra thick, domain=0:0.5] plot (\x, {2+\x});
\draw[ultra thick, domain=0:0.5] plot (\x, {1.5+\x});

\draw[decoration={brace, amplitude=6pt},decorate]
(0.05,3.55) -- node[above=4pt] {$\lambda_2=-1$} (0.45,3.55);
\draw[decoration={brace, amplitude=6pt},decorate]
(0.55,3.55) -- node[above=4pt] {$a=0$} (0.95,3.55);

\draw[decoration={brace, amplitude=3pt},decorate]
(0.75,0.05) -- (0.95,0.05);
\draw (0.85,0.05) node[above] {$y^0_3=1$};

\draw[decoration={brace, amplitude=6pt},decorate]
(-0.05,1.60) -- node[left=4pt] {$y_1(t,0)=y_3(t,0)=1$} (-0.05,1.90);

\draw[decoration={brace, mirror, amplitude=6pt},decorate]
(0.55,2.10) --  (0.55,2.45);

\draw[<-] (0.625,2.275)-- (1.1,2.275) node[right] {$y_2(t,\frac{1}{2})=-1$};

\draw[dashed] (0.5,2.5) -- (1,2.5) node[right] {$\frac{5}{2}$};
\draw[dashed] (0.5,2) -- (1,2) node[right] {$T-\frac{3}{2}$};

\coordinate (A) at (0.5,1.9995152);
\coordinate (B) at (0.6,2.0995226);
\coordinate (C) at (0.7,2.2034819);
\coordinate (D) at (0.8,2.350154);
\coordinate (E) at (0.9,2.664025);
\coordinate (F) at (1,3.5);
\draw[ultra thick] plot [smooth] coordinates { (A) (B) (C) (D) (E) (F)};

\coordinate (A') at (0.5,2.50001);
\coordinate (B') at (0.6,2.6000174);
\coordinate (C') at (0.7,2.703978);
\coordinate (D') at (0.8,2.8506586);
\coordinate (E') at (0.9,3.1645581);
\coordinate (F') at (0.953,3.5020381);
\draw[ultra thick] plot [smooth] coordinates { (A') (B') (C') (D') (E') (F')};

\draw[->] (1.2,3.7) node[right] {$\abs{y_2(T,x)}=1>\epsilon$} -- (0.98,3.525);

      \end{tikzpicture}
      }

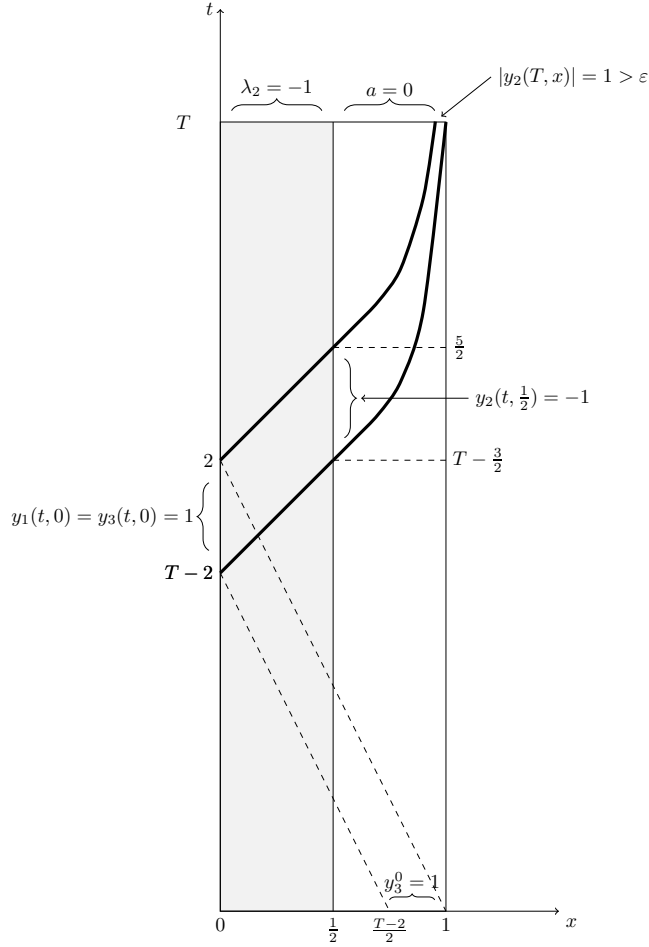
\captionof{figure}{Counterexample if \eqref{hyp Rus78 eg} fails}\label{figureLong}

\end{center}

\section{Sufficient conditions for the stability of the minimal time of control}\label{app impmap cpct}

In this appendix we prove Theorem \ref{cor DO18} and Lemma \ref{lem small times}, which provide practical sufficient conditions to ensure that the minimal time for exact controllability is invariant under bounded perturbations of the generator.
The proof is based on the compactness-uniqueness method and the Volterra integral equation satisfied by semigroups of boundedly perturbed generators.

Let us first briefly recall that the compactness-uniqueness method has been extensively used to prove the exact controllability of various systems governed by partial differential equations, see in particular \cite{Lio88} and the pioneering work \cite{RT74} concerning stability, and it has recently been improved and put in a complete abstract framework in \cite{DO18}.
We refer to the latter article and the numerous references therein for more details on this method.
We only wish to add the references \cite[Corollary 4.2]{DR77} and \cite[p. 657, p. 659]{Rus78} to those already present in \cite{DO18}.
The proof of Theorem \ref{cor DO18} is in fact a simple consequence of the following general abstract result, established in \cite[Theorem 4.1]{DO18}:

\begin{theorem}\label{thm DO18}
Let $H$ and $U$ be two complex Hilbert spaces.
Let $A:\dom{A} \subset H \longrightarrow H$ be the generator of a $C_0$-semigroup on $H$ and let $B \in \lin{U,\dom{A^*}'}$ be admissible for $A$.
Let $\Phi(T) \in \lin{L^2(0,+\infty;U),H}$ be the input map of $(A,B)$ at time $T \geq 0$.
Assume that there exist $T_0>0$, a complex Hilbert space $\widehat{H}$, a compact operator $G \in \lin{H,\widehat{H}}$ and $C>0$ such that, for every $z^1 \in H$,
\begin{equation}\label{obs ineq plus cpct}
\norm{z^1}_H^2 \leq C\left(\int_0^{T_0} \norm{\Phi(T_0)^*z^1(t)}_U^2 \, dt+\norm{G z^1}_{\widehat{H}}^2\right).
\end{equation}
Assume moreover that $(A,B)$ satisfies the Fattorini-Hautus test.
Then, $(A,B)$ is exactly controllable in time $T$ for every $T>T_0$.
\end{theorem}

Let us now give the proof of Theorem \ref{cor DO18}.
In what follows, we use the notation introduced in the statement of Theorem \ref{cor DO18}.

\begin{proof}[Proof of Theorem \ref{cor DO18}]
We first prove that $\Toptlib{A_2,B} \leq \Toptlib{A_1,B}$.
Let then $T_1>0$ be such that $(A_1,B)$ is exactly controllable in time $T_1$ and let us show that necessarily $\Toptlib{A_2,B} \leq T_1$.
By assumption and duality there exists $C>0$ such that, for every $z^1 \in H$,
$$
\norm{z^1}_{H}^2 \leq C\int_0^{T_1} \norm{\Phi_1(T_1)^*z^1(t)}_{U}^2 \, dt,
$$
so that,
$$
\norm{z^1}_{H}^2 \leq 2C\left(\int_0^{T_1} \norm{\Phi_2(T_1)^*z^1(t)}_{U}^2 \, dt
+
\int_0^{T_1} \norm{\left(\Phi_1(T_1)^*-\Phi_2(T_1)^*\right)z^1(t)}_{U}^2 \, dt\right).
$$
By assumption we know that the remainder $G=\Phi_1(T_1)^*-\Phi_2(T_1)^*$ is compact and that $(A_2,B)$ satisfies the Fattorini-Hautus test.
Therefore, we can apply Theorem \ref{thm DO18} and obtain that $(A_2,B)$ is exactly controllable in time $T_1+\epsilon$ for every $\epsilon>0$.
This shows that $\Toptlib{A_2,B} \leq T_1+\epsilon$ for every $\epsilon>0$.
Letting $\epsilon \to 0$ we obtain the claim.
The proof of the reversed inequality $\Toptlib{A_2,B} \geq \Toptlib{A_1,B}$ is exactly the same by simply changing the roles of $(A_2,B)$ and $(A_1,B)$.
\end{proof}

Let us now turn out to the proof of Lemma \ref{lem small times}.
First of all, we shall establish the following result:

\begin{proposition}\label{thm DO18++}
Under the framework of Theorem \ref{cor DO18} (we do not assume \ref{FHtest} and \ref{impmap cpct} here though), we assume that:
\begin{enumerate}[label={(\roman*)$''$}]
\setcounter{enumi}{1}
\item\label{estim BVolt}
For every $T>0$, there exist a Hilbert space $\widetilde{H}$, a compact operator $F \in \lin{H,\widetilde{H}}$ and $C>0$ such that
$$
\ds \int_0^T \norm{B^*V\tilde{z}(t)}_{U}^2 \, dt
+
\int_0^T \norm{V\tilde{z}(t)}_{H}^2 \, dt
\leq C \norm{F z^0}_{\widetilde{H}}^2, \quad \forall z^0 \in \dom{A_1^*},
$$
where $V\tilde{z}(t)=\int_0^t K(t,s)\tilde{z}(s) \, ds$ is the Volterra operator with kernel $K(t,s)=S_{A_1}(t-s)^*P^*$ and $\tilde{z}(t)=S_{A_1}(t)^*z^0$.
\end{enumerate}
Then, the assumption \ref{impmap cpct} of Theorem \ref{cor DO18} holds, i.e. $\Phi_1(T)^*-\Phi_2(T)^*$ is compact for every $T>0$.
\end{proposition}

\begin{remark}
As in Lemma \ref{lem small times}, the assumption \ref{estim BVolt} in Proposition \ref{thm DO18++} only concerns the semigroup of the unperturbed system $(A_1,B)$.
Thus, this result is also usable in practice.
It was for instance proved in \cite[p. 402]{DO18} that \ref{estim BVolt} is satisfied if $P$ is compact.
However, we emphasize that the perturbation is only assumed to be bounded in Proposition \ref{thm DO18++} (it is important because in \eqref{syst} the perturbation is not compact).
The condition \ref{estim BVolt} is an integrated version of \ref{hyp ii prime}.
It is more general but it has to be checked for any time $T$.
\end{remark}

The proof of Proposition \ref{thm DO18++} relies on some ideas of \cite{NRL86} and an estimate that can be found for instance in \cite{DO18}.
More precisely, it is based on the two following results:

\begin{lemma}
For every $f \in C^1([0,+\infty);H)$ and $t \geq 0$,
\begin{equation}\label{int in dom}
\int_0^t S_{A_1}(t-s)^*f(s) \,ds \in \dom{A_1^*}.
\end{equation}
Moreover, for every $T>0$, there exists $C>0$ such that, for every $f \in C^1([0,T];H)$,
\begin{equation}\label{key estim}
\int_0^T \norm{B^* \int_0^t S_{A_1}(t-s)^*f(s) \,ds}_{U}^2 \, dt
\leq C \norm{f}_{L^2(0,T;H)}^2.
\end{equation}
\end{lemma}

The estimate \eqref{key estim} is a consequence of the admissibility of $B$ for $A_1$.
For a proof we refer for instance to \cite[Appendix A]{DO18}.
The second result we shall need is the following:

\begin{lemma}\label{lem NRL86}
For every $T>0$, there exists $C>0$ such that, for every $z^0 \in H$,
\begin{equation}\label{estim NRL86}
\int_0^T \norm{S_{A_1}(t)^*z^0-S_{A_2}(t)^*z^0}_{H}^2 \, dt \leq C \int_0^T \norm{\int_0^t S_{A_1}(t-s)^*P^*S_{A_1}(s)^*z^0 \,ds}_{H}^2 \, dt.
\end{equation}
\end{lemma}

The proof of this second lemma is included at the end of the proof of \cite[Lemma 3]{NRL86} but let us briefly recall it for the sake of completeness:

\begin{proof}[Proof of Lemma \ref{lem NRL86}]
Let $V \in \lin{L^2(0,T;H)}$ be the bounded linear operator defined for every $y \in L^2(0,T;H)$ by
$$Vy(t)=\int_0^t K(t,s)y(s) \, ds, \quad t \in (0,T),$$
where the kernel is $K(t,s)=S_{A_1}(t-s)^*P^*$.
Since $K \in L^{\infty}((0,T)\times(0,T);\lin{H})$, the operator $V$ is well-defined and $\Id-V$ is invertible (see e.g. \cite[Theorem 2.5]{Hoc73}).
Therefore, its inverse is bounded by the closed graph theorem, meaning that there exists $C>0$ such that, for every $y \in L^2(0,T;H)$,
\begin{equation}\label{volt inv bd}
\norm{y}_{L^2(0,T;H)} \leq C \norm{(\Id-V)y}_{L^2(0,T;H)}.
\end{equation}
Let us now recall the integral equation satisfied by semigroups of boundedly perturbed operators (see e.g. \cite[Corollary III.1.7]{EN00}), valid for every $z^0 \in H$ and $t \geq 0$:
$$
S_{A_2}(t)^*z^0=
S_{A_1}(t)^*z^0
+\int_0^t S_{A_1}(t-s)^*P^*S_{A_2}(s)^*z^0 \,ds.
$$
Thus, we see that $y(t)=S_{A_1}(t)^*z^0-S_{A_2}(t)^*z^0$ is the solution to the following Volterra integral equation in $L^2(0,T;H)$:
\begin{equation}\label{volterra}
(\Id-V)y(t)=-\int_0^t S_{A_1}(t-s)^*P^*S_{A_1}(s)^*z^0 \,ds, \quad t \in (0,T),
\end{equation}
and the desired estimate \eqref{estim NRL86} then follows from \eqref{volt inv bd}.
\end{proof}

We are now ready to prove Proposition \ref{thm DO18++}.

\begin{proof}[Proof of Proposition \ref{thm DO18++}]
Let $T>0$ be fixed.
We will show that there exists $C>0$ such that, for every $z^0 \in H$,
\begin{equation}\label{estim impmaps}
\norm{(\Phi_1(T)^*-\Phi_2(T)^*) z^0}_{L^2(0,+\infty;U)} \leq C \norm{F z^0}_{\widetilde{H}}.
\end{equation}
Since $F$ is assumed to be compact, this will clearly implies that $\Phi_1(T)^*-\Phi_2(T)^*$ is compact as well.
First of all, note that we only have to prove \eqref{estim impmaps} for $z^0 \in \dom{A_1^*}$ since this set is dense in $H$ and all the operators involved in \eqref{estim impmaps} are actually continuous operators on $H$.
Besides, when $z^0 \in \dom{A_1^*}=\dom{A_2^*}$, we have the more explicit expression $(\Phi_1(T)^*-\Phi_2(T)^*) z^0(t)=B^*S_{A_1}(T-t)^*z^0-B^*S_{A_2}(T-t)^*z^0$ for a.e. $t \in (0,T)$.
The starting point to estimate this difference is again the Volterra integral equation \eqref{volterra}.
Using \eqref{int in dom} we see that each term in \eqref{volterra} actually belongs to $\dom{A_1^*}$ if $z^0 \in \dom{A_1^*}=\dom{A_2^*}$.
Therefore, we can apply $B^*$ to obtain the following identity:
\begin{multline*}
B^*S_{A_1}(t)^*z^0-B^*S_{A_2}(t)^*z^0
=
-B^*\int_0^t S_{A_1}(t-s)^*P^*S_{A_1}(s)^*z^0 \,ds
\\
+B^*\int_0^t S_{A_1}(t-s)^*P^*\left(S_{A_1}(s)^*z^0-S_{A_2}(s)^*z^0\right) \,ds.
\end{multline*}
Using now the estimate \eqref{key estim} and then \eqref{estim NRL86} on the second term of the right-hand side, we obtain
\begin{multline*}
\int_0^T \norm{B^*S_{A_1}(t)^*z^0-B^*S_{A_2}(t)^*z^0}_{U}^2 \, dt
\leq 
C\left(\int_0^T \norm{B^*\int_0^t S_{A_1}(t-s)^*P^*S_{A_1}(s)^*z^0 \,ds}_{U}^2 \, dt\right.
\\
+\left.\int_0^T \norm{\int_0^t S_{A_1}(t-s)^*P^*S_{A_1}(s)^*z^0 \,ds}_{H}^2 \, dt\right).
\end{multline*}
Using the assumption \ref{estim BVolt} this establishes \eqref{estim impmaps} for every $z^0 \in\dom{A_1^*}$.
\end{proof}

Let us now conclude this part of the appendix with the proof of Lemma \ref{lem small times}, which in fact provides sufficient conditions in small time to guarantee that the assumption \ref{estim BVolt} of Proposition \ref{thm DO18++} is satisfied.
The proof is essentially a use of the basic functional equation of semigroups.

\begin{proof}[Proof of Lemma \ref{lem small times}]
~
\begin{enumerate}[1)]
\item
By assumption, there exist $C>0$ and $\delta \in (0,\epsilon)$ such that, for every $z^0 \in \dom{A_1^*}$,
\begin{equation}\label{hyp mixed}
\left\{\begin{array}{c}
\ds \norm{V\tilde{z}({\delta})}_H \leq C \norm{G(\delta)z^0}_{\widehat{H}}, \\
\ds \int_0^{\delta} \norm{B^*V\tilde{z}(t)}^2_U \, dt
+\int_0^{\delta} \norm{V\tilde{z}(t)}^2_H \, dt \leq C \int_0^{\delta} \norm{G(t)z^0}_{\widehat{H}}^2 \, dt,
\end{array}\right.
\end{equation}
where, by abuse of notation, $G \in \mathscr{L}^2(0,\epsilon;\lin{H,\widehat{H}})$ in \eqref{hyp mixed} denotes in fact a representative of the equivalence class $G \in L^2(0,\epsilon;\lin{H,\widehat{H}})$ (so that $\norm{G(t)}_{\lin{H,\widehat{H}}}<+\infty$ for every $t \in (0,\epsilon)$, in particular for $t=\delta$) with $G(\delta)$ and $G(t)$ compact for a.e. $t \in (0,\delta)$.
Note in particular that the right-hand side in the second estimate define a compact operator from $H$ into $L^2(0,\delta;\widehat{H})$ by Lebesgue's dominated convergence theorem.
We will show that \eqref{hyp mixed} is enough to imply \ref{estim BVolt} of Proposition \ref{thm DO18++}.
In what follows, $C>0$ denotes a positive constant that may change from line to line but that remains independent of $z^0$.

\item
Let now $T>0$ be fixed.
Let $k \in \N$ be such that $k\delta \leq T \leq (k+1)\delta$.
We have
$$
\int_0^T \norm{B^*V\tilde{z}(t)}_U^2 \, dt
\leq
\sum_{j=0}^k 
\int_{j\delta}^{(j+1)\delta} \norm{B^*\int_0^t S_{A_1}(s)^*P^*S_{A_1}(t-s)^*z^0 \, ds}_U^2 \, dt.
$$
The change of variable $\tau=t-j\delta$ gives
$$
\int_0^T \norm{B^*V\tilde{z}(t)}_U^2 \, dt
\leq
\sum_{j=0}^k 
\int_0^{\delta} \norm{B^*\int_0^{\tau+j\delta} S_{A_1}(s)^*P^*S_{A_1}(\tau+j\delta-s)^*z^0 \, ds}_U^2 \, d\tau.
$$
Thus, breaking the integral into two parts, we have
\begin{multline*}
\int_0^T \norm{B^*V\tilde{z}(t)}_U^2 \, dt
\leq
\sum_{j=0}^k 
\left(2\int_0^{\delta} \norm{B^*\int_0^{\tau} S_{A_1}(s)^*P^*S_{A_1}(\tau-s)^*S_{A_1}(j\delta)^*z^0 \, ds}_U^2 \, d\tau\right.
\\
\left.+2\int_0^{\delta} \norm{B^*\int_{\tau}^{\tau+j\delta} S_{A_1}(s)^*P^*S_{A_1}(\tau+j\delta-s)^*z^0 \, ds}_U^2 \, d\tau\right).
\end{multline*}
The first integral is estimated thanks to the second inequality in \eqref{hyp mixed}:
$$
\int_0^{\delta} \norm{B^*\int_0^{\tau} S_{A_1}(s)^*P^*S_{A_1}(\tau-s)^*S_{A_1}(j\delta)^*z^0 \, ds}_U^2 \, d\tau
\leq
C \int_0^\delta \norm{G(t)S_{A_1}(j\delta)^*z^0}_{\widehat{H}}^2 \, dt.
$$
For the second integral, we perform the change of variable $\sigma=s-\tau$ and then use the admissibility of $B$ to obtain
\begin{multline*}
\int_0^{\delta} \norm{B^*\int_{\tau}^{\tau+j\delta} S_{A_1}(s)^*P^*S_{A_1}(\tau+j\delta-s)^*z^0 \, ds}_U^2 \, d\tau
\\
= \int_0^{\delta} \norm{B^* S_{A_1}(\tau)^*\int_0^{j\delta} S_{A_1}(\sigma)^*P^*S_{A_1}(j\delta-\sigma)^*z^0 \, d\sigma}_U^2 \, d\tau
\leq C\norm{V\tilde{z}(j\delta)}_H^2.
\end{multline*}
Combining both estimates, we have thus obtained
$$
\int_0^T \norm{B^*V\tilde{z}(t)}_U^2 \, dt
\leq C
\sum_{j=0}^k 
\left(\int_0^\delta \norm{G(t)S_{A_1}(j\delta)^*z^0}_{\widehat{H}}^2 \, dt
+\norm{V\tilde{z}(j\delta)}_H^2 \right).
$$
Note that all the previous computations are also valid for $B=\Id$ since we only used the second inequality in \eqref{hyp mixed} and the admissibility of $B$.
Therefore, we have
$$
\int_0^T \norm{B^*V\tilde{z}(t)}_U^2 \, dt
+\int_0^T \norm{V\tilde{z}(t)}_H^2 \, dt
\leq C
\sum_{j=0}^k 
\left(\int_0^\delta \norm{G(t)S_{A_1}(j\delta)^*z^0}_{\widehat{H}}^2 \, dt
+\norm{V\tilde{z}(j\delta)}_H^2 \right).
$$

\item
Let us now estimate $V\tilde{z}(j\delta)$.
We have
$$
V\tilde{z}(j\delta)
=\sum_{i=0}^{j-1} \int_{i \delta}^{(i+1)\delta} S_{A_1}(j\delta-s)^*P^*S_{A_1}(s)^*z^0 \, ds.
$$
Doing the change of variables $\sigma=s-i\delta$ we obtain
$$
V\tilde{z}(j\delta) =\sum_{i=0}^{j-1} S_{A_1}(j\delta-(i+1)\delta)^*\int_0^{\delta} S_{A_1}(\delta-\sigma)^*P^*S_{A_1}(\sigma)^*S_{A_1}(i\delta)^*z^0 \, d\sigma.
$$
Using now the first estimate in \eqref{hyp mixed}, it follows that
$$
\norm{V\tilde{z}(j\delta)}_H \leq C \sum_{i=0}^{j-1} \norm{G(\delta)S_{A_1}(i\delta)^*z^0}_{\widehat{H}}.
$$

\end{enumerate}

\end{proof}

\section{Removal of the coupling terms where the speeds agree}\label{app pert diag}

The goal of this appendix is to give a proof of Lemma \ref{prop hyp Rus78}.
It is essentially an appropriate change of variable.
First of all, it is convenient to introduce the following notion (see also \cite{Bru70}):

\begin{definition}
Let $M,\widetilde{M} \in L^{\infty}(0,1)^{n \times n}$.
We say that the systems $(A_{\widetilde{M}},B)$ and $(A_M,B)$ are equivalent, and we write
$$(A_{\widetilde{M}},B) \sim (A_M,B),$$
if there exist two invertible linear transformations $L \in \lin{L^2(0,1)^n}$ and $\Gamma \in \R^{m \times m}$ such that, for every $y^0 \in L^2(0,1)^n$ and $u \in L^2(0,+\infty)^m$, if $y \in C^0([0,+\infty);L^2(0,1)^n)$ denotes the solution to $(A_M,B)$ with initial data $y^0$ and control $u$, then $\widetilde{y}=Ly \in C^0([0,+\infty);L^2(0,1)^n)$ is the solution to $(A_{\widetilde{M}},B)$ with initial data $\widetilde{y}^0=Ly^0$ and control $\widetilde{u}=\Gamma u$.
\end{definition}

It is not difficult to check that $\sim$ is an equivalence relation and that, if $(A_{\widetilde{M}},B) \sim (A_M,B)$, then, for every $T>0$, the system $(A_{\widetilde{M}},B)$ is exactly controllable in time $T$ if, and only if, the system $(A_M,B)$ is exactly controllable in time $T$.

\begin{proof}[Proof of Lemma \ref{prop hyp Rus78}]
~
\begin{enumerate}[1)]
\item
The goal is to construct $\widetilde{M}$ such that \ref{mt zero} holds and $(A_{\widetilde{M}},B) \sim (A_M,B)$, so that \ref{equiv mt} will hold as well.
Thanks to \eqref{hyp speeds} and \eqref{hyp Rus78 eg}, we see that there exist $d \in \ens{1,\ldots,n}$, $n_1,\ldots, n_d \in \ens{1,\ldots,n}$ with $\sum_{k=1}^d n_k=n$ and $\lambda^1,\ldots, \lambda^d \in C^{0,1}([0,1])$ with
$$\lambda^1(x)<\cdots<\lambda^d(x), \quad \forall x \in [0,1],$$
such that, for every $x \in [0,1]$,
$$\Lambda(x)=\diag(\Lambda^1(x), \ldots,\Lambda^d(x)),$$
where
\begin{equation}\label{lambda block}
\Lambda^k(x)=\lambda^k(x)\Id_{\R^{n_k \times n_k}}.
\end{equation}


To establish the equivalence between two systems $(A_{\widetilde{M}},B)$ and $(A_M,B)$, we will use a transformation of the form
\begin{equation}\label{def yt}
\widetilde{y}(t,x)=\Psi(x)y(t,x),
\end{equation}
where $\Psi\in W^{1,\infty}(0,1)^{n\times n}$ is assumed to be block diagonal:
$$\Psi(x)=\diag (\Psi^1(x),\ldots,\Psi^d(x)),$$
where, for every $k \in \ens{1,\ldots,d}$, $\Psi^k\in W^{1,\infty}(0,1)^{n_k\times n_k}$ will be determined below.
First of all,  it is clear that the formula \eqref{def yt} is reversible if we impose that all the matrices $\Psi^1(x),\ldots, \Psi^d(x)$ are invertible for every $x \in [0,1]$, which also implies that $x \mapsto \Psi(x)^{-1} \in C^0([0,1])^{n \times n} \subset L^{\infty}(0,1)^{n \times n}$.
Let us now work formally to find what $\Psi^1(x),\ldots, \Psi^d(x)$ shall satisfy and what $\widetilde{M}$ is allowed to be. 
Let us first investigate the boundary conditions.
Let us denote by $d^+$ the index such that
$$\sum_{i=1}^{d^+}n_i=p.$$
At $x=1$, we see that we should have
$$
\widetilde{u}(t)=\widetilde{y}_-(t,1)=\left(\begin{array}{c}
\Psi^{d^++1}(1)y^{d^+ +1}(t,1)\\
\vdots\\
\Psi^{d}(1)y^{d}(t,1)
\end{array}\right)
=\Gamma u(t),
$$
with $\Gamma=\diag(\Psi^{d^+ +1}(1),\ldots, \Psi^d(1))$, and where $(y^{d^+ +1},\ldots,y^d)$ is a block notation to simply denote $y_-$.
On the other hand, at $x=0$, we see that if we impose the condition $\Psi(0)=\Id_{\R^{n\times n}}$, then
$$\widetilde{y}_+(t,0)-Q\widetilde{y}_-(t,0)=y_+(t,0)-Qy_-(t,0).$$
Let us finally look at the equations that $\Psi$ should satisfy.
Since $\Psi^k(x)$ and $\Lambda^k(x)$ commute for every $x \in [0,1]$ and $k \in \ens{1,\ldots,d}$ (see \eqref{lambda block}), so do $\Psi(x)$ and $\Lambda(x)$:
\begin{equation}\label{equ commutation}
\Psi(x)\Lambda(x)=\Lambda(x)\Psi(x), \quad \forall x \in [0,1].
\end{equation}
As a result, we have
\begin{multline*}
\pt{\widetilde{y}}(t,x)-\Lambda(x) \px{\widetilde{y}}(t,x)-\widetilde{M}(x)\widetilde{y}(t,x)
\\
=
\Psi(x)\left(\pt{y}(t,x)-\Lambda(x) \px{y}(t,x)-\Psi(x)^{-1}\left(\Lambda(x)\px{\Psi}(x)+\widetilde{M}(x)\Psi(x)\right)y(t,x)\right).
\end{multline*}
Thus, $\widetilde{y}$ is a solution to $(A_{\widetilde{M}},B)$ if $y$ is a solution to $(A_M,B)$ and $\widetilde{M}$ is defined by
\begin{equation}\label{def mt}
\widetilde{M}(x)=\left(\Psi(x)M(x)-\Lambda(x)\px{\Psi}(x)\right)\Psi(x)^{-1},
\quad \mbox{ a.e. } x \in (0,1).
\end{equation}
Note that $\widetilde{M} \in L^{\infty}(0,1)^{n \times n}$.
To summarize, we have $(A_{\widetilde{M}},B) \sim (A_M,B)$ with $\widetilde{M}$ given by \eqref{def mt} if there there exist matrices $\Psi^k \in W^{1,\infty}(0,1)^{n_k \times n_k}$ such that the following two properties hold for every $k\in\{1,\ldots,d\}$:
$$
\left\{\begin{array}{l}
\text{$\Psi^k(x)$ is invertible for every $x \in [0,1]$,} \\
\Psi^k(0)=\Id_{\R^{n_k\times n_k}}.
\end{array}\right.
$$

\item
Our previous discussion was only formal but everything can be established rigorously by coming back to the very definition of weak solution (see Definition \ref{def weak sol}) and using some density arguments.
More precisely, let $\widetilde{\varphi} \in C^1([0,T]\times[0,1])^n$ be fixed such that $\widetilde{\varphi}_+(\cdot,1)=0$ and $\widetilde{\varphi}_-(\cdot,0)=R^*\widetilde{\varphi}_+(\cdot,0)$.
Let $H(x)=\Psi(x)^*-(1-x)\Psi(0)^*-x\Psi(1)^*$.
Since $H \in H^1_0(0,1)^{n \times n}$, there exists a sequence $\theta^j \in C^{\infty}_c(0,1)^{n \times n}$ such that $\theta^j \rightarrow H$ in $H^1(0,1)^{n \times n}$ as $j \to +\infty$.
Let then $\varphi^j$ be defined by $\varphi^j(t,x)=\left(\theta^j(x)+(1-x)\Psi(0)^*+x\Psi(1)^*\right)\widetilde{\varphi}(t,x)$.
Clearly, $\varphi^j \in C^1([0,T]\times[0,1])^n$ with $\varphi^j_+(\cdot,1)=0$ and $\varphi^j_-(\cdot,0)=R^*\varphi^j_+(\cdot,0)$ (since in fact $\Psi(0)=\Id_{\R^{n \times n}}$).
Moreover, 
$$
\left\{\begin{array}{l}
\varphi^j(T,\cdot) \xrightarrow[j \to +\infty]{} \Psi^*\widetilde{\varphi}(T,\cdot)
\quad \mbox{ and } \quad
\varphi^j(0,\cdot) \xrightarrow[j \to +\infty]{} \Psi^*\widetilde{\varphi}(0,\cdot) \quad \mbox{ in } L^2(0,1)^n,  \\
\ds \varphi^j \xrightarrow[j \to +\infty]{} \Psi^*\widetilde{\varphi} \quad \mbox{ in } H^1((0,T) \times (0,1))^n, \\
\varphi^j_-(\cdot,1) \xrightarrow[j \to +\infty]{} \Gamma^*\widetilde{\varphi}_-(\cdot,1) \quad \mbox{ in } L^2(0,T)^m.
\end{array}\right.
$$
Plugging the test function $\varphi^j$ in \eqref{weak sol} and passing to the limit $j \to +\infty$, we obtain
\begin{multline*}
\int_0^1 y(T,x) \cdot \Psi(x)^*\widetilde{\varphi}(T,x) \, dx
-\int_0^1 y^0(x) \cdot \Psi(x)^*\widetilde{\varphi}(0,x) \, dx
\\
=\int_0^T \int_0^1 y(t,x) \cdot \left(\Psi(x)^*\pt{\widetilde{\varphi}}(t,x)
-\Lambda(x)\Psi(x)^*\px{\widetilde{\varphi}}(t,x)\right.
\\
+\left.\left(-\Lambda(x)\px{\Psi}(x)^*+\left(-\px{\Lambda}(x)+M(x)^*\right)\Psi(x)^*\right)\widetilde{\varphi}(t,x)\right)\, dx dt
\\
+\int_0^T u(t) \cdot \Lambda_{-}(1)\Gamma^*\widetilde{\varphi}_-(t,1) \, dt.
\end{multline*}
Using \eqref{equ commutation}, its differentiated version and the definition \eqref{def mt} of $\widetilde{M}$, we obtain
\begin{multline*}
\int_0^1 \widetilde{y}(T,x) \cdot \widetilde{\varphi}(T,x) \, dx
-\int_0^1 \widetilde{y}^0(x) \cdot \widetilde{\varphi}(0,x) \, dx
\\
=\int_0^T \int_0^1 \widetilde{y}(t,x) \cdot \left(\pt{\widetilde{\varphi}}(t,x)-\Lambda(x)\px{\widetilde{\varphi}}(t,x)+\left(-\px{\Lambda}(x)+\widetilde{M}(x)^*\right)\widetilde{\varphi}(t,x)\right)\, dx dt
\\
+\int_0^T \widetilde{u}(t) \cdot \Lambda_{-}(1)\widetilde{\varphi}_-(t,1) \, dt.
\end{multline*}
This show that $\widetilde{y}$ defined by \eqref{def yt} is indeed the weak solution of the system $(A_{\widetilde{M}},B)$.

\item
The final goal is now to design the matrices $\Psi^1,\ldots,\Psi^d$ such that the matrix $\widetilde{M}$ given by \eqref{def mt} satisfies the condition \ref{mt zero} of Lemma \ref{prop hyp Rus78}, namely:
\begin{equation}\label{mt zero bis}
\widetilde{M}^k(x)=\px{\Lambda^k}(x), \quad \mbox{ a.e. } x \in (0,1), \quad \forall k \in \ens{1,\ldots,d},
\end{equation}
where $\widetilde{M}^k \in L^{\infty}(0,1)^{n_k \times n_k}$ denotes the submatrix $(\widetilde{m}_{i,j})_{\sum_{\ell=1}^{k-1} n_\ell+1 \leq i,j \leq \sum_{\ell=1}^k n_\ell}$.
To this end, for every $k \in \ens{1,\ldots,d}$, we take $\Psi^k\in W^{1,\infty}(0,1)^{n_k\times n_k}$ to be the solution to the O.D.E.
$$
\left\{\begin{array}{l}
\ds
\px{\Psi^k}(x)
=\Psi^k(x)\Lambda^k(x)^{-1}M^k(x)-\Lambda^k(x)^{-1}\px{\Lambda^k}(x)\Psi^k(x)
, \quad x \in (0,1), \\
\Psi^k(0)
=\Id_{\R^{n_k \times n_k}}.
\end{array}\right.
$$
Since $\Psi^k(x)$ commute with $\Lambda^k(x)^{-1}=\frac{1}{\lambda^k(x)}\Id_{\R^{n_k \times n_k}}$, we see that this implies that $\widetilde{M}$ given by \eqref{def mt} satisfies \eqref{mt zero bis}.
Moreover, it is clear that $\Psi^{k}(x)$ is invertible for every $x \in [0,1]$.
This completes the proof of Lemma \ref{prop hyp Rus78}.
\end{enumerate}
\end{proof}

\bibliographystyle{amsalpha}
\bibliography{biblio}

\end{document}